\documentclass[11pt]{article}
\usepackage{appendix,latexsym,amsfonts,amsmath,amssymb,graphicx,hyperref,amsthm,authblk}

\usepackage[T1]{fontenc}
\usepackage[utf8]{inputenc}
\usepackage{authblk}

\newcommand{\EE}{\mbox{\bf E}\,}

\newcommand{\R}{\mathbb{R}}
\newcommand{\C}{\mathbb{C}}
\newcommand{\Q}{\mathbb{Q}}
\newcommand{\HH}{\mathbb{H}}
\newcommand{\N}{\mathbb{N}}
\newcommand{\D}{\mathbb{D}}
\newcommand{\TT}{\mathbb{T}}
\newcommand{\Z}{\mathbb{Z}}
\newcommand{\Ree}{\mbox{Re}\,}
\newcommand{\Imm}{\mbox{Im}\,}

\newcommand{\pa}{\partial}

\newcommand{\F}{{\cal F}}
\newcommand{\no}{\noindent}

\newcommand{\BGE}{\begin{equation}}
\newcommand{\BGEN}{\begin{equation*}}
\newcommand{\EDE}{\end{equation}}
\newcommand{\EDEN}{\end{equation*}}

\def\eps{\varepsilon}

\def\til{\widetilde}
\def\ha{\widehat}
\def\sem{\setminus}
\def\lin{\overline}

\def\vphi{\varphi}

\def\conf{\stackrel{\rm Conf}{\twoheadrightarrow}}
\def\luto{\stackrel{\rm l.u.}{\longrightarrow}}
\def\dto{\stackrel{\rm Cara}{\longrightarrow}}

\DeclareMathOperator{\ccap}{cap} \DeclareMathOperator{\rad}{rad}
 \DeclareMathOperator{\diam}{diam}
\DeclareMathOperator{\dist}{dist} \DeclareMathOperator{\dcap}{dcap}
\DeclareMathOperator{\hcap}{hcap} \DeclareMathOperator{\id}{id}

 \DeclareMathOperator{\SLE}{SLE}

 \DeclareMathOperator{\JP}{JP}

\DeclareMathOperator{\mA}{m} 
 
 \DeclareMathOperator{\cc}{c}

\newtheorem{Theorem}{Theorem}[section]
\newtheorem{Lemma}[Theorem]{Lemma}
\newtheorem{Definition}[Theorem]{Definition}
\newtheorem{Corollary}[Theorem]{Corollary}
\newtheorem{Proposition}[Theorem]{Proposition}
\newtheorem{Question}[Theorem]{Question}
\numberwithin{equation}{section}

\begin{document}
\title{Backward SLE and the symmetry of the welding}
\author[1]{Steffen Rohde\thanks{Research partially supported by NSF grant DMS-1068105.}}
\author[2]{Dapeng Zhan\thanks{Research partially supported by NSF grants DMS-0963733, DMS-1056840, and Sloan fellowship.}}
\affil[1]{University of Washington}
\affil[2]{Michigan State University}
\renewcommand\Authands{ and }

\maketitle
\abstract {The backward chordal Schramm-Loewner Evolution naturally defines a conformal welding homeomorphism of the real line. We show that this homeomorphism is invariant under the automorphism $x\mapsto -1/x$, and conclude that the associated solution to the welding problem (which is  a natural renormalized limit of the finite time Loewner traces) is reversible. The proofs rely on an analysis of the action of analytic circle diffeomorphisms on the space of hulls, and on the coupling techniques of the second author.}
\tableofcontents
\section{Introduction}
\subsection{Introduction and results}
The Schramm-Loewner Evolution $\SLE_{\kappa}$, first introduced in \cite{S-SLE},  is a stochastic process of random conformal maps
that has received a lot of attention over the last decade. We refer to the introductory text \cite{LawSLE}
for basic facts and definitions. In this paper we are largely concerned with chordal $\SLE_{\kappa}$, which can be viewed as
a family of random curves $\gamma$ that join $0$ and $\infty$ in the closure of the upper half plane $\overline\HH.$ A fundamental property of  chordal
SLE is {\it reversibility}: The law of $\gamma$ is invariant under the automorphism $z\mapsto -1/z$ of $\HH,$ modulo time parametrization. This has first been proved by the second author in \cite{reversibility} for $\kappa\leq4$, and recently by Miller and Sheffield for $4<\kappa\leq8$ in \cite{M-SIII}. It is known to be false for $\kappa>8$ (\cite{RS-basic},\cite{duality}).

In the early years of SLE,  Oded Schramm, Wendelin Werner and the first author made an attempt to prove reversibility along the following lines: The ``backward'' flow
$$ \pa_t f_t(z)=\frac{-2}{f_t(z)-\sqrt{\kappa}B_t}, \quad f_0(z)=z,\quad 0\leq t\leq T,$$ 
generates curves $\beta_T=\beta[0,T]$ whose law is that of the chordal SLE trace $\gamma[0,T]$ (up to translation by $\sqrt{\kappa}B_T$) .
When $\kappa\leq4,$ these curves are simple, and
each   point of $\beta$ (with the exception of the endpoints) corresponds to two points on the real line under the conformal map $f_t$. The  {\it conformal welding homeomorphism} $\phi$
of $\beta_T$ is the auto-homeomorphism of the interval $f_T^{-1}(\beta_T)$ that interchanges these two points. In other words, it is the rule that describes which points on the real line are to be identified (laminated) in order to form the curve $\beta_T.$
It is known \cite{RS-basic} that, for $\kappa<4$, the welding almost surely uniquely determines the curve. The welding homeomorphism can be obtained by restricting the backward flow to the real line: Two points $x\neq y$ on the real line are to be welded if and only if their swallowing times coincide, $\phi(x)=y$ if and only if $\tau_x=\tau_y$, see Section \ref{welding}.
An idea to prove reversibility was to prove the invariance of $\phi$ under $x\mapsto-1/x$, and to relate this to reversibility of a suitable limit of the curves $\beta_T$. But the attempts to prove invariance of $\phi$ failed, and this program was  never completed successfully.

In this paper, we use the coupling techniques of the second author, introduced in \cite{reversibility}
for his  proof of reversibility of (forward) SLE traces. We use it to prove the  invariance of the welding:

\begin{Theorem}\label{rev}
  Let $\kappa\in(0,4]$, and $\phi$ be a backward chordal SLE$_\kappa$ welding. Let $h(z)=-1/z$. Then $h\circ \phi\circ h$ has the same distribution as $\phi$.
\end{Theorem}

As a  consequence, in the range $\kappa\in(0,4)$ where the SLE trace is conformally removable,
we obtain the reversibility of suitably normalized limits of the $\beta_T$ (see Section \ref{section-reversibility} for details):

\begin{Theorem}
  Let $\kappa\in(0,4)$, and $\beta$ be a normalized global backward chordal SLE$_\kappa$ trace. Let $h(z)=-1/z$. Then $h(\beta\sem\{0\})$ has the same distribution as $\beta\sem\{0\}$ as random sets.
\end{Theorem}

In the important paper \cite{SS-GFF}, Sheffield obtains a representation of the SLE welding in terms of a
quantum gravity boundary length measure, and also relates it to a simple Jordan arc, which differs from our $\beta$ only through normalization. A similar random welding homeomorphism is constructed in \cite{AJKS}, where the main point is the very difficult existence of a curve solving the welding problem. Our approach to the welding is  different:
In order to prove Theorem \ref{rev}, in Section \ref{Section-Conformal}  we develop a framework to study the effect of analytic perturbations of weldings on the corresponding  hulls.
We show in Section \ref{Conformal}  that a M\"obius image of a backward chordal SLE$_\kappa$ process is a
backward radial SLE($\kappa, -\kappa-6$) process, and the welding is preserved under this conformal transformation. In Section \ref{comm} we apply the coupling technique to show that backward radial SLE($\kappa, -\kappa-6$) started from  an ordered pair of points $(a,b)$ commutes with backward radial SLE($\kappa, -\kappa-6$) started from $(b,a)$, and use this in Section \ref{section-reversibility} to prove Theorem \ref{rev}.

In a subsequent paper \cite{tip} of the second author, Theorem \ref{rev} is used to study the ergodic properties of a forward SLE$_\kappa$ trace near the tip at a fixed capacity time.

\subsection{Notation}
Let $\ha\C=\C\cup\{\infty\}$, $\D=\{z\in\C:|z|<1\}$,  $\D^*=\ha\C\sem\lin\D$,
$\TT=\{z\in\C:|z|=1\}$, and $\HH=\{z\in\C:\Imm z>0\}$. Let $I_\R(z)=\lin z$ and $I_\TT(z)=1/\lin z$ be the reflections about $\R$ and $\TT$, respectively. Let $e^i$ denote the map $z\mapsto e^{iz}$. Let $\cot_2(z)=\cot(z/2)$ and $\sin_2(z)=\sin(z/2)$. For a real interval $J$, let $C(J)$ denote the space of real valued continuous functions on $J$. An increasing or decreasing function in this paper is assumed to be strictly monotonic. We use $B(t)$ to denote a standard real Brownian motion. By $f:D\conf E$ we mean that $f$ maps $D$ conformally onto $E$. By $f_n\luto f$ in $U$ we mean that $f_n$ converges to $f$ uniformly on every compact subset of $U$. We will frequently use the notation $D_n\dto D$ as in Definition \ref{def-lim}.

The outline of this paper is the following. In Section \ref{Section-Conformal}, we derive some fundamental results in Complex Analysis, which are interesting on their own. In Section \ref{section-Loewner}, we review the properties of forward Loewner processes, and derive some properties of backward Loewner processes. In Section \ref{Conformal}, we discuss how are backward Loewner processes transformed by conformal maps. In Section \ref{comm}, we present and prove certain commutation relations between backward SLE$(\kappa;\vec\rho)$ processes. In the last section, we prove the reversibility of backward chordal SLE$_\kappa$ processes for $\kappa\in(0,4]$ and propose questions in other cases. In the appendix, we discuss some results on the topology of domains and hulls.

\section{Extension of Conformal Maps} \label{Section-Conformal}
\subsection{Interior hulls in $\C$}
An interior hull (in $\C$) is a nonempty compact connected set $K\subset\C$ such that $\C\sem K$ is also connected. For every interior hull $K$ in $\C$, there are a unique $r\ge 0$ and a unique $\phi_K:\ha\C\sem K\conf \ha\C\sem r\lin{\D}$ such that $\phi_K(\infty)=\infty$ and $\phi_K'(\infty):=\lim_{z\to\infty}z/\phi_K(z)=1$. We call $\rad(K):=r$ the radius of $K$ and $\ccap(K):=\ln(r)$ the capacity of $K$. The radius is $0$ iff $K$ contains only one point. In general, we have $\rad(K)\le \diam(K)\le 4\rad(K)$. We call $K$ nondegenerate if it contains more than one point. For such $K$, there is a unique $\vphi_K:\ha\C\sem K\conf \D^*$ such that $\vphi_K(\infty)=\infty$ and $\vphi_K'(\infty)>0$. In fact, $\vphi_K=\phi_K/\rad(K)$. Let $\psi_K=\vphi_K^{-1}$ for such $K$.

For any Jordan curve $J$ in $\C$, let $D_J$ denote the Jordan domain bounded by $J$, and let $D_J^*=\ha\C\sem(D_J\cup J)$. Suppose $f_J:\D\conf D_J$ and $f_J^*=\psi_{\lin{D_J}}:\D^*\conf D_J^*$. Then both $f_J$ and $f_J^*$ extend continuously to a homeomorphism from $\TT$ onto $J$. Let $h=(f_J^*)^{-1}\circ f_J$. Then $h$ is an orientation-preserving automorphism of $\TT$. We call such $h$ a conformal welding. Not every homeomorphism of $\TT$ is a conformal welding, but it is well-known (and an easy consequence of the uniformization theorem) that every analytic automorphism is a conformal welding, and that the associated Jordan curve is analytic. See \cite{LV} for the quasiconformal theory of conformal welding, and \cite{B} for deep generalizations and further references.

\begin{Lemma}
  Let $\beta$ be an analytic Jordan curve. Let $\Omega\subset\C$ be a neighborhood of $\TT$. Suppose $W$ is a conformal map defined in $\Omega$, maps $\TT$ onto $\TT$, and preserves the orientation of $\TT$. Let $\Omega^\beta=\beta\cup D_{\beta}\cup \psi_{\lin{D_\beta}}(\Omega\cap\D^*)$. Then there is a conformal map $V$ defined in $\Omega^\beta$ such that $V\circ \psi_{\lin{D_\beta}}=\psi_{\lin{D_{V(\beta)}}}\circ W$ in $\Omega\cap\D^*$.\label{quasi}
\end{Lemma}

\begin{proof} Fix a conformal map $f_{\beta}:\D\conf D_{\beta}$ and let
$h_{\beta}=\vphi_{\lin{D_\beta}}\circ f_{\beta}$ be the associated conformal welding homeomorphism. Define $h=W\circ h_{\beta}.$ Since
$\beta$ is analytic, $h$ is analytic and there is an analytic Jordan curve $\gamma$ and a conformal map
$f_{\gamma}:\D\conf D_{\gamma}$ such that $h=h_{\gamma}=\vphi_{\lin{D_\gamma}}\circ f_{\gamma}$.
Define $V=f_{\gamma}\circ f_{\beta}^{-1}$ on $D_{\beta}.$ Since ${\beta}$ and $\gamma$ are analytic curves, $V$ extends conformally to a neighborhood of $\beta$ with $V(\beta)=\gamma$. On $\beta$, this extension (still denoted $V$) satisfies
$V=(\psi_{\lin{D_\gamma}}\circ h_\gamma)\circ(h_{\beta}^{-1}\circ\psi_{\lin{D_\beta}}^{-1}) = \psi_{\lin{D_\gamma}}\circ W \circ\psi_{\lin{D_\beta}}^{-1}$.
Therefore $V$ extends conformally to all of $\Omega^\beta$ and satisfies the desired property.
\end{proof}

\begin{figure}[htp]
\centering
\includegraphics[scale=0.70]{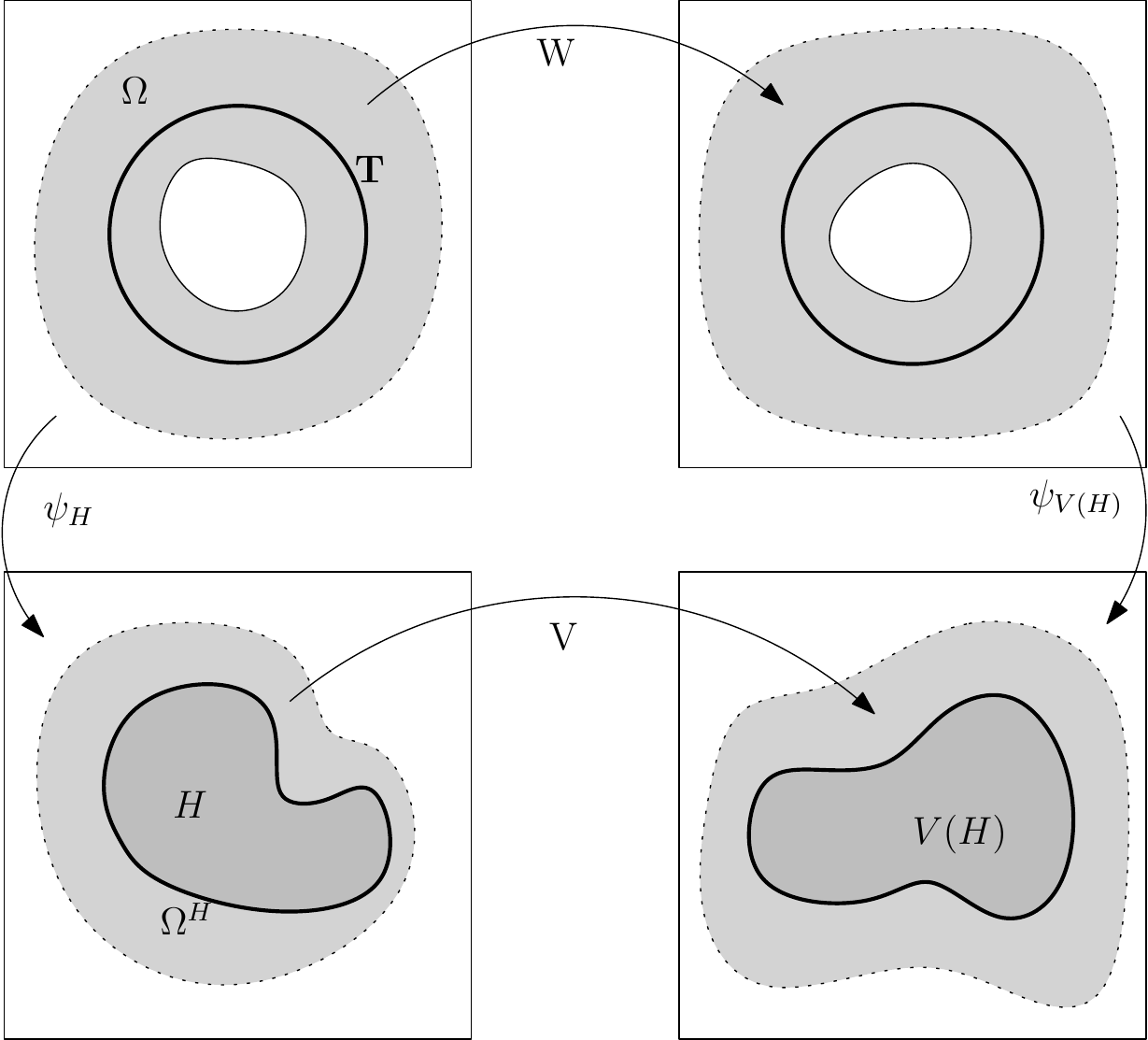}
\caption{The situation of Theorem \ref{extend1} and Lemma \ref{quasi}. Given $H$ and $W$, $V$ can be constructed to be analytic in $H$. In Lemma \ref{quasi}, the boundary of $H$ is assumed to be an analytic Jordan curve, while in Theorem \ref{extend1}, no regularity assumption is made.}\label{fig:annuli}
\end{figure}

\begin{Theorem}
  Let $H$ be a nondegenerate interior hull. Let $\Omega\subset\C$ be a neighborhood of $\TT$. Suppose $W$ is a conformal map defined in $\Omega$, maps $\TT$ onto $\TT$, and preserves the orientation of $\TT$. Let $\Omega^H=H\cup \psi_H(\Omega\cap\D^*)$. Then there is a conformal map $V$ defined in $\Omega^H$ such that $V\circ \psi_H=\psi_{V(H)}\circ W$ in $\Omega\cap\D^*$. 
  If another conformal map $\til V$ satisfies the properties of $V$, then $\til V=aV+b$ for some $a>0$ and $b\in\C$. \label{extend1}
\end{Theorem}

\begin{proof} First, define a sequence of analytic Jordan curves $(\beta_n)$ by
$$\beta_n=\psi_H(\{e^{\frac 1n +i\theta}:0\le \theta\le 2\pi\}),\quad n\in\N.$$
Then $\beta_n\cup D_{\beta_n}\to H$ in $d_{\cal H}$ (see Appendix \ref{B}).
From Lemma \ref{quasi}, for each $n\in\N$, there is a conformal map $V_n$ defined in $\Omega^{\beta_n}:=\beta_n\cup D_{\beta_n}\cup \psi_{\beta_n}(\Omega\cap\D^*)$ such that $V_n\circ \psi_{\beta_n}=\psi_{V_n(\beta_n)}\circ W$ in $\Omega\cap\D^*$. Note that for any $a_n>0$ and $b_n\in\C$, $a_nV_n+b_n$ satisfies the same property as $V_n$. Thus, we may assume that $0\in V_n(\beta_n)\subset\lin\D$ and $V_n(\beta_n)\cap\TT\ne\emptyset$. Let $\gamma_n=V_n(\beta_n)$, $n\in\N$. Then each $\gamma_n$ is an interior hull contained in the interior hull $\lin\D$, and $\diam(\gamma_n)\ge 1$. So $\rad(\gamma_n)\ge 1/4$. From Corollary \ref{compact-cor}, $(\gamma_n)$ contains a subsequence which converges to some interior hull $K$ contained in $\lin\D$ with radius at least $1/4$. So $K$ is nondegenerate. By passing to a subsequence, we may assume that $\gamma_n\to K$. From $\beta_n\to H$ and $\gamma_n\to K$ we get $\psi_{\beta_n}\luto\psi_H$ in $\Omega\cap \D^*$ and $\psi_{\gamma_n}\luto\psi_K$ in $W(\Omega\cap\D^*)$. Thus, $\psi_{\beta_n}(\Omega\cap\D^*)\dto \psi_H(\Omega\cap\D^*)$ by Lemma \ref{domain convergence*}.

Since $V_n\circ \psi_{\beta_n}=\psi_{\gamma_n}\circ W$ in $\Omega\cap\D^*$, we find that $V_n=\psi_{\gamma_n}\circ W\circ \psi_{\beta_n}^{-1}$ in $\psi_{\beta_n}(\Omega\cap\D^*)$. Let $V=\psi_K\circ W\circ \psi_H^{-1}$ in $\psi_H(\Omega\cap\D^*)$. Then $V_n\luto V$ in $\psi_H(\Omega\cap\D^*)$. We may find $r>1$ such that for any $s\in(1,r]$, $s\TT\subset\Omega\cap\D^*$. Then $\psi_{H}(r\TT)$ is a Jordan curve in $\psi_H(\Omega\cap\D^*)$ surrounding $H$, and the Jordan domain bounded by $\psi_{H}(r\TT)$ is contained in $\Omega^H=H\cup \psi_H(\Omega\cap\D^*)$. Since $\psi_{H}(r\TT)$ is a compact subset of  $\psi_H(\Omega\cap\D^*)$, we have $V_n\to V$ uniformly on $\psi_{H}(r\TT)$. It is easy to see that $\Omega^{\beta_n}\dto\Omega^H$. For $n$ big enough, $\psi_{H}(r\TT)$ together with its interior is contained in $\Omega^{\beta_n}$. From the maximum principle, $V_n$ converges uniformly in the interior of $\psi_{H}(r\TT)$ to a conformal map which extends $V$. We still use $V$ to denote the extended conformal map. Then $V$ is a conformal map defined in $\Omega^H$, and $V_n\luto V$ in $\Omega^H$. Letting $n\to\infty$ in the equality $V_n\circ \psi_{\beta_n}=\psi_{\gamma_n}\circ W$ in $\Omega\cap\D^*$, we conclude that $V\circ \psi_H=\psi_{V(H)}\circ W$ in $\Omega\cap\D^*$. So the existence part is proved.

If $\til V=aV+b$ for some $a>0$ and $b\in\C$, then $\psi_{\til V(H)}=a\psi_{V(H)}+b$, which implies $\til V\circ \psi_H=\psi_{\til V(H)}\circ W$. Finally, suppose $\til V$ satisfies the properties of $V$. Then $\til V\circ V^{-1}$ is a conformal map in $V(\Omega^H)$. Since $V\circ \psi_H=\psi_{V(H)}\circ W$ and $\til V\circ \psi_H=\psi_{\til V(H)}\circ W$ in $\Omega\cap\D^*$, we find that $\til V\circ V^{-1}=\psi_{\til V(H)}\circ \psi_{V(H)}^{-1}$ in $\psi_{V(H)}(W(\Omega\cap\D^*))=V(\Omega^H)\sem V( H)$. Note that $\psi_{\til V(H)}\circ \psi_{V(H)}^{-1}$ is a conformal map defined in $\ha\C\sem V(H)$. Since $V(\Omega^H)\cup (\ha\C\sem V(H))=\ha\C$, we may define an analytic function $h$ in $\C$ such that $h=\til V\circ V^{-1}$ in  $V(\Omega^H)$ and $h=\psi_{\til V(H)}\circ \psi_{V(H)}^{-1}$ in $\C\sem V(H)$. From the properties of $\psi_{\til V(H)}$ and  $\psi_{V(H)}$, we have $h(\infty)=\infty$ and $h'(\infty)>0$. Thus, $h(z)=az+b$ for some $a>0$ and $b\in\C$, which implies that $\til V=aV+b$. \end{proof}

Now we obtain a new proof of the following well-known result about conformal welding.

\begin{Corollary}
 Let $W$ be conformal in a neighborhood of $\TT$, maps $\TT$ onto $\TT$, and preserves the orientation of $\TT$. If $h$ is a conformal welding, then $W\circ h$ and $h\circ W$ are also conformal weldings. \label{weld}
\end{Corollary}
 \begin{proof} Apply Theorem \ref{extend1} to $H=\lin{D_J}$, where $J$ is the Jordan curve for the conformal welding $h$. We find a conformal map $V$ defined in $\Omega^H=D_J\cup f_J^*(\Omega\cap\D^*)$ such that $V\circ f_J^*=\psi_{V(H)}\circ W$ in $\Omega\cap\D^*$. Let $J'=V(J)$. Then $J'$ is also a Jordan curve, $V(H)=\lin{D_{J'}}$, and $\psi_{V(H)}=f_{J'}^*$. Let $f_{J'}=V\circ f_J$. Then $f_{J'}:\D\conf D_J$. Thus,
$$W\circ h=W\circ (f_J^*)^{-1}\circ f_J= \psi_{H(H)}^{-1}\circ V\circ f_J=(f_{J'}^*)^{-1}\circ f_{J'},$$
which implies that $W\circ h$ is a conformal welding. As for $h\circ W$, note that $(h\circ W)^{-1}=W^{-1}\circ h^{-1}$ and that $h$ is a conformal welding if and only if $h^{-1}$ is a conformal welding.
\end{proof}

\subsection{Hulls in the upper half plane}
Let $\HH=\{z\in\C:\Imm z>0\}$. A subset $K$ of $\HH$ is called an $\HH$-hull if it is bounded, relatively closed in $\HH$, and $\HH\sem K$ is simply connected. For every $\HH$-hull $K$, there is are a unique $c\ge 0$ and a unique $g_K:\HH\sem K\conf \HH$ such that $g_K(z)=z+\frac{c}z+O(\frac 1{z^2})$ as $z\to\infty$. The number $c$ is called the $\HH$-capacity of $K$, and is denoted by $\hcap(K)$. Let $f_K=g_K^{-1}$. The empty set is an $\HH$-hull with $\hcap(\emptyset)=0$ and $g_\emptyset=f_\emptyset=\id_{\HH}$.

\begin{Definition} Let $K_1$ and $K_2$ be $\HH$-hulls.
\begin{enumerate}
  \item  If $K_1\subset K_2$, define $K_2/K_1=g_{K_1}(K_2\sem K_1)$. We call $K_2/K_1$ a quotient hull of $K_2$, and write $K_2/K_1\prec K_2$.
  \item The product $K_1\cdot K_2$ is defined to be $K_1\cup f_{K_1}(K_2)$.
\end{enumerate}\label{/}
\end{Definition}

The following facts are easy to check.
\begin{enumerate}
  \item $K_2/K_1$ and $K_1\cdot K_2$ in the definition are also $\HH$-hulls.
  \item For any two $\HH$-hulls $K_1$ and $K_2$, $K_1\subset K_1\cdot K_2$ and $K_2=(K_1\cdot K_2)/K_1\prec K_1\cdot K_2$. If $K_1\subset K_2$, then $K_1\cdot(K_2/K_1)=K_2$.
  \item The space of $\HH$-hulls with the product ``$\cdot$'' is a semigroup with identity element $\emptyset$, and ``$\prec$'' is a transitive relation of this space.
  \item  $f_{K_1\cdot K_2}=f_{K_1}\circ f_{K_2}$ in $\HH$; $g_{K_1\cdot K_2}=g_{K_2}\circ g_{K_1}$ in $\HH\sem(K_1\cdot K_2)$.
  \item $\hcap(K_1\cdot K_2)=\hcap(K_1)+\hcap(K_2)$.  If $K_1\subset K_2$ or $K_1\prec K_2$, then $\hcap(K_1)\le\hcap(K_2)$.
\end{enumerate}

From $f_{K_1\cdot K_2}=f_{K_1}\circ f_{K_2}$ in $\HH$ we can conclude that $f_{K_1}=f_{K_1\cdot K_2}\circ g_{K_2}$ in $\HH\sem K_2$. So $f_{K_1}$ is an analytic extension of $f_{K_1\cdot K_2}\circ g_{K_2}$, which means that $K_1$ is uniquely determined by $K_1\cdot K_2$ and $K_2$. So the following definition makes sense.

\begin{Definition} Let $K_1$ and $K_2$ be $\HH$-hulls such that $K_1\prec K_2$. We use $K_2:K_1$ to denote the unique $\HH$-hull $K\subset K_2$ such that $K_2/K=K_1$.\label{//}
\end{Definition}

For an $\HH$-hull $K$, the base of $K$ is the set $B_K=\lin{K}\cap \R$. Let the double of $K$ be defined by $\ha K=K\cup I_{\R}(K)\cup B_K$, where $I_\R(z):=\lin{z}$. Then $g_K$ extends to a conformal map (still denoted by $g_K$) in $\ha\C\sem\ha K$, which satisfies $g_K(\infty)=\infty$, $g_K'(\infty)=1$, and $g_K\circ I_\R=I_\R\circ g_K$. Moreover, $g_K(\ha\C\sem \ha K)=\ha\C\sem S_K$ for some compact $S_K\subset \R$, which is called the support of $K$. So $f_K$ extends to a conformal map from $\ha\C\sem S_K$ onto $\ha\C\sem \ha K$. 

\begin{Lemma}
  $f_K$ can not be extended analytically beyond $\ha\C\sem S_K$. \label{extend}
\end{Lemma}
\begin{proof}
  Suppose $f_K$ can be extended analytically near $x_0\in\R$, then the image of $f_K$ contains a neighborhood of $f_K(x_0)\in\R$. So $f_K(\HH)=\HH\sem K$ contains a neighborhood of $f_K(x_0)$ in $\HH$. This then implies that $f_K(x_0)\in\R\sem B_K$. Thus, there is $y_0\in\R\sem S_K$ such that $f_K(y_0)=f_K(x_0)$. Since $f_K$ is conformal in $\HH$, we must have $x_0=y_0\in\R\sem S_K$.
\end{proof}

\begin{Lemma}
  If $K_1=K_2/K_0\prec K_2$, then $S_{K_1}\subset S_{K_2}$, $f_{K_2}=f_{K_0}\circ f_{K_1}$ in $\ha\C\sem S_{K_2}$, and $g_{K_2}=g_{K_1}\circ g_{K_0}$ in $\ha\C\sem \ha K_2$. \label{SKprec}
\end{Lemma}
\begin{proof} Since $K_2=K_0\cdot K_1$, we have $f_{K_2}=f_{K_0}\circ f_{K_1}$ in $\HH$, which implies that $g_{K_0}\circ f_{K_2}=f_{K_1}$ in $\HH$. Since $f_{K_2}$ maps $\ha\C\sem S_{K_2}$ conformally onto $\ha\C\sem \ha K_2\subset \ha\C\sem \ha K_0$, and $g_{K_0}$ is analytic in $\ha\C\sem \ha K_2$, we see that $g_{K_0}\circ f_{K_2}$ is analytic in $\ha\C\sem S_{K_2}$. Since  $g_{K_0}\circ f_{K_2}=f_{K_1}$ in $\HH$, from Lemma \ref{extend} we have $S_{K_1}\subset S_{K_2}$, and $g_{K_0}\circ f_{K_2}=f_{K_1}$ in $\ha\C\sem S_{K_2}$. Composing $f_{K_0}$ to the left of both sides, we get $f_{K_2}=f_{K_0}\circ f_{K_1}$ in $\ha\C\sem S_{K_2}$. Taking inverse, we obtain the equality for $g_K$'s.
\end{proof}

\begin{Definition}
  $S\subset \ha\C$ is called $\R$-symmetric if $I_\R(S)=S$. An $\R$-symmetric map $W$ is a function defined in an $\R$-symmetric domain $\Omega$, which commutes with $I_\R$, and maps $\Omega\cap\HH$ into $\HH$. \label{R-symmetric}
\end{Definition}

\no{\bf Remarks.}
\begin{enumerate}
  \item For any $\HH$-hull $K$,  $g_K$ and $f_K$ are $\R$-symmetric conformal maps.
  \item Let $W$ be an $R$-symmetric conformal map defined in $\Omega$. If an $\HH$-hull $K$ satisfies $\ha K\subset \Omega$ and $\infty\not\in W(\ha K)$, then $W(K)$ is also an $\HH$-hull and $\ha{W(K)}=W(\ha K)$.
\end{enumerate}

\begin{Definition}
  Let $\Omega$ be an $\R$-symmetric domain and $K$ be an $\HH$-hull. If $\ha K\subset \Omega$, we write $\Omega_K$ or $(\Omega)_K$ for $S_K\cup g_K(\Omega\sem \ha K)$, and call it the collapse of $\Omega$ via $K$.  If $S_K\subset\Omega$, we write $\Omega^K$ or $(\Omega^K)$ for $\ha K\cup f_K(\Omega\sem S_K)$, and call it the lift of $\Omega$ via $K$. \label{OmegaK}
\end{Definition}

\no{\bf Remarks.}
\begin{enumerate}
  \item In the definition, $\Omega_K$ is an $\R$-symmetric domain containing $S_K$; $\Omega^K$ is an $\R$-symmetric domain containing $\ha K$.
  \item $(\Omega_K)^K=\Omega$ and $(\Omega^K)_K=\Omega$ if the lefthand sides are well defined.
  \item  $\Omega_{K_1\cdot K_2}=(\Omega_{K_1})_{K_2 }$ and  $\Omega^{K_1\cdot K_2}=(\Omega^{K_2 })^{K_1}$ if either sides are well defined.
\end{enumerate}

\begin{Definition}
  Let $W$ be an $\R$-symmetric conformal map with domain $\Omega$. Let $K$ be an $\HH$-hull such that $\ha K\subset\Omega$ and $\infty\not\in W(\ha K)$. We write $W_K$ or $(W)_K$ for the conformal extension of $g_{W(K)}\circ W\circ f_K$ to $\Omega_K$, and call it the collapse of $W$ via $K$. \label{WK}
\end{Definition}

\no{\bf Remarks.}
\begin{enumerate}
  \item Since $g_{W(K)}\circ W\circ f_K:\Omega_K\sem S_K\conf W(\Omega)\sem S_{W(K)}$,
  the existence of $W_K$ follows from the Schwarz reflection principle. $W_K$ is an $\R$-symmetric conformal map, and $W_K(S_K)=S_{W(K)}$.
  \item The $g_K$ and $f_K$ defined at the beginning of this section should not be understood as the collapse of $g$ and $f$ via $K$.
  \item $W_{K_1\cdot K_2}=(W_{K_1})_{K_2 }$ if either side is well defined.
  \item $V_{W(K)}\circ W_K=(V\circ W)_K$ if either side is well defined. In particular, $(W^{-1})_{W(K)}=(W_K)^{-1}$.
\end{enumerate}

Let $B_K^*$ and $S_K^*$ be the convex hulls of $B_K$ and $S_K$, respectively. Let $\ha K^*=\ha K\cup B_K^*$. Then $g_K:\ha\C\sem \ha K^*\conf \ha\C\sem S_K^*$. If $K\ne\emptyset$, then $S_K^*$ is a bounded closed interval, $\ha K^*$ is a nondegenerate interior hull, and $\psi_{\ha K^*}=f_K\circ \psi_{S_K^*}$.
If $S_K^*\subset \Omega$, then $\Omega^K=\ha K^*\cup f_K(\Omega\sem S_K^*)$.
The lemma below is a part of  Lemma 5.3 in \cite{LERW}, where $S_K^*$ was denoted by $[c_K,d_K]$.

\begin{Lemma}
   If $K_1\subset K_2$, then $S_{K_1}^*\subset S_{K_2}^*$. \label{lemma-LERW}
\end{Lemma}

\begin{Theorem}
Let $W$ be an $\R$-symmetric conformal map with domain $\Omega$. Let $K$ be an $\HH$-hull such that $S_K\subset\Omega$ and $\infty\not\in W(S_K)$. Then there is a unique $\R$-symmetric conformal map $V$ defined in $\Omega^K$ such that $V_K=W$. \label{W^K}
\end{Theorem}

\begin{figure}[htp]
\centering
\includegraphics[scale=0.60]{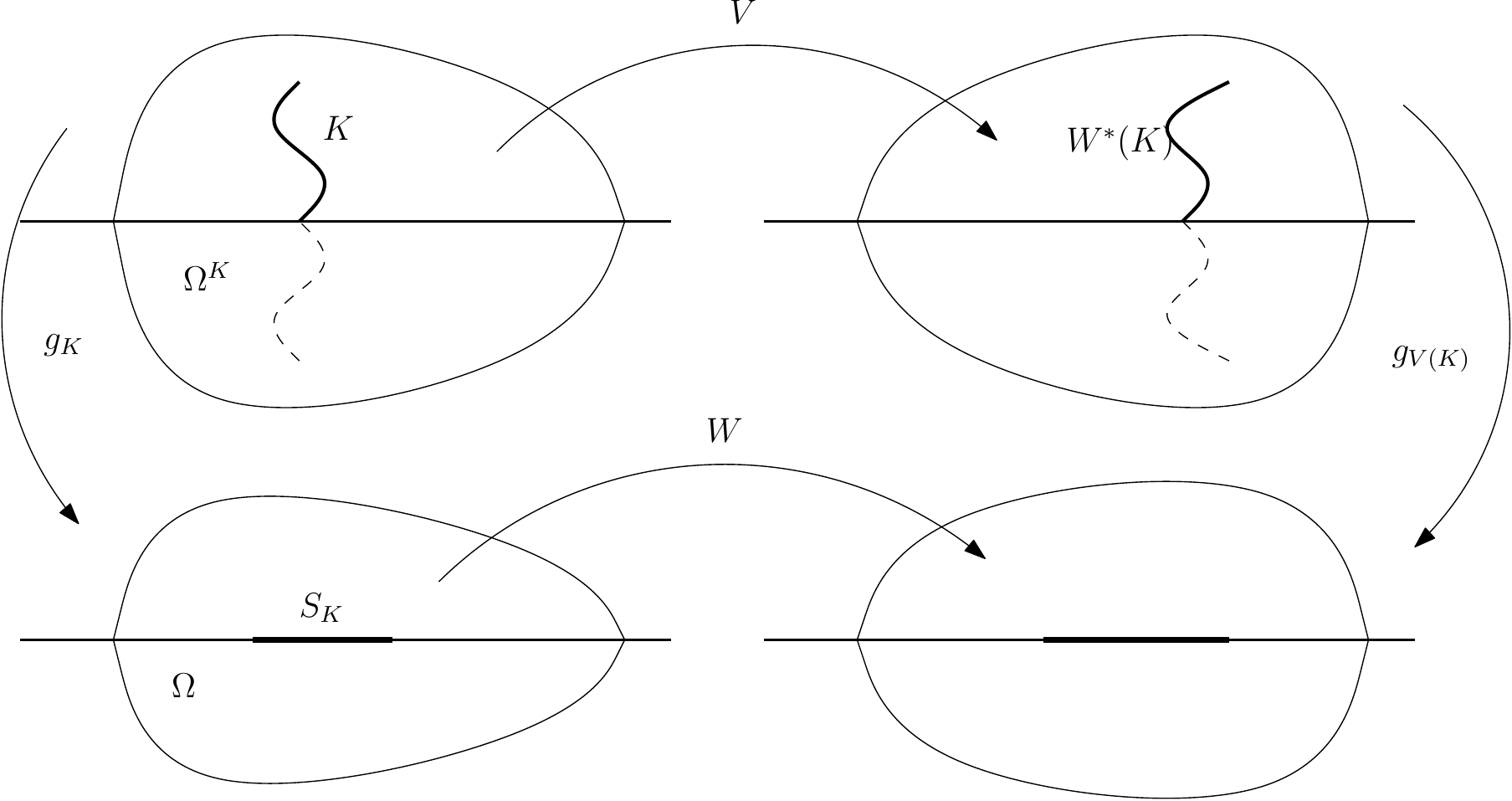}
\caption{The situation of Theorem \ref{W^K}. Given $K$ and $W$, there is a unique $V$, also denoted $W^K$,
which is analytic across $K$ and its reflection $I_{\R}(K)$,
see Definition
\ref{W^K-def}.}\label{fig:W*^K-fig}
\end{figure}

 \begin{proof} We first consider the existence. If $K=\emptyset$, since $f_\emptyset=\id$ and $\Omega^\emptyset=\Omega$, $V=W$ is what we need. Now suppose $K\ne \emptyset$ and $S_K^*\subset\Omega$. Note that $S_K^*$ is a bounded closed interval, and so is $W(S_K^*)$. Let $\Omega_{\TT}=\psi_{S_K^*}^{-1}(\Omega\sem S_K^*)$. Define a conformal map $W_{\TT}$ in  $\Omega_{\TT}$ by $W_{\TT}=\psi_{W(S_K^*)}^{-1}\circ W\circ \psi_{S_K^*}$. Then $W_\TT(z)\to \TT$ as $\Omega_\TT\ni z\to \TT$. Thus, $W_\TT$ extends conformally across $\TT$, maps $\TT$ onto $\TT$, and preserves the orientation of $\TT$. Apply Theorem \ref{extend1} to $W_\TT$ and $\ha K^*$. We find a conformal map $\ha V$ defined in
$$\ha K^*\cup \psi_{\ha K^*}(\Omega_\TT)=\ha K^*\cup f_K(\Omega\sem S_K^*) =\Omega^K$$
such that $\ha V\circ \psi_{\ha K^*}=\psi_{\ha V(\ha K^*)}\circ W_\TT$ in $\Omega_\TT$. Let $\til V=I_\R\circ \ha V\circ I_\R$. Then $\til V(\ha K^*)=I_\R\circ \ha V(\ha K^*)$. So $\psi_{\til V(\ha K^*)}=I_\R\circ \psi_{\ha V(\ha K^*)}\circ I_\R$. Since $I_\R$ commutes with $\psi_{\ha K^*}$ and $W_\TT$, we see that $\til V$ also satisfies the properties of $\ha V$. So $\til V=a\ha V+b$ for some $a>0$ and $b\in\C$. Thus, $I_\R\circ \ha V\circ I_\R=a\ha V+b$. Considering the values of $\ha V$ on $\Omega^K\cap\R$, we find that $a=1$ and $\Ree b=0$. Note that $\ha V-\frac b2$ satisfies the property of $\ha V$, and commutes with $I_\R$.
By replacing $\ha V$ with $\ha V-\frac b2$, we may assume that $\ha V$ is an $\R$-symmetric conformal map.

Since $\ha V\circ \psi_{\ha K^*}=\psi_{\ha V(\ha K^*)}\circ W_\TT$ in $\Omega_\TT$, from $\psi_{\ha K^*}=f_K\circ \psi_{S_K^*}$, $\psi_{\ha V(\ha K^*)}=f_{\ha V(K)}\circ \psi_{S_{\ha V(K)}^*}$, and the definitions of $W_\TT$ and $\Omega_\TT$, we have
\BGE \ha V\circ f_K=f_{\ha V(K)}\circ \psi_{S_{\ha V(K)}^*}\circ \psi_{W(S_K^*)}^{-1}\circ W\label{VK}\EDE
on $\Omega\sem S_K^*$. Let $h=\psi_{S_{\ha V(K)}^*}\circ \psi_{W(S_K^*)}^{-1}$. Since $S_{\ha V(K)}^*$ and $W(S_K^*)$ are both bounded closed intervals, we have $h(z)=az+b$ for some $a>0$ and $b\in\R$. Let $V=h^{-1}\circ \ha V$. Then $V$ is also  an $\R$-symmetric conformal map defined on $\Omega^K$, and $f_{V(K)}=h^{-1}\circ f_{\ha V(K)}\circ h$. From (\ref{VK}) we have
$$f_{V(K)}\circ W=h^{-1}\circ f_{\ha V(K)}\circ h\circ W=h^{-1}\circ\ha V\circ f_K=V\circ f_K.$$
This finishes the existence part in the case that $K\ne \emptyset$ and $S_K^*\subset\Omega$.

Now we still assume that $K\ne\emptyset$ but do not assume that $S_K^*\subset\Omega$. Let $\Omega_0=\Omega$ and $W_0=W$. We will construct $\HH$-hulls $K_1,\dots,K_n$ and $\R$-symmetric domains $\Omega_1,\dots,\Omega_n$ such that $K_n\cdot K_{n-1}\cdots K_1=K$, $\Omega_j=\Omega_{j-1}^{K_j}$, and $S_{K_j}^*\subset \Omega_{j-1}$, $1\le j\le n$. When they are constructed, using the above result, we can obtain $\R$-symmetric conformal maps $W_j$ defined on $\Omega_j$, $1\le j\le n$, such that $(W_j)_{K_j}=W_{j-1}$, $1\le j\le n$. Let $V=W_n$. Then $V$ is defined in $\Omega_n=\Omega_0^{K_n\cdots K_1}=\Omega^K$, and $V_K=(W_n)_{K_n\cdots K_1}=W_0=W$. So $V$ is what we need.

It remains to construct $K_j$ and $\Omega_j$ with the desired properties. Since $\Omega\cap\R$ is a disjoint union of  open intervals, and $S_K$ is a compact subset of $\Omega\cap \R$, we may find finitely many components of $\Omega\cap\R$ which cover $S_K$. There exist mutually disjoint $\R$-symmetric Jordan curves $J_1,\dots,J_n$ in $\Omega$ such that their interiors $D_{J_1},\dots,D_{J_n}$ are mutually disjoint and contained in $\Omega$, and $S_K\subset \bigcup_{k=1}^n D_{J_k}$. Then $J^K_j:=f_K(J_j)$, $1\le j\le n$ are $\R$-symmetric Jordan curves, which together with their interiors are mutually disjoint, and $\ha K\subset\bigcup_{k=1}^n D_{f_K(J_k)}$. Let $H_j=K\cap \bigcup_{k=j}^n D_{J^K_j}$, $1\le j\le n$. Then each $H_j$ is an $\HH$-hull, and $K=H_1\supset H_2\supset\cdots\supset H_n$. Let $K_j=H_j/H_{j+1}$, $1\le j\le n-1$, and $K_n=H_n$. Then we have $K_n\cdots K_1=H_1=K$.

Construct $\Omega_j$, $1\le j\le n$, such that $\Omega_j=\Omega_{j-1}^{K_j}$, $1\le j\le n$. Then
$$\Omega_{j-1}=(\Omega_0)^{K_{j-1}\cdots K_1}=(\Omega^{K_n\cdots K_1})_{K_n\cdots K_j}=(\Omega^K)_{H_j},\quad 1\le j\le n.$$
It suffices to show that $S_{K_j}^*\subset \Omega_{j-1}$. We have
$$K_j=H_j/H_{j+1}=g_{H_{j+1}}(H_j\sem H_{j+1})=g_{H_{j+1}}(K\cap D_{J^K_j}).$$
Thus, $\ha K_j\subset D_{g_{H_{j+1}}(J^K_j)}$, which implies that $S_{K_j}\subset D_{g_{H_{j}}(J^K_j)}$. Since $\R\cap  D_{g_{H_{j}}(J^K_j)}$ is an interval, we have $S_{K_j}^*\subset D_{g_{H_{j}}(J^K_j)}$. Since $\overline{D_{J^K_j}}\subset \Omega^K$, and $J^K_j$ has positive distance from $H_j$, we have $D_{g_{H_{j}}(J^K_j)}\subset (\Omega^K)_{H_j}=\Omega_{j-1}$. So $K_j$ and $\Omega_j$ satisfy the properties we need. This finishes the proof of the existence part.

Now we prove the uniqueness. Suppose $\til V$ is another $\R$-symmetric conformal map defined on $\Omega^K$ such that $\til V_K=W$. Then
$$g_{V(K)}\circ V=W\circ g_K=g_{\til V(K)}\circ \til V$$
on $\Omega\sem \ha K$. Thus, $\til V\circ V^{-1}=f_{\til V(K)}\circ g_{V(K)}$ on $V(\Omega\sem \ha K)=V(\Omega)\sem V(\ha K)$. We know that $\til V\circ V^{-1}$ is a conformal map defined on $V(\Omega)$, while $f_{\til V(K)}\circ g_{V(K)}$ is a conformal map defined on $\ha\C\sem \ha{V(K)}=\ha C\sem V(\ha K)$. Since $V(\Omega)$ and $\ha\C\sem V(\ha K)$ cover $\ha C$, we may define an analytic function $h$ on $\C$ such that $h=\til V\circ V^{-1}$ on $V(\Omega)$ and $h=f_{\til V(K)}\circ g_{V(K)}$ on $\ha\C\sem V(\ha K)$.  From the properties of $f_{\til V(K)}$ and $g_{V(K)}$, we see that $h(z)-z\to 0$ as $z\to\infty$. So $h=\id$, which implies that $\til V=V$. So the uniqueness is proved. \end{proof}

\begin{Definition}
  We use $W^K$ to denote the unique $V$ in Theorem \ref{W^K}, and call it the lift of $W$ via $K$. Let $W^*$ be the map defined by
   $W^*(K)=W^K(K)$. \label{W^K-def}
\end{Definition}

\no{\bf Remarks.}
\begin{enumerate}
    \item $(W_K)^K=W$ and $(W^K)_K=W$.
    \item The range of $W^K$ is $W^K(\Omega^K)=(W(\Omega))^{W^*(K)}$.
  \item   $W^{K_1\cdot K_2}=(W^{K_2 })^{K_1}$, $V^{W^K(K)}\circ W^K=(V\circ W)^K$, and $(W^K)^{-1}=(W^{-1})^{W(K)}$.
   \item The domain (resp.\ range) of $W^*$ is the set of $\HH$-hulls whose supports are contained in the domain (resp.\ range) of $W$;  and $S_{W^*(K)}=W(S_K)$.
   \item $V^*\circ W^*=(V\circ W)^*$; $(W^*)^{-1}=(W^{-1})^*$.
\end{enumerate}

\begin{Lemma}
  Suppose $K_1\prec K_2$, $S_{K_2}$ lies in the domain of an $\R$-symmetric conformal map $W$, and $\infty\not\in W(S_{K_2})$. Then $W^*(K_1)\prec W^*(K_2)$, and
  \BGE W^*(K_2):W^*(K_1)=W^{K_2}(K_2:K_1).\label{K2:K1}\EDE \label{W*K/K}
\end{Lemma}
\begin{proof} From Lemma \ref{SKprec}, $S_{K_1}\subset S_{K_2}$. So $W^{K_1}$ and $W^{K_2}$ exist. Let $K_0=K_2:K_1\subset K_2$. Then $W^{K_2}(K_0)\subset W^{K_2}(K_2)$ and
$$W^{K_2}(K_2)/W^{K_2}(K_0)=g_{W^{K_2}(K_0)}\circ W^{K_2}(K_2\sem K_0)$$
$$=g_{W^{K_2}(K_0)}\circ W^{K_2}\circ f_{K_0}(K_2/K_0)=(W^{K_2})_{K_0}(K_2/K_0)$$
$$=(W^{K_0\cdot K_1})_{K_0}(K_1)=W^{K_1}(K_1).$$ Thus, $W^{K_1}(K_1)\prec W^{K_2}(K_2)$ and $W^{K_2}(K_2):W^{K_1}(K_1)=W^{K_2}(K_0)$.
\end{proof}

\begin{figure}[htp]
\centering
\includegraphics[scale=0.80]{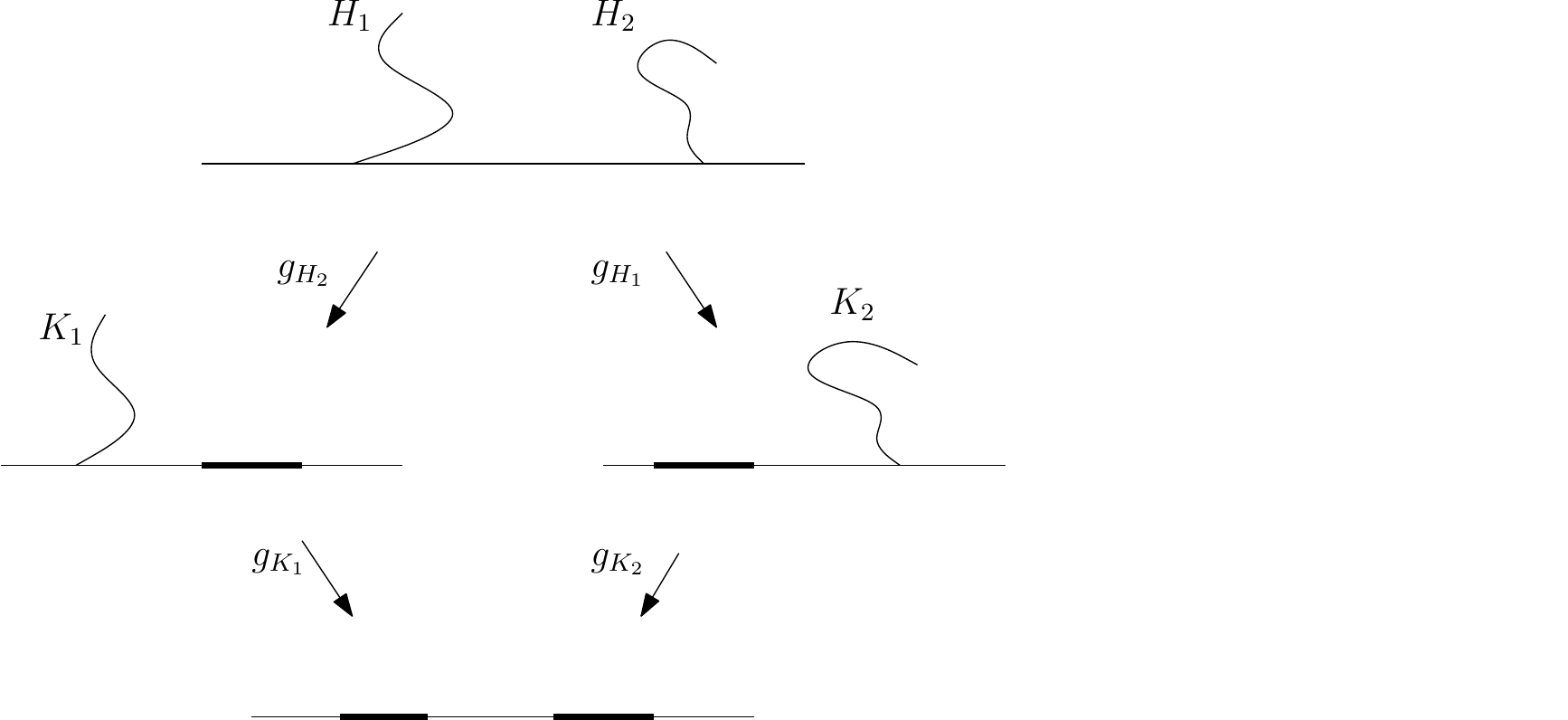}
\caption{The pair $(H_1,H_2)$ uniquely determines the pair $(K_1,K_2),$ and vice versa, see Theorem \ref{bijection-gf*}
and Definition \ref{g*} .}\label{fig2-15}
\end{figure}

\begin{Definition}
Let ${\cal P}^*$ denote the set of pair of $\HH$-hulls $(H_1,H_2)$ such that $\ha H_1\cap\ha H_2=\emptyset$. Let ${\cal P}_*$ denote the set of pair of $\HH$-hulls $(K_1,K_2)$ such that $S_{K_1}\cap S_{K_2}=\emptyset$. Define $g_*$ on  ${\cal P}^*$ by $g_*(H_1,H_2)=(g_{H_2}(H_1),g_{H_1}(H_2))$. Define $f^*$ on ${\cal P}_*$ by $f^*(K_1,K_2)=(f_{K_2}^*(K_1),f_{K_1}^*(K_2))$. \label{g*}
\end{Definition}

\no{\bf Remarks.}
\begin{enumerate}
  \item $g_*$ is well defined on ${\cal P}^*$ because for $j=1,2$, $\ha K_{3-j}$ is contained in the domain of $g_{K_j}$: $\ha\C\sem \ha K_j$. The value of $g_*$ is a pair of $\HH$-hulls.
 \item   $f^*$ is well defined on ${\cal P}_*$ because for $j=1,2$, $S_{K_{3-j}}$ is contained in the domain of $f_{K_j}$: $\ha\C\sem S_{K_j}$. The value of $f^*$ is a pair of $\HH$-hulls.
\end{enumerate}

\begin{Theorem}
 $g_*$ and $f^*$ are bijections between ${\cal P}^*$ and ${\cal P}_*$, and are inverse of each other.
 Moreover, if $(H_1,H_2)=f^*(K_1,K_2)$, then
 \begin{enumerate}
   \item [(i)] $H_1\cdot K_2=H_2\cdot K_1=H_1\cup H_2$;
   \item [(ii)] $f_{K_2}(S_{K_1})=S_{H_1}$ and $f_{K_1}(S_{K_2})=S_{H_2}$;
   \item [(iii)] $S_{H_1\cup H_2}=S_{K_1}\cup S_{K_2}$.
 \end{enumerate}
 \label{bijection-gf*}
\end{Theorem}
\begin{proof}
  Let $(H_1,H_2)\in{\cal P}^*$ and $(K_1,K_2)=g_*(H_1,H_2)$. Then $(\ha\C\sem \ha H_1)_{H_2}=\ha\C\sem \ha K_1$, $S_{H_1}\subset\ha\C\sem \ha K_2$, and $(\ha\C\sem S_{H_1})_{K_2}=\ha\C\sem g_{K_2}(S_{H_1})$. Since $g_{H_1}:\ha\C\sem \ha H_1\conf \ha\C\sem S_{H_1}$ and $g_{H_1}(H_2)=K_2$, we get $(g_{H_1})_{H_2}:\ha\C\sem \ha K_1\conf \ha\C\sem g_{K_2}(S_{H_1})$. From the normalization of $g_{H_1},g_{H_2},g_{K_2}$ at $\infty$, we conclude that
  \BGE (g_{H_1})_{H_2}=g_{K_1},\quad g_{K_2}(S_{H_1})=S_{K_1}.\label{gHK}\EDE
  From $S_{H_1}\subset\ha\C\sem \ha K_2$ and $g_{K_2}(S_{H_1})=S_{K_1}$, we see that $S_{K_1}\cap S_{K_2}=\emptyset$, i.e., $(K_1,K_2)\in{\cal P}_*$. Since $f_{H_1}=g_{H_1}^{-1}$, $f_{K_1}=g_{K_1}^{-1}$, and $g_{H_1}(H_2)=K_2$, from (\ref{gHK}) we get $(f_{H_1})_{K_2}=f_{K_1}$, which implies that $(f_{K_1})^{K_2}=f_{H_1}$. Thus, $f_{K_1}^*(K_2)=f_{H_1}(K_2)=H_2$. Similarly, $f_{K_2}^*(K_1)=H_1$. Thus, $f^*(K_1,K_2)=(H_1,H_2)$. So $f^*\circ g_*=\id_{{\cal P}^*}$.

   Let $(K_1,K_2)\in{\cal P}_*$ and $H_1=f_{K_2}^*(K_1)$. Then $S_{H_1}=f_{K_2}(S_{K_1})$ is disjoint from $\ha K_2$. Thus, we may define another $\HH$-hull $H_2:=f_{H_1}(K_2)$. Then $\ha H_2\subset\ha\C\sem \ha H_1$. So $(H_1,H_2)\in{\cal P}^*$.
   We have $(\ha\C\sem S_{K_2})^{K_1}=\ha\C\sem f_{K_1}(S_{K_2})$ and $(\ha\C\sem \ha K_2)^{H_1}=\ha\C\sem \ha H_2$. Since $f_{K_2}:\ha\C\sem S_{K_2}\conf \ha\C\sem \ha K_2$ and $(f_{K_2})^*(K_1)=H_1$, we see that $(f_{K_2})^{K_1}:\ha\C\sem f_{K_1}(S_{K_2})\conf \ha\C\sem \ha H_2$. From the normalization of $f_{K_1},f_{H_1},f_{K_2}$ at $\infty$, we conclude that
   \BGE (f_{K_2})^{K_1}=f_{H_2},\quad f_{K_1}(S_{K_2})=S_{H_2}.\label{f-S}\EDE
   Since $H_1=f_{K_2}^*(K_1)$, we get $f_{K_2}=g_{H_1}\circ f_{H_2}\circ f_{K_1}$ on $(\ha\C\sem S_{K_2})\sem S_{K_1}$, which implies that $f_{H_1}\circ f_{K_2}= f_{H_2}\circ f_{K_1}$ on $\ha\C\sem (S_{K_1}\cup S_{K_2})$. So
   \BGE  H_2\cdot K_1=H_1\cdot K_2=H_1\cup f_{H_1}(K_2)=H_1\cup H_2.\label{circ-f-H}\EDE
   Thus, $K_1=g_{H_2}(H_1)$ and $K_2=g_{H_1}(H_2)$, i.e., $(K_1,K_2)=g_*(H_1,H_2)$. This shows that the range of $g_*$ is ${\cal P}_*$, which combining with $f^*\circ g_*=\id_{{\cal P}^*}$ shows that $f^*=(g_*)^{-1}$ and $g_*=(f^*)^{-1}$.

   In the previous paragraph, since $(K_1,K_2)=g_*(H_1,H_2)$, $f^*(K_1,K_2)=(H_1,H_2)$. Thus, (i) follows from (\ref{circ-f-H}); the second parts of (ii) follow from (\ref{f-S}), and the first part follows from symmetry. Finally, since $g_{K_2}\circ g_{H_1}=g_{H_1\cdot K_2}=g_{H_1\cup H_2}$, from $g_{H_1}:\ha\C\sem (\ha H_1\cup \ha H_2)\conf \ha\C\sem(S_{H_1}\cup\ha K_2)$, $g_{K_2}:\ha\C\sem(S_{H_1}\cup \ha K_2)\conf \ha\C\sem (g_{K_2}(S_{H_1})\cup S_{K_2})$, and (\ref{gHK}), we get (iii).
\end{proof}

\begin{Definition}
  For $(K_1,K_2)\in {\cal P}_*$, we define the quotient union of $K_1$ and $K_2$ to be
  $K_1\vee K_2=H_1\cup H_2$, where $(H_1,H_2)=f^*(K_1,K_2)$.
\end{Definition}

\no{\bf Remark.} From Theorem \ref{bijection-gf*},  $K_1,K_2\prec K_1\vee K_2$ and $S_{K_1\vee K_2}=S_{K_1}\cup S_{K_2}$.
\vskip 4mm

The space of $\HH$-hulls has a natural metric $d_{\cal H}$ described in Appendix \ref{C}.  Let ${\cal H}_S$ denote the set of $\HH$-hulls whose supports are contained in $S$. From Lemma \ref{compact-new},  if $F$ is compact, $({\cal H}_F,d_{\cal H})$ is compact, and $H_n\to H$ in ${\cal H}_F$ implies that $f_{H_n}\luto f_H$ in $\C\sem F$.

\begin{Theorem}
  \begin{enumerate}
    \item [(i)] Let $F\subset \R$ be compact. Let $W$ be an $\R$-symmetric conformal map whose domain contains $F$. Then $W^*:{\cal H}_F\to {\cal H}_{W(F)}$ is continuous.
    \item [(ii)] Let $E$ and $F$ be two nonempty compact subsets of $\R$ with $E\cap F=\emptyset$. Then $f^*$ and $(K_1,K_2)\mapsto K_1\vee K_2$ are continuous on ${\cal H}_{E}\times {\cal H}_{F}$.
  \end{enumerate} \label{continuous}
\end{Theorem}
\begin{proof}
  (i). First, $W^*$ is well defined on ${\cal H}_F$, and the range of $W^*$ is ${\cal H}_{W(F)}$. Suppose $(H_n)$ is a sequence in ${\cal H}_F$ and $H_n\to H_0\in{\cal H}_F$. To prove the continuity of $W^*$, we need to show that $W^*(H_n)\to W^*(H_0)$. Suppose this is not true. Since ${\cal H}_{W(F)}$ is compact, by passing to a subsequence, we may assume that  $W^*(H_n)\to K_0 \ne W^*(H_0)$. For each $n_k$, $W^{H_{n_k}}=f_{W^*(H_{n_k})}\circ W\circ g_{H_{n_k}}$ on $f_{H_{n_k}}(\Omega\sem F)$. We have $g_{H_{n_k}}\luto g_{H_0}$ in $f_{H_{0}}(\Omega\sem F)$ and $f_{W^*(H_{n_k})}\luto f_{K_0}$ in $W(\Omega)\sem W(F)$. Thus, $W^{H_{n_k}}\luto f_{K_0}\circ W\circ g_{H_0}=:V$ in $f_{H_{0}}(\Omega\sem F)$. The domain of $W^{H_{n_k}}$ is $\Omega^{H_{n_k}}=\ha H_{n_k}\cup f_{H_{n_k}}(\Omega\sem S_{H_{n_k}})$, which converges to $\Omega^{H_{0}}=\ha H_{0}\cup f_{H_{0}}(\Omega\sem S_{H_{0}})\supset f_{H_{0}}(\Omega\sem F)$. It is clear that $\Omega^{H_0}\sem f_{H_0}(\Omega\sem F)$ is compact. Since $W^{H_{n_k}}\luto V$ in $f_{H_{0}}(\Omega\sem F)$, from the maximum principle,  $W^{H_{n_k}}$ converges locally uniformly in $\Omega^{H_0}$. We still let $V$ denote the limit function. Since $H_{n_k}\to H_0$ and $W^{H_{n_k}}(H_{n_k})\to K_0$, we have $V(H_0)=K_0$. Since $f_{K_0}\circ W\circ g_{H_0}=V$ in $f_{H_{0}}(\Omega\sem F)$, we see that $f_{V(H_0)}\circ W\circ g_{H_0}=V$ in $f_{H_0}(\Omega\sem S_{H_0})$. Thus, $V=W^{H_0}$. So $K_0=W^{H_0}(H_0)=W^*(H_0)$. This is the contradiction we need.

  (ii). To show $f^*$ is continuous, it suffices to show that, if $(K_1^{n},K_2^{n})$ is a sequence in ${\cal H}_{E}\times {\cal H}_{F}$ which converges to $(K_1^{0},K_2^{0})\in {\cal H}_{E}\times {\cal H}_{F}$,  then it has a subsequence $(K_1^{(n_k)},K_2^{(n_k)})$ such that $f^*(K_1^{(n_k)},K_2^{(n_k)})\to f^*(K_1^{0},K_2^{0})$. Let $(H_1^{n},H_2^{n})=f^*(K_1^{n},K_2^{n})$, $n\in\N$. From Theorem \ref{bijection-gf*} (iii), $S_{H_1^{n}\cup H_2^{n}}=S_{K_1^{n}}\cup S_{K_2^{n}}\subset E\cup F$. From Lemma \ref{compact-new}, $(H_1^{n}\cup H_2^{n})$ has a convergent subsequence with limit in ${\cal H}_{E\cup F}$. From Lemma \ref{lemma-LERW}, $S_{H_1^{n}}\subset S_{H_1^{n}\cup H_2^{n}}^*\subset A$, where $A$ is the convex hull of $E\cup F$. From Lemma \ref{compact-new}, $(H_1^{n})$ has a convergent subsequence. For the same reason, $(H_2^{n})$ also has a convergent subsequence.
  By passing to subsequences, we may assume that $H_1^{n}\cup H_2^{n}\to M^{0}\in {\cal H}_{E\cup F}$ and $H_j^{n}\to H_j^{0}$, $j=1,2$.

From Theorem \ref{bijection-gf*} (i) and the continuity of the dot product, we get $H_1^{0}\cdot K_2^{0}=H_2^{0}\cdot K_1^{0}=M^{0}$. This implies that $M^{0}=H_1^{0}\cup f_{H_1^{0}}(K_2^{0})$. The measures $(\mu_{H_1^{n}})$ (see Appendix \ref{C}) converges to $\mu_{H_1^{0}}$ weakly. Each $\mu_{H_1^{n}}$ is supported by $S_{H_1^{n}}$. From Theorem \ref{bijection-gf*} (ii), $S_{H_1^{n}}=f_{{K_2^{n}}}(S_{K_1^{n}})\subset f_{K_2^{n}}(E)$. Since $E$ is a compact subset of $\C\sem F$, we have $f_{K_2^{n}}\to f_{K_2^{0}}$ uniformly on $E$. Thus, $f_{K_2^{n}}(E)\to f_{K_2^{0}}(E)$ in the Hausdorff metric. So $\mu_{H_1^{0}}$ is supported by $ f_{K_2^{0}}(E)$, which implies that $S_{H_1^{0}}\subset f_{K_2^{0}}(E)$. Hence $f_{H_1^{0}}(K_2^{0})$ is another $\HH$-hull, which is bounded away from $H_1^{0}$. From $K_2^{n}\to K_2^{0}$ we have $\HH\sem K_2^{n}\dto \HH\sem K_2^{0}$. From (\ref{f-z}) we get $f_{H_1^{n}}\luto f_{H_1^{0}}$ in $\C\sem S_{H_1^{0}}$. Thus, $\HH\sem f_{H_1^{n}}(K_2^{n})\dto\HH\sem f_{H_1^{0}}(K_2^{0})$. Since $H_2^{n}=f_{H_1^{n}}(K_2^{n})$, we have $\HH\sem H_2^{n}\dto \HH\sem f_{H_1^{0}}(K_2^{0})$. On the other hand, $\HH\sem H_2^{n}\dto \HH\sem H_2^{0}$. Since $\HH\sem H_2^{0}$ and $\HH\sem f_{H_1^{0}}(K_2^{0})$ both contain a neighborhood of $\infty$ in $\HH$, they must be the same domain. Thus, $H_2^{0}=f_{H_1^{0}}(K_2^{0})$ is bounded away from $H_1^{0}$, i.e., $(H_1^{n},H_2^{n})\in{\cal P}^*$. For the same reason, $H_1^{0}=f_{H_2^{0}}(K_1^{0})$. Thus, $(H_1^{n},H_2^{n})\to (H_1^{0},H_2^{0})=f^*(K_1^{0},K_2^{0})$. This shows that $f^*$ is continuous. Finally, since $K_1\vee K_2=H_1\cdot K_2$ if $(H_1,H_2)=f^*(K_1,K_2)$, we see that $(K_1,K_2)\mapsto K_1\vee K_2$ is also continuous.
\end{proof}

\begin{Corollary}
  \begin{enumerate}
    \item [(i)] Let $W$ be an $\R$-symmetric conformal map with domain $\Omega$. Then $W^*$ is measurable on ${\cal H}_{\Omega\cap \R}$.
    \item [(ii)] $f^*$ and $(K_1,K_2)\mapsto K_1\vee K_2$ are measurable on ${\cal P}_*$.
  \end{enumerate}\label{continuous-Cor}
\end{Corollary}
\begin{proof}
  (i) We may find an increasing sequence of compact subsets $(F_n)$ of $\Omega\cap\R$ such that ${\cal H}_{\Omega\cap \R}=\bigcup_{n=1}^\infty{\cal H}_{F_n} $. From Theorem \ref{continuous} (i), $W^*$ is continuous on each ${\cal H}_{F_n} $. Thus, $W^*$ is measurable on ${\cal H}_{\Omega\cap \R}$.

  (ii) We may find a sequence of pairs of disjoint bounded closed intervals of $\R$: $(E_n,F_n)$, $n\in\N$, such that ${\cal P}_*=\bigcup_{n=1}^\infty {\cal H}_{E_n}\times{\cal H}_{F_n}$. 
  From Theorem \ref{continuous} (ii), $f^*$ and $(K_1,K_2)\mapsto K_1\vee K_2$ are continuous on each ${\cal H}_{E_n}\times {\cal H}_{F_n}$, and so they are measurable on ${\cal P}_*$.
\end{proof}

\subsection{Hulls in the unit disc}
A subset $K$ of $\D=\{|z|<1\}$ is called a $\D$-hull if $\D\sem K$ is a simply connected domain containing $0$. For every $\D$-hull $K$, there is a unique $g_K:\D\sem K\conf \D$ such that $g_K(0)=0$ and $g_K'(0)>0$. Then $\ln g_K'(0)\ge 0$ is called the $\D$-capacity of $K$, and is denoted by $\dcap(K)$. Let $f_K=g_K^{-1}$. 

We may define $K_1\cdot K_2$,  $K_2/K_1$ (when $K_1\subset K_2$), and $K_1\prec K_2$ on the space of $\D$-hulls as in Definition \ref{/}. Then the remarks after Definition \ref{/} still hold if $\HH$ is replaced by $\D$ and $\hcap$ is replaced by $\dcap$. Then we may define $K_2:K_1$ (when $K_1\prec K_2$) as in Definition \ref{//}. For a $\D$-hull $K$, the base $B_K$ of $K$ is $\lin{K}\cap \TT$, and the double of $K$ is $\ha K=K\cup I_\TT(K)\cup B_K$, where $I_\TT(z):=1/\lin z$. Then $g_K$ extends to a conformal map (still denoted by $g_K$) on $\ha\C\sem\ha K$, which commutes with $I_\TT$. Moreover, $g_K(\ha\C\sem \ha K)=\ha\C\sem S_K$ for some compact $S_K\subset \TT$, which is called the support of $K$. So $f_K$ extends to a conformal map from $\ha\C\sem S_K$ onto $\ha\C\sem \ha K$, which commutes with $I_\TT$. Then Lemma \ref{extend} and Lemma \ref{SKprec} still hold here.

We may define $\TT$-symmetric sets and $\TT$-symmetric conformal maps using Definition \ref{R-symmetric} with $\R$ and $\HH$ replaced by $\TT$ and $\D$, respectively. For a $\TT$-symmetric domain $\Omega$ and a $\D$-hull $K$, we may define domains $\Omega_K$ (when $\ha  K\subset \Omega$) and $\Omega^K$ (when $S_K\subset \Omega$) using Definition \ref{OmegaK}. If $W$ is a $\TT$-symmetric conformal map with domain $\Omega$, and if $\Omega_K$ is defined, we may then define $W_K$ using Definition \ref{WK}, which is a $\TT$-symmetric conformal map on $\Omega_K$. The remarks after Definition \ref{R-symmetric}, Definition \ref{OmegaK}, and Definition \ref{WK} hold here with minor modifications. We claim that Theorem \ref{W^K} holds here with modifications. We need several lemmas.

The lemma below relates the $\HH$-hulls with $\D$-hulls. To distinguish the two set of symbols, we use $f^{\HH}_K$,  $g^\HH_K$, $B_K^\R$, $S_K^\R$, and $\ha K^\R$ for $\HH$-hulls, and $f^{\D}_K$,  $g^\D_K$, $B_K^\TT$, $S_K^\TT$, and $\ha K^\TT$ for $\D$-hulls.

\begin{Theorem}
  \begin{enumerate}
    \item [(i)] Let $W$ be a M\"obius transformation that maps $\D$ onto $\HH$, and $K$ be a $\D$-hull such that $W^{-1}(\infty)\not\in S_K^\TT$. Then there is a unique M\"obius transformation $W^K$ that maps $\D$ onto $\HH$ such that $W^K(K)$ is an $\HH$-hull,  $g^\HH_{W^K(K)}\circ W^K\circ f^\D_K=W$ in $\ha\C\sem S_K^\TT$, and $S^\R_{W^K(K)}=W(S^\TT_K)$.
        \item [(ii)] Let $W$ be a M\"obius transformation that maps $\HH$ onto $\D$, and $K$ be an $\HH$-hull. Then there is a unique M\"obius transformation $W^K$ that maps $\HH$ onto $\D$ such that $W^K(K)$ is a $\D$-hull,  $g^\D_{W^K(K)}\circ W^K\circ f^\HH_K=W$ in $\ha\C\sem S_K^\R$, and $S^\TT_{W^K(K)}=W(S^\R_K)$.
  \end{enumerate}\label{Mob}
\end{Theorem}
\begin{proof}
  (i) Let $z_0=W^{-1}(\infty)\in\TT\sem S_K^\TT$. Then $w_0:=f_K^\D(z_0)\in\TT\sem B_K^\TT$ is well defined.
  Let $W^K_0(z)=i\frac{w_0+z}{w_0-z}$. Then $W^K_0$ is a M\"obius transformation that maps $\D$ onto $\HH$ and takes $w_0$ to $\infty$. Let $L_0=W^K_0(K)$. Since $w_0 $ is bounded away from $K$, we see that $L_0$ is an $\HH$-hull. We have $W^K_0:\ha\C\sem \ha K\conf \ha\C\sem \ha L_0$. Define $G=g^\HH_{L_0}\circ W^K_0\circ f^\D_K\circ W^{-1}$ on $\ha\C\sem W(S_K^\TT)$. Then $G:\ha\C\sem W(S_K^\TT)\conf \ha\C\sem S_{L_0}^\R$, fixes $\infty$, and maps $\HH$ onto $\HH$. So $G(z)=az+b$ for some $a>0$ and $b\in\R$. Let $W^K=G^{-1}\circ W^K_0$. Then $W^K$ is also a M\"obius transformation that maps $\D$ onto $\HH$, and $W^K(K)$ is also an $\HH$-hull with $S_{W^K(K)}^\R=G^{-1}(S_{L_0}^\R)=W(S_K^\TT)$ and $g^\HH_{W^K(K)}=G^{-1}\circ g^\HH_{L_0}\circ G$. Thus,
  $$g^\HH_{W^K(K)}\circ W^K\circ f^\D_K\circ W^{-1}=G^{-1}\circ g^\HH_{L_0}\circ G\circ G^{-1}\circ W^K_0\circ  f^\D_K\circ W^{-1}$$
  $$=G^{-1}\circ g^\HH_{L_0}\circ   W^K_0\circ  f^\D_K\circ W^{-1}=G^{-1}\circ G=\id_{\ha\C\sem W(S_K)}.$$
  This implies that $g^\HH_{L}\circ W^K\circ f^\D_K=W$ in $\ha\C \sem S_K$.  So we proved the existence. On the other hand, if $W^K$ satisfies the desired property, then from $W^K=f^\HH_{L}\circ W\circ g^\D_K$ we get $W^K(w_0)=\infty$. So $W^K=G_0\circ W^K_0$, where $G_0(z)=az+b$ for some $a>0$ and $b\in\R$. The above argument shows that $G_0=G^{-1}$. So we get the uniqueness.

  (ii) We may use the proof of (i) with slight modifications: replace $\infty$  by $0$, swap $\HH$ and $\D$, swap $\R$ and $\TT$, and define $W^K_0(z)=\frac{z-w_0}{z-\lin{w_0}}$.
\end{proof}

We also use $W^*(K)$ to denote the hull $W^K(K)$ in the above lemma. The following lemma is similar to Lemma \ref{W*K/K}.
\begin{Lemma}
Let $K_1$ and $K_2$ be two $\HH$-(resp.\ $\D$-)hulls such that $K_1\prec K_2$. Let $W$ be a M\"obius transformation that maps $\HH$ onto $\D$ (resp.\ maps $\D$ onto $\HH$) such that $\infty\not\in W(S_{K_2})$. Then $W^*(K_1)\prec W^*(K_2)$ and (\ref{K2:K1}) still holds.
\label{W*K/K-Mob}
\end{Lemma}

The following lemma is used to treat the case $S_K=\TT$ in Theorem \ref{W^K-d}.

\begin{Lemma}
  Let $W$ be a $\TT$-symmetric conformal map with domain $\Omega\supset\TT$. Let $(K_n)$ be a sequence of $\D$-hulls which converges to $K$. Suppose that for each $n$, there is a $\TT$-symmetric conformal map $V^{\langle n\rangle}$ defined on $\Omega^{K_n}$ such that $V^{\langle n\rangle}_{K_n}=W$. Then there is a $\TT$-symmetric conformal map $V $ defined on $\Omega^{K}$ such that $V_K=W$. Moreover, $V(K)$ is a subsequential limit of $(V^{\langle n\rangle }(K_n))$. \label{T-support}
\end{Lemma}
\begin{proof} Since $K_n\to K$, $\Omega^{K_n}\dto\Omega^K$. Since $V^{\langle n\rangle}$ maps $\Omega^{K_n}\cap\D$ into $\D$, the family $(V^{\langle n\rangle}|_{\Omega^{K_n}\cap\D})$ is uniformly bounded. Thus, $(V^{\langle n\rangle})$ contains a subsequence, which convergence locally uniformly in $\Omega^K\cap\D$. To save the symbols, we assume that $(V^{\langle n\rangle})$ itself converges locally uniformly in $\Omega^K\cap\D$. Since each $V^{\langle n\rangle}$ is $\TT$-symmetric, the sequence also converges locally uniformly in $\Omega^K\cap\D^*$. From the maximum principle, $(V^{\langle n\rangle})$ converges locally uniformly in $\Omega^K$. Let $V$ be the limit function. Since each $V^{\langle n\rangle}$ maps $\TT$ onto $\TT$, and $V^{\langle n\rangle}\to V$ uniformly on $\TT$,  $V$ can not be constant. From Lemma \ref{domain convergence*}, $V$ is a conformal map. It is $\TT$-symmetric because each $V^{\langle n\rangle}$ is $\TT$-symetric. Since $K_n\to K$, we have $V^{\langle n\rangle}(K_n)\to V(K)$. From $V^{\langle n\rangle}_{K_n}=W$ we have $g_{V^{\langle n\rangle}(K_n)}\circ V^{\langle n\rangle}\circ f_{K_n}=W$ in $\Omega\sem \TT$. Letting $n\to\infty$ we get $g_{V(K)}\circ V\circ f_K=W$ in $\Omega\sem\TT$. By continuation, this equality also holds on $\Omega\sem S_K$. Thus, $V_K=W$.
\end{proof}

\begin{Theorem}
Let $W$ be a $\TT$-symmetric conformal map with domain $\Omega$. Let $K$ be a $\D$-hull such that $S_K\subset\Omega$. Then there is a unique $\TT$-symmetric conformal map $V$ defined on $\Omega^K$ such that $V_K=W$. \label{W^K-d}
\end{Theorem}
 \begin{proof} We first consider the existence. Case 1. $S_K^\TT\ne\TT$. We will apply Theorems \ref{W^K} and \ref{Mob} for this case. Pick $z_0\in\TT\sem S_K^\TT$ and let $h(z)=i\frac{z_0+z}{z_0-z}$. From Theorem \ref{Mob} (i), there is a M\"obius transformation $h^K$ that maps $\D$ onto $\HH$ such that $L:=h^K(K)$ is an $\HH$-hull, and $g^\HH_{L}\circ h^K\circ f^\D_K=h$ in $\ha\C\sem S_K^\TT$.   Since $W$ is a homeomorphism on $S_K$, $W(S_K)\ne \TT$. So there is $z_W\in\TT\sem W(S_K)$. Let $h_W(z)=z_W\cdot\frac{z-i}{z+i}$. Then $h_W$ is a M\"obius transformation that maps $\HH$ onto $\D$ and takes $\infty$ to $z_W$. Let $\til W=h_W^{-1}\circ W\circ h^{-1}$. Then $\til W$ is an $\R$-symmetric conformal map with domain $h(\Omega)$, and $\til W(S^\R_L)=h_W^{-1}\circ W(S_K^\TT) \not\ni\infty$. From Theorem \ref{W^K}, there is an $\R$-symmetric conformal map $\til V$ with domain $\ha L^\R\cup f_L^\R(h(\Omega)\sem S_L^\R)$ such that $L^*:=\til V(L)$ is an $\HH$-hull, and $\til V=f^\HH_{L^*}\circ \til W\circ g^\HH_L$ in $\ha\C\sem \ha L^\R$. From Theorem \ref{Mob} (ii), there is a M\"obius transformation $h_W^{L^*}$ that maps $\HH$ onto $\D$ such that $K^*:=h_W^{L^*}(L^*)$ is a $\D$-hull, and $g^\D_{K^*}\circ h_W^{L^*}\circ f^\HH_{L^*}=h_W$ in $\ha\C\sem S_{L^*}^\R$. Finally, let $V=h_W^{L^*}\circ \til V\circ h^K$. Then
 $$V(K)=h_W^{L^*}\circ \til V(L)=h_W^{L^*}(L^*)=K^*,$$
 and
 $$g_{K^*}\circ V\circ f_K=g_{K^*}\circ h_W^{L^*}\circ \til V\circ h^K\circ f_K$$
 $$=g_{K^*}\circ h_W^{L^*}\circ f^\HH_{L^*}\circ \til W\circ g^\HH_L\circ h^K\circ f_K$$
 $$=h_W\circ \til W\circ h=W$$
 in $\ha\C\sem \ha K$. This finishes the existence part for Case 1.

 Case 2. $S_K=\TT$. First, we may approximate $K$ using $\D$-hulls bounded by $\TT$ and a Jordan curve in $\D$. For example, let $J_n=f_{K}(\{|z|=1-1/(2n)\})$, and let $K_n=\D\sem D_{J_n}$. Then each $K_n$ is a $\D$-hull, and $K_n\to K$. Second, if $K'$ has the form of $\D\sem D_J$ for some Jordan curve $J$, then we may define a curve $\beta$, which starts from $\beta(0)=z_0\in\TT$, then follows a simple curve in $\D\cap D_J^*$ to a point on $J$, and then follows $J$ in the clockwise direction, and ends when it finishes one round. Suppose the domain of $\beta$ is $[0,1]$. Then $\beta$ is simple on $[0,1-\eps]$ for any $\eps>0$. Let $K_n=\beta((0,1-1/n])$, $n\in\N$. Then each $K_n$ is a $\D$-hull with $S_{K_n}\ne \TT$, and $K_n\to K'$. Thus, $K$ can be approximated by a sequence of $\D$-hulls $(K_n)$ such that $S_{K_n}\ne\TT$ for each $K_n$. Then the existence of $V$ follows from Case 1 and Lemma \ref{T-support}.

 Now we prove the uniqueness. Suppose $\til V$ is another $\TT$-symmetric conformal map defined on $\Omega^K$ such that $\til V_K=W$. We may use the argument in the proof of Theorem \ref{W^K} to construct an analytic function $h$ on $\C$ such that $h=\til V\circ V^{-1}$ on $V(\Omega)$ and $h=f_{\til V(K)}\circ g_{V(K)}$ on $\C\sem V(\ha K)$. Then $h$ is $\TT$-symmetric.
  From the properties of $f_{\til V(K)}$ and $g_{V(K)}$, we see that $h(0)=0$ and $h'(0)>0$. So $h=\id$, which implies that $\til V=V$.
 \end{proof}

 We may then define $W^K$ and $W^*$ using Definition \ref{W^K-def} with Theorem \ref{W^K-d} in place of Theorem \ref{W^K} and $\D$ in place of $\HH$. The remarks after Definition \ref{W^K-def} hold here with minor modifications, and so does Lemma \ref{W*K/K}. Then we define ${\cal P}^*$, ${\cal P}_*$, $g_*$, and $f^*$ using Definition \ref{g*} with $\HH$  replaced by $\D$. Then Theorem \ref{bijection-gf*} still holds here, and we may define the quotient union $K_1\vee K_2$ for $(K_1,K_2)\in{\cal P}_*$. 

 The space of $\D$-hulls has a natural metric $d_{\cal H}$ described in Appendix \ref{D}. Let ${\cal H}_S$ denote the set of $\D$-hulls whose supports are contained in $S$. We claim that Theorem \ref{continuous} still holds here if every $\R$ is replaced by $\TT$. For part (i), if $F\ne\TT$, then the proof of Theorem \ref{continuous} (i) still goes through with Lemma \ref{compact-D-F} in place of Lemma \ref{compact-new}; if $F=\TT$, then the continuity of $W^*$ follows from Lemma \ref{T-support}. For part (ii), the proof of Theorem \ref{continuous} still goes through with some modifications. The relatively compactness of $(H_n\cup J_n)$ follows from Lemma \ref{compact-D-F} instead of Lemma \ref{compact-new} because $S_{H_n\cup J_n}\subset E\cup F\subsetneqq \TT$. To show the relatively compactness of $(H_n)$ and $(J_n)$, instead of applying Lemma \ref{lemma-LERW}, we now apply Lemma \ref{compact-D-M}, and use the relatively compactness of $(H_n\cup J_n)$ and the inequalities $\dcap(H_n),\dcap(J_n)\le \dcap(H_n\cup J_n)$. In addition, (\ref{f/z}) will be used in place of (\ref{f-z}). This finishes the proof of Theorem \ref{continuous} in the radial case.  Then Corollary \ref{continuous-Cor} in the radial case immediately follows.

 The proof of Theorem \ref{continuous} (i) may also be used to show that the map $K\mapsto W^K(K)$ in Theorem \ref{Mob} (i) (resp.\ (ii)) is continuous if restricted to ${\cal H}_F^\D$ (resp.\ ${\cal H}_F^\HH$), where $F$ is a compact subset of $\TT\sem W^{-1}(\infty)$ (resp.\ $\R$). We then can conclude that the maps $K\mapsto W^K(K)$ in Theorem \ref{Mob} (i) and (ii) are both measurable.

\section{Loewner Equations and Loewner Chains} \label{section-Loewner}
\subsection{Forward Loewner equations}
We review the definitions and basic facts about (forward) Loewner equations. The reader is referred to \cite{LawSLE} for details.
Let $\lambda\in C([0,T))$, where $T\in(0,\infty]$. The chordal Loewner equation driven by $\lambda$ is
$$\pa_t g_t(z)=\frac{2}{g_t(z)-\lambda(t)}, \quad g_0(z)=z.$$
We assume that $g_t(\infty)=\infty$ for $0\le t<\infty$. For $z\in\C$, suppose that the maximal interval for $t\mapsto g_t(z)$ is $[0,\tau_z)$. Let $K_t=\{z\in\HH:\tau_z\le t\}$, i.e., the set of $z\in\HH$ such that $g_t(z)$ is not defined. Then $g_t$ and $K_t$, $0\le t<T$, are called the chordal Loewner maps and hulls driven by $\lambda$. It is known that each $K_t$ is an $\HH$-hull with $\hcap(K_t)=2t$, and for $0<t<T$, $g_t=g_{K_t}$ with exactly the same domain: $\ha\C\sem \ha K_t$. At $t=0$, $K_0=\emptyset$ and $g_0=\id_{\ha\C\sem\{\lambda(0)\}}$.

We say that $\lambda$ generates a chordal trace $\beta$ if
$$\beta(t):=\lim_{\HH\ni z\to \lambda(t)} g_t^{-1}(z)\in\lin{\HH}$$
exists for $0\le t<T$, and $\beta$ is a continuous curve. We call such $\beta$ the chordal trace driven by $\lambda$. If the chordal trace $\beta$ exists, then for each $t$, $\HH\sem K_t$ is the unbounded component of $\HH\sem \beta((0,t])$, and $f_t$ extends continuously from $\HH$ to $\HH\cup\R$. The trace $\beta$ is called simple if it is a simple curve and $\beta(t)\in\HH$ for $0<t<T$, in which case $K_t=\beta((0,t])$ for $0\le t<T$.

The radial Loewner equation driven by $\lambda$ is
$$\pa_t g_t(z)=g_t(z)\frac{e^{i\lambda(t)}+g_t(z)}{e^{i\lambda(t)}-g_t(z)},\quad 0\le t<T;\quad g_0(z)=z.$$
We assume that $g_t(\infty)=\infty$ for $0\le t<\infty$.
For each $t\in[0,T)$, let $K_t$ be the set of $z\in\D:=\{|z|<1\}$ at which $g_t$ is not defined. Then $g_t$ and $K_t$, $0\le t<T$, are called the radial Loewner maps and hulls driven by $\lambda$. It is known that, each $K_t$ is a $\D$-hull with $\dcap(K_t)=t$, and for $0<t<T$, $g_t=g_{K_t}$ with exactly the same domain: $\ha\C\sem \ha K_t$. At $t=0$, $K_0=\emptyset$ and $g_0=\id_{\ha\C\sem \{e^{i\lambda(0)}\}}$.

We say that $\lambda$ generates a radial trace $\beta$ if
$$\beta(t):=\lim_{\D\ni z\to e^{i\lambda(t)}} g_t^{-1}(z)\in\lin\D$$
exists for $0\le t<T$, and $\beta$ is a continuous curve. We call such $\beta$ the radial trace driven by $\lambda$. If the radial trace $\beta$ exists, then for each $t$, $\D\sem K_t$ is the component of $\D\sem \beta((0,t])$ that contains $0$. The trace $\beta$ is called simple if it is a simple curve and $\beta(t)\in\D$ for $0<t<T$, in which case $K_t=\beta((0,t])$ for $0\le t<T$.

Let $\cot_2(z)=\cot(z/2)$. The covering radial Loewner equation driven by $\lambda$ is
$$\pa \til g_t(z)=\cot_2(\til g_t(z)-\lambda(t)),\quad 0\le t<T,\quad \til g_0(z)=z.$$
For each $t\in[0,T)$, let $\til K_t$ be the set of all $z\in\HH$ at which $\til g_t$ is not defined. Then $\til g_t$ and $\til K_t$, $0\le t<T$, are called the covering radial Loewner maps and hulls driven by $\lambda$. We have $\til g_t:\HH\sem\til K_t\conf \HH$. If $g_t$ and $K_t$, $0\le t<T$, are the radial Loewner maps and hulls driven by $\lambda$, then $\til K_t=(e^i)^{-1}(K_t)$ and $e^i\circ \til g_t=g_t\circ e^i$, where $e^i$ denotes the map $z\mapsto e^{iz}$.

For $\kappa>0$, chordal (resp.\ radial) SLE$_\kappa$ is defined by solving the chordal (resp.\ radial) Loewner equation with $\lambda(t)=\sqrt\kappa B(t)$. Such driving function a.s.\ generates a chordal (resp.\ radial) trace, which is simple if $\kappa\in(0,4]$.

\subsection{Backward  Loewner equations}
Let $\lambda\in C([0,T))$. The backward chordal Loewner equation driven by $\lambda$ is
\BGE \pa_t f_t(z)=\frac{-2}{f_t(z)-\lambda(t)}, \quad f_0(z)=z.\label{chordal-backward}\EDE
We assume that $f_t(\infty)=\infty$ for $0\le t<T$. Let $L_t=\HH\sem f_t(\HH)$. We call $f_t$ and $L_t$, $0\le t<T$, the backward chordal Loewner maps and hulls driven by $\lambda$.

Define a family of maps $f_{t_2,t_1}$, $t_1,t_2\in[0,T)$, such that, for any fixed $t_1\in[0,T)$ and $z\in\ha\C\sem\{\lambda(t_1)\}$, the function $t_2\mapsto f_{t_2,t_1}(z)$ is the maximal solution of the ODE
$$ \pa_{t_2} f_{t_2,t_1}(z)=\frac{-2}{f_{t_2,t_1}(z)-\lambda(t_2)},\quad f_{t_1,t_1}(z)=z.$$ 
Note that $f_{t,0}=f_{t}$ and $f_{t,t}=\id_{\ha\C\sem\{\lambda(t)\}}$, $0\le t<T$. If $t_1\in(0,T)$, then $t_2$ could be bigger or smaller than $t_1$. Some simple observations give the following lemma.

\begin{Lemma}
\begin{enumerate}
  \item[(i)] For any $t_1,t_2,t_3\in[0,T)$, $f_{t_3,t_2}\circ f_{t_2,t_1}$ is a restriction of $f_{t_3,t_1}$. In particular, this implies that $f_{t_1,t_2}=f_{t_2,t_1}^{-1}$.
  \item[(ii)] For any fixed $t_0\in[0,T)$, $f_{t_0+t,t_0}$, $0\le t<T-t_0$, are the backward chordal Loewner maps driven by $\lambda(t_0+t)$, $0\le t<T-t_0$. 
  \item[(iii)] For any fixed $t_0\in[0,T)$, $f_{t_0-t,t_0}$, $0\le t\le t_0$, are the (forward) chordal Loewner maps driven by $\lambda(t_0-t)$, $0\le t\le t_0$.
\end{enumerate} \label{ft2t1}
\end{Lemma}

Let $L_{t_2,t_1}=\HH\sem f_{t_2,t_1}(\HH)$ for $0\le t_1\le t_2<T$. From (i), (iii), and the properties of forward chordal Loewner maps, we see that, if $0\le t_1< t_2<T$, then $L_{t_2,t_1}$ is an $\HH$-hull with $\hcap(L_{t_2,t_1})=2(t_2-t_1)$, and $f_{t_2,t_1}=f_{L_{t_2,t_1}}$. If $t_1=t_2$, this is almost still true except that $f_{t_1,t_1}=\id_{\ha\C\sem\{\lambda(t_1)\}}$ and $f_{L_{t_1,t_1}}=f_\emptyset=\id_{\ha\C}$. Since $L_{t,0}=L_t$, and $\lambda(t)\in\R$ does not lie in the range of $f_t$, which is $\ha\C\sem \ha L_t$ for $t>0$, we get the following lemma.

\begin{Lemma}
  For $0\le t<T$, $L_t$ is an $\HH$-hull with $\hcap(L_t)=2t$. If $t\in(0,T)$, then $f_t=f_{L_t}$ with the same domain: $\ha\C\sem S_{L_t}$, and $\lambda(t)\in B_{L_t}$. \label{chordal-S}
\end{Lemma}

If $t_2\ge  t_1\ge t_0$, from $f_{t_2,t_1}\circ f_{t_1,t_0}=f_{t_2,t_0}$  we get $L_{t_2,t_0}=L_{t_2,t_1}\cdot L_{t_1,t_0}$.
From Lemmas \ref{SKprec} and \ref{ft2t1}, we obtain the following lemma.

\begin{Lemma}
  For any $0\le t_1<t_2<T$, $L_{t_1}\prec L_{t_2}$ and $S_{L_{t_1}}\subset S_{L_2(t_2)}$. For any fixed $t_0\in[0,T)$, the family $L_{t_0}:L_{t_0-t}=L_{t_0,t_0-t}$, $0\le t\le t_0$, are the chordal Loewner hulls driven by $\lambda(t_0-t)$, $0\le t\le t_0$. 
  \label{backward-forward-chordal}
\end{Lemma}

Note that $S_{L_0}=S_\emptyset=\emptyset$, and its is easy to see that, for $0<t_0<T$, $S_{L_{t_0}}$ is the set of $x\in\R$ such that the solution $f_t(x)$ to (\ref{chordal-backward}) blows up before or at $t_0$, i.e., $S_{L_{t_0}}=\{x\in\R:\tau_x\le t_0\}$. So every $S_{L_t}$, $0<t<T$, is a real interval, and $\bigcap_{0<t<T} S_{L_t}=\{\lambda(0)\}$.

If for every $t_0\in[0,T)$, $\lambda(t_0-t)$, $0\le t\le t_0$, generates a (forward) chordal trace, which we denote by $\beta_{t_0}(t_0-t)$, $0\le t\le t_0$, then we say that $\lambda$ generates backward chordal traces $\beta_{t_0}$, $0\le t_0<T$. If this happens, then for any $0\le t_1\le t_2<T$, $\HH\sem L_{t_2,t_1}$ is the unbounded component of  $\HH\sem \beta_{t_2}([t_1,t_2))$, and  $f_{t_2,t_1}$ extends continuously from $\HH$ to $\lin\HH$ such that
\BGE f_{t_2,t_1}(\lambda(t_1))=\beta_{t_2}(t_1),\quad 0\le t_1\le t_2<T.\label{ft2t1lambda}\EDE
Here we still use $f_{t_2,t_1}$ to denote the continuation if there is no confusion. 
For $0\le t_0\le t_1\le t_2<T$, the equality $f_{t_2,t_0}=f_{t_2,t_1}\circ f_{t_1,t_0}$ still holds after continuation, which together with (\ref{ft2t1lambda}) implies that
\BGE f_{t_2,t_1}(\beta_{t_1}(t))=\beta_{t_2}(t),\quad 0\le t\le t_1\le t_2<T.\label{ft2t1beta}\EDE

\vskip 4mm
\no{\bf Remark.}  One should keep in mind that each $\beta_t$ is a continuous function defined on $[0,t]$, $\beta_t(0)$ is the tip of $\beta_t$, and $\beta_t(t)$ is the root of $\beta_t$, which lies on $\R$. The parametrization is different from a forward chordal trace $\beta$, of which $\beta(0)$ is the root.
\vskip 4mm

The backward radial Loewner equations and the backward covering radial Loewner equation driven by $\lambda\in C([0,T))$ are the following two equations respectively:
$$ \pa_t f_t(z)=-f_t(z)\frac{e^{i\lambda(t)}+f_t(z)}{e^{i\lambda(t)}-f_t(z)},\quad f_0(z)=z; $$ 
$$ \pa_{t} \til f_{t}(z)=-\cot_2({\til f_{t}(z)-\lambda(t)}),\quad \til f_{0}(z)=z. $$ 
We have $f_t\circ e^i=e^i\circ \til f_t$.
Let $L_t=\D\sem f_t(\D)$. We call $f_t$ and $L_t$, $0\le t<T$,  the backward radial Loewner maps and hulls driven by $\lambda$, and call $\til f_t$, $0\le t<T$, the backward covering radial Loewner maps driven by $\lambda$.

By introducing $f_{t_2,t_1}$ in the radial setting, we find that Lemma \ref{ft2t1} holds if the word ``chordal'' is replaced by ``radial''. The following lemma is similar to Lemma \ref{chordal-S}.

\begin{Lemma}
  For $0\le t<T$, $L_t$ is a $\D$-hull with $\dcap(L_t)=t$. If $t\in(0,T)$, then $f_t=f_{L_t}$ with the same domain: $\ha\C\sem S_{L_t}$, and $e^{i\lambda(t)}\in B_{L_t}$. \label{radial-S}
\end{Lemma}

We find that Lemma \ref{backward-forward-chordal} holds here if the word ``chordal'' is replaced by ``radial''. So we may define backward radial traces $\beta_t$, $0\le t<T$, in a similar manner.

The following lemma holds only in the radial case.

\begin{Lemma}
If $T=\infty$, then $\TT\sem \bigcup_{0<t<\infty} S_{L_t}$ contains at most one point.\label{simple-trace-lem-r}
\end{Lemma}
\begin{proof}
Let $S_\infty=\bigcup_{0<t<\infty} S_{L_t}$. From Koebe's $1/4$ theorem, as $t\to\infty$, $\dist(0,L_t)\to 0$, which implies that the harmonic measure of $\TT\sem B_{L_t}$ in $\D\sem L_t$ seen from $0$ tends to $0$. Since $f_t:\D\conf \D\sem L_t$, $f_t(0)=0$, and $f_t(\TT\sem S_{L_t})=\TT\sem B_{L_t}$, the above harmonic measure at time $t$ equals to $|\TT\sem S_{L_t}|/(2\pi)$. Thus, $|\TT\sem S_\infty|=\lim_{t\to\infty}|\TT\sem S_{L_t}|= 0$.
\end{proof}

For $\kappa>0$, the backward chordal (resp.\ radial) SLE$_\kappa$ is defined by solving backward chordal (resp.\ radial) Loewner equation with $\lambda(t)=\sqrt\kappa B(t)$, $0\le t<\infty$. Since for any fixed $t_0>0$, $(\lambda(t_0-t)-\lambda(t_0),0\le t\le t_0)$ has the distribution of $(\sqrt\kappa B(t),0\le t\le t_0)$, using the existence of forward chordal (resp.\ radial) SLE$_\kappa$ traces, we conclude that $\lambda$ a.s.\ generates a family of backward chordal (resp.\ radial) traces.

\subsection{Normalized global backward trace} \label{normalized}

First we consider a backward chordal Loewner process generated by $\lambda(t)$, $0\le t<T$. Let $S_t=S_{L_t}$, $0\le t<T$, and $S_T=\bigcup_{0\le t<T} S_t$. Then $(S_t)$ is an increasing family, and $S_T$ is an interval.
The following Lemma is similar in spirit to Proposition 5.1 in \cite{SS-GFF}.

\begin{Lemma}
  There exists a family of conformal maps $F_{T,t}$, $0\le t<T$, on $\HH$ such that $F_{T,t_1}=F_{T, t_2}\circ f_{t_2,t_1}$ in $\HH$ if $0\le t_1\le t_2<T$. Let $D_t=F_{T,t}(\HH)$, $0\le t<T$, and $D_T=\bigcup_{t<T} D_t$. If $(\ha F_{T,t})$ satisfies the same property as $(F_{T,t})$, then there is a conformal map $h_T$ defined on $D_T$ such that $\ha F_{T,t}=h_T\circ F_{T,t}$, $0\le t<T$.
  If there is $z_0\in\HH$ such that
  \BGE \lim_{t\to T} \frac{\Imm f_t(z_0)}{|f_t'(z_0)|}=\infty, \label{z0}\EDE
  then we may construct $(F_{T,t})$ such that $D_T=\C$, and we have $S_T=\R$.
  \label{limit-T*}
\end{Lemma}
 \begin{proof} Fix $z_0\in \HH$. Let $z_{t}=f_{t}(z_0)$ and $u_t=f_t'(z_0)$, $0\le t<T$. For $t\in[0,T)$, let $M_t(z)=\frac{z-z_t}{u_t}$ and $F_t=M_t\circ f_t$. Then $F_t$ maps $z_0$ to $0$ and has derivative $1$ at $z_0$. For $0\le t_1\le t_2<T$, define $F_{t_2,t_1}=M_{t_2}\circ f_{t_2,t_1}$. Then $F_{t_2,t_1}\circ f_{t_1,t_0}=F_{t_2,t_0}$ if $t_0\le t_1\le t_2$. Setting $t_0=0$ we get $F_{t_2,t_1}\circ f_{t_1}=F_{t_2}$. Thus, $F_{t_2,t_1}$ is a conformal map on $\HH$ with $F_{t_2,t_1}(z_{t_1})=0$ and $F_{t_2,t_1}'(z_{t_1})=1/{u_{t_1}}$. By Koebe's distortion theorem, for any $t_1\in[0,T)$, $\{F_{t_2,t_1}:t_2\in[t_1,T)\}$ is uniformly bounded on each compact subset of $\HH$. This implies that every sequence in this family contains a subsequence which converges locally uniformly, and the limit function is also conformal on $\HH$, maps $z_{t_1}$ to $0$, and has derivative $1/{u_{t_1}}$ at $z_{t_1}$.

From a diagonal argument, we can find a sequence $(t_n)$ in $[0,T)$ such that $t_n\to T$ and for any $q\in\Q\cap [0,T)$, $(F_{t_n,q})$ converges locally uniformly on $\HH$. Let $F_{T,q}$, $q\in\Q\cap[0,T)$, denote the limit functions, which are conformal on $\HH$. Since $F_{t_n,q_2}\circ f_{q_2,q_1}=F_{t_n,q_1}$ for each $n$, we have $F_{T,q_2}\circ f_{q_2,q_1}=F_{T,q_1}$.
For $t\in[0,T)$, choose $q\in\Q\cap[t,T)$ and define the conformal map $F_{T,t}=F_{T,q}\circ f_{q,t}$ on $\HH$. If $q_1\le q_2\in \Q\cap[t,T)$, then $F_{T,q_1}\circ f_{q_1,t}=F_{T,q_2}\circ f_{q_2,q_1}\circ f_{q_1,t}=F_{T,q_2}\circ f_{q_2,t}$.
Thus, the definition of $F_{T,t}$ does not depend on the choice of $q$.
If $0\le t_1\le t_2<T$, by choosing  $q\in \Q\cap[0,T)$ with $q\ge t_1\vee t_2$, we get
$F_{T,t_2}\circ f_{t_2,t_1}=F_{T,q}\circ f_{q,t_2}\circ f_{t_2,t_1}=F_{T,q}\circ f_{q,t_1}=F_{T,t_1}$.

If (\ref{z0}) holds, then we start the construction of $(F_{T,t})$ with such $z_0$. Since $F_{T,t}:(\HH;z_t)\conf (D_t;0)$ and $F_{T,t}'(z_t)=1/u_t$,  Koebe's $1/4$ theorem implies that $\dist(0,\pa D_t)\ge \frac 14{\Imm z_t}/{|u_t|}=\frac 14\frac{\Imm f_t(z_0)}{|f_t'(z_0)|}$, which tends to $\infty$ as $t\to T$. So $D_T$ has to be $\C$.

Suppose $\ha F_{T,t}$, $0\le t<T$, satisfies the same property as $F_{T,t}$, $0\le t<T$. Let $h_t=\ha F_{T,t}\circ F_{T,t}^{-1}$, $0\le t<T$. Then each $h_t$ is a conformal map defined on $D_t$. If $0\le t_1<t_2<T$, then 
$$h_{t_1}\circ F_{T,t_1}=\ha F_{T,t_1}(z)=\ha F_{T,t_2}\circ f_{t_2,t_1}=h_{t_2}\circ F_{T,t_2}\circ f_{t_2,t_1}=h_{t_2}\circ F_{T,t_1}$$
in $\HH$, which implies that $h_{t_1}=h_{t_2}|_{D_{t_1}}$. So we may define a conformal map $h_T$ on $D_T$ such that $h_t=h_T|_{D_t}$ for $0\le t<T$. Such $h_T$ is what we need.

Suppose that (\ref{z0}) holds but $S_T\ne \R$. Since $S_T$ is an interval, $\lin{S_T}\ne\R$. Choose $\ha z_0\in\R\sem \lin{S_T}$, and start the construction with $\ha z_0$ in place of $z_0$ at the beginning of this proof. Let $\ha F_{T,t}$, $0\le t<T$, denote the family of maps constructed in this way. Then each $\ha F_{T,t}$ is an $\R$-symmetric conformal map, which implies that $\ha D_T\subset\HH$. However, now $D_T=\C$ and $h_T:D_T\conf \ha D_T$, which is impossible. Thus, $S_T= \R$ when (\ref{z0}) holds.
\end{proof}

Let $(F_{T,t})$, $D_t$, and $D_T$ be as in Lemma \ref{limit-T*}. Let $F_T=F_{T,0}$. Suppose $\lambda$ generates backward chordal traces $\beta_{t_0}$, $0\le t_0<T$, which satisfy
\BGE \forall t_0\in[0,T),\quad \exists t_1\in(t_0,T),\quad \beta_{t_1}([0,t_0])\subset\HH.\label{beta>0}\EDE
We may define $\beta(t)$, $0\le t<T$, as follows. For every $t\in[0,T)$, pick $t_0\in(t,T)$ such that $\beta_{t_0}(t)\in\HH$, which is possible by (\ref{beta>0}), and define
\BGE \beta(t)=F_{T,t_0}\beta_{t_0}(t)\in  D_{t_0}\subset D_T. \label{beta}\EDE
Since $F_{T,t_1}=F_{T, t_2}\circ f_{t_2,t_1}$ in $\HH$, from (\ref{ft2t1beta}) we see that the definition of $\beta$ does not depend on the choice of $t_0$. Let $t_0\in[0,T)$. From (\ref{beta>0}), there is $t_1>t_0$ such that  $\beta_{t_1}([0,t])\in\HH$. Since $\beta(t)=F_{T,t_0}(\beta_{t_0}(t))$, $0\le t\le t_0$, we see that $\beta$ is continuous on $[0,t_0]$. Thus, $\beta(t)$, $0\le t<T$, is a continuous curve in $D_T$.

Fix any $x\in S_T$. Then $x\in S_{t_0}$ for some $t_0\in(0,T)$. So $f_{t_0}(x)$ lies on the outer boundary of $L_{t_0}$, which implies that $f_{t_0}(x)\in\beta_{t_0}(t)$ for some $t\in[0,t_0]$. From (\ref{beta>0}), there is $t_1\in(t_0,T)$ such that $\beta_{t_1}([0,t_0])\subset \HH$. Then $f_{t_1}(x)=f_{t_1,t_0}(\beta_{t_0}(t))= \beta_{t_1}(t)\in\HH$. From the continuity of $f_{t_1}$ on $\HH\cup\R$, there is a neighborhood $U$ of $x$ in $\HH\cup\R$ such that $f_{t_1}(U)\subset\HH$. This shows that $U\cap\R\subset S_{t_1}\subset S_T$. Since $F_T=F_{T,t_1}\circ f_{t_1}$ in $\HH$, we find that $F_T$ has continuation on $U$. Since $x\in S_T$ is arbitrary, we conclude that $S_T$ is an open interval, and $F_T$ has continuation to $\HH\cup S_T$.

Now we assume that $\lambda$ generates backward chordal traces, and both (\ref{z0}) and (\ref{beta>0}) hold. Then $D_T=\C$, $S_T=\R$, a continuous curve $\beta(t)$, $0\le t<T$, is well defined, and $F_T$ extends continuously to $\HH\cup\R$. Moreover, $F_T$ is unique up to a M\"obius transformation that fixes $\infty$. With some suitable normalization condition, the family $F_{T,t}$ and the curve $\beta$ will be determined by $\lambda$. We will use the following normalization:
\BGE F_T(\lambda(0))=\lambda(0),\quad F_T(\lambda(0)+i)=\lambda(0)+i.\label{norm}\EDE
If (\ref{norm}) holds, we call $\beta$ the normalized global backward chordal trace generated by $\lambda$. From (\ref{norm}) we see that $\beta(0)=\lambda(0)$, and $\beta$ does not pass through $\lambda(0)+i$.

For the radial case, Lemma \ref{limit-T*} still holds with $\HH$ replaced by $\D$, and (\ref{z0}) replaced by $T=\infty$. For the construction, we choose $z_0=0$ and let $F_{t_2,t_1}(z)=e^{t_2} f_{t_2,t_1}(z)$. If $\lambda$ generates backward radial traces $\beta_t$, $0\le t<T$, which satisfy \BGE \forall t_1\in[0,T),\quad \exists t_2\in(t_1,T),\quad \beta_{t_2}(t_1)\in\D,\label{beta>0-r}\EDE
then we may define a continuous curve $\beta(t)$, $0\le t<T$, in $D_T$ using (\ref{beta}). If $T=\infty$, then $D_T=\C$, and such $\beta$ is determined by $\lambda$ up to a M\"obius transformation that fixes $\infty$, which means that  we may define a normalized global backward radial trace once a normalization condition is fixed.


\subsection{Forward and backward Loewner chains}

%
%
In this section, we review a condition on a family of hulls that corresponds to continuously driven (forward) Loewner hulls, and discuss the corresponding condition for backward Loewner chains.

Let $D\subset\ha\C$ be a simply connected domain such that $\ha\C\sem D$ contains more than one point. A relatively closed subset $H$ of $D$ is called a (boundary) hull in $D$ if $D\sem H$ is simply connected. For example, a hull in $\HH$ is an $\HH$-hull iff it is bounded; a hull in $\D$ is a $\D$-hull iff it does not contain $0$. Let $T\in(0,\infty]$. A family of hulls in $D$: $K_t$, $0\le t<T$, is called a Loewner chain in $D$ if
\begin{enumerate}
  \item $K_0=\emptyset$ and $K_{t_1}\subsetneqq K_{t_2}$ whenever $0\le t_1<t_2<T$;
  \item for any fixed $a\in[0,T)$ and a compact set $F\subset D\sem K_a$, the conjugate extremal distance (c.f.\ \cite{Ahl}) between $K_{t+\eps}\sem K_t$ and $F$ in $D\sem K_t$ tends to $0$ as $\eps\to 0$, uniformly in $t\in[0,a]$.
\end{enumerate}
If $K_t$, $0\le t<T$, is a Loewner chain in $D$, and $a\in[0,T)$, then we also call the restriction $K_t$, $0\le t\le a$, a Loewner chain in $D$.

There are two important properties for Loewner chains. If $K_t$, $0\le t<T$, is a Loewner chain in $D$, and $u$ is a continuous  increasing function defined on $[0,T)$ with $u(0)=0$, then $K_{u^{-1}(t)}$, $0\le t<u(T)$, is also a Loewner chain in $D$, which is called a time-change of $(K_t)$ via $u$. If $W$ maps $D$ conformally onto $E$, then $W(K_t)$, $0\le t<T$, is a Loewner chain in $E$.

An $\HH$-(resp.\ $\D$-)Loewner chain is a Loewner chain in $\HH$ (resp.\ $\D$) such that each hull is an $\HH$-(resp.\ $\D$-)hull.
An $\HH$-(resp.\ $\D$-)Loewner chain $(K_t)$ is said to be normalized if $\hcap(K_t)=2t$ (resp.\ $\dcap(K_t)=t$) for each $t$.

The conditions for the conformal invariance property of  $\HH$-(resp.\ $\D$-)Loewner chains can be slightly weakened as below.
\begin{Proposition}
  If $K_t$, $0\le t<T$, is an $\HH$-(resp.\ $\D$-)Loewner chain, and $W$ is an $\R$-(resp.\ $\TT$-)symmetric conformal map, whose domain contains $\ha K_t$ for each $t$ and whose image does not contain $\infty$ (resp.\ $0$), then $W(K_t)$, $0\le t<T$, is also an $\HH$-(resp.\ $\D$-)Loewner chain. \label{conformal-Loewner}
\end{Proposition}

The following proposition combines some results in \cite{LSW1}  and \cite{Pom}.
\begin{Proposition}
Let $T\in(0,\infty]$.  The following are equivalent.
\begin{enumerate}
    \item [(i)] $K_t$, $0\le t<T$, are chordal (resp.\ radial) Loewner hulls driven by some $\lambda\in C([0,T))$.
    \item [(ii)] $K_t$, $0\le t<T$, is a normalized $\HH$-(resp.\ $\D$-)Loewner chain.
\end{enumerate}
If either of the above holds, with $\mathring\lambda(t)=\lambda(t)$ (resp.\ $\mathring\lambda(t)=e^{i\lambda(t)}$ in the radial case) we have
$$ \{\mathring\lambda(t)\}=\bigcap_{\eps>0}\overline{K_{t+\eps}/K_t}, \quad 0\le t<T. $$ 
In addition, if $K_t$, $0\le t<T$, is any $\HH$-(resp.\ $\D$-)Loewner chain, then the function $u(t):=\hcap(K_t)/2$ (resp.\ $u(t):=\dcap(K_t)$), $0\le t<T$, is continuous increasing with $u(0)=0$, which implies that $K_{u^{-1}(t)}$, $0\le t<u(T)$, is a normalized $\HH$-(resp.\ $\D$-)Loewner chain. \label{Loewner-eqn-chain}
\end{Proposition}


%
%
\begin{Definition}
A family of $\HH$-(resp.\ $\D$-)hulls: $L_{t}$, $0\le t <T$, is called a backward $\HH$-(resp.\ $\D$-)Loewner chain if they satisfy
\begin{enumerate}
  \item $L_0=\emptyset$ and $L_{t_1}\prec L_{t_2}$ if $0\le t_1\le t_2 <T$;
  \item $L_{t_0}:L_{t_0-t}$, $0\le t\le t_0$, is an $\HH$-(resp.\ $\D$-)Loewner chain for any $t_0\in(0,T)$.
\end{enumerate} \label{Loewner-chain-backward}
\end{Definition}

If $u$ is a continuous increasing function defined on $[0,T)$ with $u(0)=0$, then $L_{u^{-1}(t)}$, $0\le t <u(T)$, is also a backward $\HH$-(resp.\ $\D$-)Loewner chain, and is called a time-change of $(L_{t})$ via $u$.
A backward $\HH$-(resp.\ $\D$-)Loewner chain $(L_{t})$ is said to be normalized if $\hcap(L_{t})=2t$ (resp.\ $\dcap(L_{t})=t$) for any $t\in[0,T)$.

Using  Lemma \ref{backward-forward-chordal} and Proposition \ref{Loewner-eqn-chain}, we obtain the following.
\begin{Proposition}
  Let $T\in(0,\infty]$. The following are equivalent.
  \begin{enumerate}
    \item[(i)] $L_{t}$, $0\le t<T$, are backward chordal (resp.\ radial) Loewner hulls driven by some $\lambda\in C([0,T))$.
    \item[(ii)] $L_{t}$, $0\le t <T$, is a normalized backward $\HH$-(resp.\ $\D$-)Loewner chain.
  \end{enumerate}
If either of the above holds, with $\mathring\lambda(t)=\lambda(t)$ (resp.\ $\mathring\lambda(t)=e^{i\lambda(t)}$ in the radial case) we have
\BGE \{\mathring\lambda(t)\}=\bigcap_{\eps>0}\overline{L_{t}:L_{t-\eps}}, \quad 0< t<T,\label{backward-hull-driving}\EDE
In addition, if $L_t$, $0\le t<T$, is any backward $\HH$-(resp.\ $\D$-)Loewner chain, then the function $u(t):=\hcap(K_t)/2$ (resp.\ $u(t):=\dcap(K_t)$), $0\le t<T$, is continuous increasing with $u(0)=0$, which implies that $L_{u^{-1}(t)}$, $0\le t<u(T)$, is a normalized backward $\HH$-(resp.\ $\D$-)Loewner chain. \label{backward-eqn-chain}
\end{Proposition}

We say that $f_t$ and $L_t$, $0\le t<T$, are backward chordal (resp.\ radial) Loewner maps and hulls, via a time-change $u$, driven by $\lambda$, if $u$ is a continuous increasing function defined on $[0,T)$ with $u(0)=0$, such that $f_{u^{-1}(t)}$ and $L_{u^{-1}(t)}$, $0\le t<u(T)$, are backward chordal (resp.\ radial) Loewner maps and hulls driven by $\lambda\circ u^{-1}$. From the above proposition, if $(L_t)$ is any $\HH$-(resp.\ $\D$-)Loewner chain, then $L_t$, $0\le t<T$, are backward chordal (resp.\ radial) Loewner hulls, via a time-change $u(t):=\hcap(L_t)/2$ (resp.\ $\dcap(L_t)$), driven by $\lambda$, which satisfies (\ref{backward-hull-driving}).

\subsection{Simple curves and weldings} \label{welding}
An $\HH$-simple (resp.\ $\D$-simple) curve is a half-open simple curve in $\HH$ (resp.\ $\D\sem\{0\}$), whose open side approaches a single point on $\R$ (resp.\ $\TT$). Every $\HH$(resp.\ $\D$)-simple curve $\beta$ is an $\HH$(resp.\ $\D$)-hull, whose base $B_\beta$ is a single point, and whose support $S_\beta$ is an $\R$(resp.\ $\TT$-)interval. Here an $\TT$-interval is an arc on $\TT$. The function $f_\beta$ extends continuously from $\HH$ (resp.\ $\D$) to $\lin\HH$ (resp.\ $\lin\D$), which maps $S_\beta$ onto $\lin\beta$, sends the two ends of $S_\beta$ to $B_\beta$, and sends a unique point, say $z_\beta\in S_\beta$ to the tip of $\beta$. The point $z_\beta$ divides $S_\beta$ into two $\R$(resp.\ $\TT$-)intervals such that the restriction of $f_\beta$ to either interval is a homeomorphism onto $\lin\beta$. Thus, there is a unique involution (an auto homeomorphism whose inverse is itself) $\phi_\beta$ of $S_\beta$, which fixes only one point: $z_\beta$, swaps the two end points of $S_\beta$, and satisfies that $y=\phi_\beta(x)$ implies that $f_\beta(x)=f_\beta(y)$. We call $\phi_\beta$ the welding induced by $\beta$.

Suppose $K$ is an $\HH$- or $\D$-simple curve. Let $W$ be as in Theorems \ref{W^K}, \ref{Mob}, or \ref{W^K-d}. Then
$W^*(K)$ is also an $\HH$- or $\D$-simple curve. The equality $W^K\circ f_K= f_{W^*(K)}\circ W$ holds after continuous extension from $\HH$ or $\D$ to its closure. So the weldings induced by $K$ and $W^*(K)$ satisfy $\phi_{W^*(K)}=W\circ \phi_K\circ W^{-1}$.

Suppose the hulls $(L_t)$ generated by a backward chordal (resp.\ radial) Loewner process driven by $\lambda$ are all $\HH$(resp.\ $\D$)-simple curves. Then the process generates backward chordal (resp.\ radial) traces $(\beta_t)$ such that every $\beta_t$ is a simple curve, and $L_{t}=\beta_{t}([0,t))$, $0\le t<T$.
Let $\phi_{t}$ be the welding induced by $L_t$, which is an involution of $S_t:=S_{L_t}$. Recall that $(S_t)$ is an increasing family because $L_{t_1}\prec L_{t_2}$ for $t_1<t_2$. If $0\le t_1<t_2<T$, then from $f_{t_2,t_1}\circ f_{t_1}=f_{t_2}$ we see that $\phi_{t_2}|_{S_{t_1}}=\phi_{t_1}$. Thus, there is a unique involution $\phi$ of $S_T:=\bigcup_t S_{t}$ such that $\phi|_{S_t}=\phi_t$ for each $t\in[0,T)$. In other words, $y=\phi(x)$ implies that $f_t(x)=f_t(y)$ for some $t\ge 0$, where $f_t$ is the continuous extension of the Loewner map at time $t$ from $\HH$(resp.\ $\D$) to $\lin\HH$ (resp.\ $\lin\D$). We say that $\phi$ is the welding induced by this process. In the case that $S_T=\R$ (resp. $\TT\sem\{z_0\}$ for some $z_0\in\TT$), we will extend $\phi$ to an involution of $\ha\R:=\R\cup\{\infty\}$ (resp. $\TT$) such that $\infty$ (resp.\ $z_0$) is the other fixed point of $\phi$.

Here is another way to view the welding $\phi$. For every $t\in(0,T)$, $\phi$ swaps the two end points of $S_t$. Let $\mathring \lambda(0)=\lambda(0)$ (resp.\ $e^{i\lambda(0)}$).
Since $f_t(\mathring\lambda(0))=\beta_t(0)$ is the tip of $L_t$ for each $t$, we see that $\mathring\lambda(0)$ is the only fixed point of $\phi$. On the other hand, it is easy to see that, $x$ and $y$ are end points of $S_t$ if and only if $\tau_x=\tau_y=t$, $0<t<T$; and every point on $S_T\sem \{\mathring\lambda(0)\}$ is an end point of some $S_t$, $0<t<T$. Thus, for $x\ne y\in S_T\sem \{\mathring\lambda(0)\}$, $y=\phi(x)$ if and only if $\tau_x=\tau_y$, i.e., $x$ and $y$ are swallowed at the same time.


 Let $\kappa\in(0,4]$. Since the backward chordal (resp.\ radial) SLE$_\kappa$ traces are $\HH$(resp.\ $\D$-)simple curves, so the process induces a random welding, which we call a backward chordal (resp.\ radial) SLE$_\kappa$ welding. In the chordal case, For any $x\in\R\sem\{\lambda(0)\}=\R\sem\{0\}$, the process $X^x_t:=\lambda(t)-f_t(x)$ is a rescaled Bessel process of dimension $1-\frac 4\kappa<1$, which implies that a.s.\ $X^x_t\to 0$ at some finite time. Thus, $S_\infty=\R$. which implies that a chordal SLE$_\kappa$ welding is an involution of $\ha\R$ with two fixed points: $\lambda(0)=0$ and $\infty$. In the radial case, since $T=\infty$, Lemma \ref{simple-trace-lem-r} says that $S_\infty=\TT$ or $\TT\sem\{z_0\}$ for some $z_0\in\TT$. The first case can not happen since $\phi$ has only one fixed point on $S_\infty$. Thus, a radial SLE$_\kappa$ welding is an involution of $\TT$ with two fixed points, one of which is $e^{i\lambda(0)}=1$.

Suppose a backward chordal (resp.\ radial) Loewner process generates $\HH$(resp.\ $\D$)-simple backward traces $\beta_{t}$, $0\le t<T$. Then (\ref{beta>0}) (resp.\ (\ref{beta>0-r})) is satisfied because $\beta_{t_2}(t_1)$ lies in $\HH$ (resp.\ $\D$) if $t_2>t_1$. It is clear that the curve $\beta$ defined by (\ref{beta}) is simple, and $D_t=D_T\sem \beta([t,T))$ for $0\le t<T$. Let $\phi$ be the welding induced by the process. If $y=\phi(x)$, there is $t\in [0,T)$ such that $y,x\in S_t$ and $f_t(y)=f_t(x)$. From $F_{T,t}\circ f_t=F_T$, we get $F_T(y)=F_T(x)$. This means that $\phi$ can be realized by the conformal map $F_T$.

If a backward chordal (resp.\ radial) Loewner chain $(L_t)$ is composed of $\HH$(resp.\ $\D$)-simple curves, then $(L_t)$ induces a welding $\phi$, which is an involution of $\bigcup S_{L_t}$, and agrees with $\phi_{L_t}$ on $S_{L_t}$ for each $t$. To see this, one may first normalized the backward Loewner chain so that it is generated by a backward Loewner process.

\section{Conformal Transformations}\label{Conformal}

\begin{Proposition}
   Suppose $L_{t}$, $0\le t<T$, is a backward $\HH$-(resp.\ $\D$-)Loewner chain, $W$ is an $\R$-(resp.\ $\TT$-)symmetric conformal map whose domain contains every $S_{L_{t}}$, and $\infty\not\in W(S_{L_t})$ for $0\le t<T$. Then  $W^*(L_t)$, $0\le t <T$, is also a backward $\HH$-(resp.\ $\D$-)Loewner chain.\label{conformal-Loewner-backward}
\end{Proposition}
\begin{proof}
From Theorem \ref{W^K}, $W^{L_t}$ and $W^*(L_t)$ are well defined.
Since $L_0=\emptyset$, $W^*(L_0)=\emptyset$. Let $0\le t_1\le t_2<T$. Since $L_{t_1}\prec L_{t_2}$,
from Lemma \ref{W*K/K}, $W^*(L_{t_1})\prec W^*(L_{t_2})$.
Fix $t_0\in(0,T)$. Since $L_{t_0}:L_{t_0-t}$, $0\le t\le t_0$, is an $\HH$-(resp.\ $\D$-)Loewner chain, from Lemma \ref{W*K/K} and Proposition \ref{conformal-Loewner} we see that
$$W^*(L_{t_0}):W^*(L_{t_0-t})=W^{L_{t_0}}(L_{t_0}:L_{t_0-t}),\quad 0\le t\le t_0,$$
is also an $\HH$-(resp.\ $\D$-)Loewner chain. This finishes the proof.
\end{proof}
Using Lemma \ref{W*K/K-Mob} instead of Lemma \ref{W*K/K}, we can show that a similar proposition holds.
\begin{Proposition} Suppose $L_{t}$, $0\le t<T$, is a backward $\HH$-(resp.\ $\D$-)Loewner chain, $W$ is a Mobius transform that maps $\HH$ onto $\D$ (resp.\ maps $\D$ onto $\HH$) such that $\infty\not\in W(S_{L_t})$ for $0\le t<T$. Then  $W^*(L_t)$, $0\le t <T$, is a backward $\D$-(resp.\ $\HH$-)Loewner chain. \label{conformal-Loewner-backward-Mob}
\end{Proposition}

Suppose $(L_t)$ is composed of $\HH$- or $\D$-simple curves. Then $(W^*(L_t))$ is also composed of $\HH$ or $\D$-simple curves.
Let $\phi$ and $\phi_W$ be the weldings induced by these two chains, which are involutions of $\bigcup S_{L_t}$ and $\bigcup S_{W^*(L_t)}$, respectively. Since for each $t\in(0,T)$, $\phi|_{S_{L_t}}=\phi_{L_t}$, $\phi_W|_{S_{W^*(L_t)}}=\phi_{W^*(L_t)}$, $S_{W^*(L_t)}=W(S_{L_t})$, and $\phi_{W^*(L_t)}=W\circ \phi_{L_t}\circ W^{-1}$, we see that $\bigcup S_{W^*(L_t)}=W(\bigcup S_{L_t})$ and
\BGE \phi_W=W\circ \phi\circ W^{-1}.\label{welding-trans}\EDE
This means that the conformal transformation preserves the welding.

The following proposition is essentially Lemma 2.8 in \cite{LSW1}.
\begin{Proposition}
Let $W$ be an $\R$-symmetric conformal map, whose domain contains $z_0\in\R$, such that $W(z_0)\ne \infty$. Then
$$\lim_{H\to z_0} \frac{\hcap(W(H))}{\hcap(H)}=|W'(z_0)|^2,$$
where $H\to z_0$ means that $\diam (H\cup\{z_0\})\to 0$ with  $H$ being a nonempty $\HH$-hull. \label{chordal-hull}
\end{Proposition}

Using the integral formulas for capacities of $\HH$-hulls and $\D$-hulls, it is not hard to derive the following similar proposition.

\begin{Proposition}
\begin{enumerate}
  \item[(i)] Let $W$ be a conformal map on a $\TT$-symmetric domain $\Omega$, which satisfies $I_\R\circ W=W\circ I_{\TT}$ and $W(\Omega\cap\D)\subset\HH$. Let $z_0\in \Omega\cap\TT$ be such that $W(z_0)\ne\infty$. Then
  $$\lim_{H\to z_0} \frac{\hcap(W(H))}{\dcap(H)}=2|W'(z_0)|^2,$$
  where $H\to z_0$ means that $\diam (H\cup\{z_0\})\to 0$ with $H$ being a nonempty $\D$-hull.
  \item[(ii)] Proposition \ref{chordal-hull} holds with $\R$ replaced by $\TT$, $\hcap$ replaced by $\dcap$, and $H\to z_0$ understood as in (i).
\end{enumerate}\label{radial-hull}
\end{Proposition}

\subsection{Transformations between backward $\HH$-Loewner chains} \label{Conformal-chordal}
Suppose $L_t$ and $f_t$, $0\le t<T$, are backward chordal Loewner hulls and maps driven by $\lambda\in C([0,T))$. From Proposition \ref{backward-eqn-chain}, $(L_t)$ is a backward $\HH$-Loewner chain.
 Let $W$ be an $\R$-symmetric conformal map, whose domain $\Omega$ contains the support of every $L_t$. Write $W_t$ for $W^{L_t}$. The domain of $W_t$ is $\Omega^{L_t}$, which contains $\ha L_t$. If $t>0$, $\lambda(t)\in \ha L_t$, so $\lambda(t)$ is contained in the domain of $W_t$. This is also true for $t=0$ because $W_0=W$ and $\{\lambda(0)\}=S_0\subset S_t=S_{L_t}\subset\Omega$ for any $t\in(0,T)$.
 Let $L^*_t=W^*(L_t)=W_t(L_t)$, $0\le t<T$. From Proposition \ref{conformal-Loewner-backward}, $(L^*_t)$ is a backward $\HH$-Loewner chain, and
 \BGE W_t(L_t:L_{t-\eps})=L^*_t:L^*_{t-\eps},\quad 0\le t-\eps<t<T.\label{W(:)}\EDE
 From Proposition \ref{backward-eqn-chain}, $L^*_{t}$, $0\le t<T$, are backward chordal Loewner hulls via a time-change $u(t):=\hcap(L^*_t)/2$, driven by some $\lambda^*$, which satisfies
$$\{\lambda^*(t)\}=\bigcap_{\eps>0}\overline{L^*_{t}:L^*_{t-\eps}}, \quad 0< t<T. $$
From (\ref{backward-hull-driving}), (\ref{W(:)}), 
and continuity, we find that \BGE \lambda^*(t)=W_t(\lambda(t)),\quad 0\le t<T.\label{W=}\EDE
Since $(L_t)$ and $(L^*_{u^{-1}(t)})$ are normalized, we know that $\hcap(L_t:L_{t-\eps})=2\eps$ and $\hcap(L^*_t:L^*_{t-\eps})=2u(t)-2u(t-\eps)$. From (\ref{W(:)}) and Proposition \ref{chordal-hull}, we find that
\BGE u'(t)=W_t'(\lambda(t))^2,\quad 0\le t<T.\label{u'}\EDE

Let $f^*_t=f_{L^*_t}$.  From the definition of $W_t=W^{L_t}$, we have the equality
\BGE W_t\circ f_t=f^*_t\circ W,\label{circ}\EDE
which holds in $\Omega\sem S_{L_t}$. Differentiating (\ref{circ}) w.r.t.\ $t$, and using (\ref{W=}) and (\ref{u'}), we get
$$\pa_t W_t(f_t(z))+W_t'(f_t(z))\frac{-2}{f_t(z)-\lambda(t)}=\frac{-2u'(t)}{f^*_t(W(z))-\lambda^*(t)}$$
$$=\frac{-2W_t'(\lambda(t))^2}{W_t(f_t(z))-W_t(\lambda(t))}.$$
Thus, for any $w=f_t(z)\in f_t(\Omega\sem S_{L_t})=\Omega^{L_t}\sem \ha{L_t}$,
\BGE\pa_t W_t(w)=\frac{-2W_t'(\lambda(t))^2}{W_t(w)-W_t(\lambda(t))}-W_t'(w)\frac{-2}{w-\lambda(t)}.\label{patWtw}\EDE
By analytic extension, the above equality holds for any $w\in \Omega^{L_t}\sem \{\lambda(t)\}$. Letting $w\to\lambda(t)$, we find that
\BGE \pa_t W_t(\lambda(t))=3W_t''(\lambda(t)),\quad 0\le t<T.\label{-3}\EDE
Differentiating (\ref{patWtw}) w.r.t.\ $w$ and letting $w\to\lambda(t)$, we get
\BGE \frac{\pa_t W_t'(\lambda(t))}{W_t'(\lambda(t))}=-\frac 12\Big(\frac{W_t''(\lambda(t))}{W_t'(\lambda(t))}\Big)^2+\frac 43 \frac{ W_t'''(\lambda(t))}{W_t'(\lambda(t))}.\label{W''/W'-chordal}\EDE

\subsection{Transformations involving backward $\D$-Loewner chains}\label{Conformal-radial}
Now suppose $L_t$, $0\le t<T$, are backward radial Loewner hulls driven by $\lambda$. Let $f_t$ and $\til f_t$ be the corresponding radial Loewner maps and covering maps. Suppose $W$ is a $\TT$-symmetric conformal map, whose domain $\Omega$ contains the support of every $L_t$. Let $W_t=W^{L_t}$, $L^*_t=W_t(L_t)=W^*(L_t)$, and $u(t)=\dcap(L^*_t)$, $0\le t<T$. Then $L^*_{t}$, $0\le t<T$, are backward radial Loewner hulls via a time-change $u(t):=\dcap(L^*_t)$, driven by some $\lambda^*$, which satisfies
$$\{e^{i\lambda^*(t)}\}=\bigcap_{\eps>0} \overline{ L^*_t:L^*_{t-\eps}},\quad 0<t<T.$$
Let $f^*_t$ (resp.\ $\til f^*_t$), $0\le t<T$, denote the backward radial (resp.\ covering radial) Loewner hulls via the time-change $u$ driven by $\lambda^*$. The argument in the last subsection still works with Proposition \ref{radial-hull} in place of Proposition \ref{chordal-hull}. We can conclude that $e^{i\lambda(t)}$ lies in the domain of $W_t$ for $0\le t<T$; $W_t(e^{i\lambda(t)})=e^{i\lambda^*(t)}$; $u'(t)=|W_t'(e^{i\lambda(t)})|^2$; and (\ref{circ}) still holds.
Suppose $\til W$ is an $\R$-symmetric conformal map defined on $\til\Omega=(e^i)^{-1}(\Omega)$, which satisfies $e^i\circ \til W=W\circ e^i$. Define $\til W_t$ to be the analytic extension of $\til f^*_t\circ \til W\circ \til f_t^{-1}$ to $\til \Omega_t:=(e^i)^{-1}(\Omega^{L_t})$. Then we get
\BGE \til W_t\circ \til f_t=\til f^*_t\circ \til W;\label{circ-til-til}\EDE
Comparing this with (\ref{circ}) we find $e^i\circ \til W_t=W_t\circ e^i$. So $\lambda(t)$ lies in the domain of $\til W_t$, and
\BGE u'(t)=\til W_t'(\lambda(t))^2,\quad 0\le t<T.\label{u'-til}\EDE
Since $W_t(e^{i\lambda(t)})=e^{i\lambda^*(t)}$, from the continuity, there is $n\in\N$ such that $\til W_t(\lambda(t))=\lambda^*(t)+2n\pi$ for $0\le t<T$. Since $\lambda^*$ and $\lambda^*+2n\pi$ generate the same backward radial Loewner objects via the time-change $u$, by replacing $\lambda^*$ with $\lambda^*+2n\pi$, we may assume that
\BGE \til W_t({\lambda(t)})={\lambda^*(t)},\quad 0\le t<T.\label{W=-til}\EDE
Differentiating (\ref{circ-til-til}) w.r.t.\ $t$ and letting $w=\til f_t(z)$, we get
\BGE\pa_t \til W_t(w)=-\til W_t'(\lambda(t))^2\cot_2({\til W_t(w)-\til W_t(\lambda(t))})+\til W_t'(w)\cot_2({w-\lambda(t)}),\label{patWtw-radial}\EDE
which holds for $w\in (e^i)^{-1}(\Omega^{L_t}\sem\{e^{i\lambda(t)}\})$. Letting $w\to\lambda(t)$, we get
\BGE \pa_t \til W_t(\lambda(t))=3\til W_t''(\lambda(t)),\quad 0\le t<T.\label{-3-til}\EDE
Differentiating (\ref{patWtw-radial}) w.r.t.\ $w$ and letting $w\to\lambda(t)$, we get
\BGE \frac{\pa_t \til W_t'(\lambda(t))}{\til W_t'(\lambda(t))}=-\frac 12\Big(\frac{W_t''(\lambda(t))}{W_t'(\lambda(t))}\Big)^2+\frac 43 \frac{ W_t'''(\lambda(t))}{W_t'(\lambda(t))}+\frac 16\til W_t'(\lambda(t))^2-\frac 16.\label{W''/W'-radial}\EDE
The number $\frac 16$ comes from the Laurent series of $\cot_2(z)$: $\frac 2z-\frac z6+O(z^3)$.

Let $(L_t)$, $(f_t)$, and $(\til f_t)$ be as above. Now suppose $W$ is a M\"obius transformation that maps $\D$ onto $\HH$ such that $W^{-1}(\infty)\not\in S_{L_t}$ for every $t$. Let $W^{L_t}$ be as in Theorem \ref{Mob}. Let $W_t=W^{L_t}$ and $L^*_t=W_t(L_t)=W^*(L_t)$, $0\le t<T$. Then $L^*_{t}$, $0\le t<T$, are backward chordal Loewner hulls via a time-change $u(t):=\hcap(L^*_t)/2$, driven by some $\lambda^*$. Let $f^*_t=f_{L^*_t}$. Then (\ref{circ}) still holds, and we have $u'(t)=|W_t'(e^{i\lambda(t)})|^2$ and $W_t(e^{i\lambda(t)})=e^{i\lambda^*(t)}$. Let $\til W=W\circ e^i$ and $\til W_t=W_t\circ e^i$. We get (\ref{u'-til}), (\ref{W=-til}), and
$\til W_t\circ \til f_t=f^*_t\circ \til W$. 
Differentiating this equality w.r.t.\ $t$ and letting $w=\til f_t(z)$ tend to $\lambda(t)$,  we find that (\ref{-3-til}) still holds.

\subsection{Conformal invariance of backward SLE$(\kappa;\rho)$ processes} \label{section-kappa-rho}
 We now define backward chordal and radial SLE$(\kappa;\vec\rho)$ processes, where $\vec\rho=(\rho_1,\dots,\rho_n)\in\R^n$. Let $x_0,q_1,\dots,q_n\in\R$ such that $q_k\ne x_0$ for all $k$. Let $\lambda(t)$, $0\le t<T$, be the maximal solution of the equation
 \BGE d\lambda(t)=\sqrt\kappa dB(t)+\sum_{k=1}^n \frac{-\rho_k}{\lambda(t)-f^\lambda_t(q_k)}\,dt;\quad \lambda(0)=x_0.\label{kappa-rho-chordal}\EDE
 Here $f^\lambda_t$, $0\le t<T$, are the backward chordal Loewner maps driven by $\lambda$. Then we call the backward chordal Loewner process driven by $\lambda$ the chordal SLE$(\kappa;\vec\rho)$ process started from $x_0$ with force points $(q_1,\dots,q_n)$, or simply started from $(x_0;q_1,\dots,q_n)$. We allow some $q_k$ to be $\infty$. In that case, $f^\lambda_t(q_k)$ is always $\infty$, and the term $\frac{-\rho_k}{\lambda(t)-f^\lambda_t(q_k)}$ vanishes.

Let  $x_0,q_1,\dots,q_n\in\R$ be such that $q_k\not\in x_0+2\pi\Z$ for all $k$.  Let $\lambda(t)$, $0\le t<T$, be the maximal solution of the equation
 \BGE d\lambda(t)=\sqrt\kappa dB(t)+\sum_{k=1}^n \frac{-\rho_k}2 \cot_2({\lambda(t)-\til f^\lambda_t(q_k)})\,dt;\quad \lambda(0)=x_0.\label{kappa-rho-radial}\EDE
 Here $\til f^\lambda_t$, $0\le t<T$, are the covering backward radial Loewner maps driven by $\lambda$. Then the backward radial Loewner process driven by $\lambda$ is called the radial SLE$(\kappa;\vec\rho)$ process started from $e^{ix_0}$ with marked points $(e^{iq_1},\dots,e^{iq_n})$, or simply started from $(e^{ix_0};e^{iq_1},\dots,e^{iq_n})$.

The existence of backward chordal (resp.\ radial) SLE$_\kappa$ traces and Girsanov's Theorem imply the existence of a backward chordal (resp.\ radial) SLE$(\kappa;\vec\rho)$ traces. The traces are $\HH$(resp.\ $\D$)-simple curves if $\kappa\in(0,4]$.


 The following lemma is easy to check.
\begin{Lemma}
  Let $W$ be a M\"obius transformation. Then the following hold.
  \begin{enumerate}
    \item[(i)] For any $z\in\C\cap W^{-1}(\C)$ and $w\in\ha\C$,
    $$\frac{2W'(z)}{W(z)-W(w)}-\frac{2}{z-w}=\frac{W''(z)}{W'(z)}.$$
    \item[(ii)] Let $\til W=W\circ e^i$. For any $z\in\C\cap \til W^{-1}(\C)$ and $w\in\C$,
    $$\frac{2\til W'(z)}{\til W(z)-\til W(w)}- \cot_2(z-w)= \frac{\til W''(z)}{\til W'(z)}.$$
    \item[(iii)] Suppose an analytic function $\til W:\Omega\to\C$ satisfies $e^i\circ \til W=W\circ e^i$ in $\Omega$. Then for any $z,w\in\Omega$,
        $$\til W'(z)\cot_2(\til W(z)-\til W(w))-\cot_2(z-w)= \frac{\til W''(z)}{\til W'(z)}.$$
  \end{enumerate}\label{Mob-W}
\end{Lemma}

  \begin{Theorem}
   Let $L_t$, $0\le t<T$, be the backward chordal SLE$(\kappa;\vec\rho)$ hulls started from $(x_0;q_1,\dots,q_n)$. Suppose $\sum\rho_k=-\kappa-6$. Let $W$ be a M\"obius transformation from $\HH$ onto $\HH$ such that $\{\infty,W^{-1}(\infty)\}\subset\{q_1,\dots,q_n\}$. Then, after a time-change, $W^*(L_t)$, $0\le t<T$, are the backward chordal SLE$(\kappa;\vec\rho)$ hulls started from $(W(x_0);W(q_1),\dots,W(q_n))$. \label{cord-change-1}
 \end{Theorem}
 \begin{proof}
 Since $W^{-1}(\infty)$ is a force point, it is not contained in the support of any $L_t$. So $\infty\not\in W(S_{L_t})$, $0\le t<T$. Let $\lambda$ be the driving function, and $f_t=f^\lambda_t$, $0\le t<T$, be the corresponding maps.
   We may and now adopt the notation in Section \ref{Conformal-chordal}. Let $(\F_t)$ be the complete filtration generated by $B(t)$ in (\ref{kappa-rho-chordal}). Then $(\lambda_t)$ and $(L_t)$ are $(\F_t)$-adapted.
   From Corollary \ref{continuous-Cor} (i), $(W^*(L_t))$ is also $(\F_t)$-adapted. Since $W_t=W^{L_t}=f_{W^*(L_t)}\circ W\circ  g_{L_t}$ on $\Omega^{L_t}\sem \ha {L_t}$, $(W_t)$ is $(\F_t)$-adapted. So we may apply It\^o's formula (c.f.\ \cite{RY}).
   From (\ref{W=}) and (\ref{-3}), we get
   $$d\lambda^*(t)=W_t'(\lambda(t))d\lambda(t)+\Big(\frac\kappa 2+3\Big) W_t''(\lambda(t))dt,\quad 0\le t<T.$$
   Applying (\ref{kappa-rho-chordal}) and Lemma \ref{Mob-W} (i), and using the condition that $\sum\rho_k=-\kappa-6$, we find that
   $$d\lambda^*(t)=W_t'(\lambda(t))\sqrt\kappa dB(t)+\sum_{k=1}^n \frac{-\rho_k W_t'(\lambda(t))^2}{W_t(\lambda(t))-W_t\circ f^\lambda_t(q_k)}\,dt$$
   $$=W_t'(\lambda(t))\sqrt\kappa dB(t)+\sum_{k=1}^n \frac{-\rho_k W_t'(\lambda(t))^2}{\lambda^*(t)- f^{*}_t\circ W(q_k)}\,dt, \quad 0\le t<T.$$
From (\ref{W=}) we get $\lambda^*(0)=W_0(\lambda(0))=W(x_0)$.
Since $L^*_t=W^*(L_t)$ and $f^*_t$ are backward chordal Loewner hulls and maps via the time-change $u$ driven by $\lambda^*$, from (\ref{u'}) and the above equation, we conclude that, after a time-change, $W^*(L_t)$, $0\le t<T$, are the backward chordal SLE$(\kappa;\vec\rho)$ hulls started from $(W(x_0);W(q_1),\dots,W(q_n))$ and stopped at some time.

   It remains to show that the above process is completed. If not, the process can be extended without swallowing the force points $W(q_1),\dots,W(q_n)$. From the condition, $W(\infty)$ is among these force points. So $(W^{-1})^*$ is well defined at the hulls of the extended process. From Propositions \ref{conformal-Loewner-backward} and \ref{backward-eqn-chain}, this implies that the backward chordal Loewner hulls $L_t$, $0\le t<T$, can be extended without swallowing any of $q_1,\dots,q_n$, which is a contradiction.
 \end{proof}

The following theorem can be proved using the above proof with minor modifications: we now use the argument in Section \ref{Conformal-radial} instead of that in Section \ref{Conformal-chordal}, apply Lemma \ref{Mob-W} (ii) and (iii) instead of (i), and use Proposition \ref{conformal-Loewner-backward-Mob} in addition to Proposition \ref{conformal-Loewner-backward}.

\begin{Theorem}
Suppose $\sum\rho_k=-\kappa-6$.
   Let $(L_t)$ be the backward radial SLE$(\kappa;\vec\rho)$ hulls started from $(e^{ix_0};e^{iq_1},\dots,e^{iq_n})$. Let $W$ map $\D$ conformally onto $\HH$ (resp.\ $\D$) such that $\{W^{-1}(\infty)\}\cap\TT\subset \{e^{iq_1},\dots,e^{iq_n}\}$. Then, after a time-change, $(W^*(L_t))$ are the backward chordal (resp.\ radial) SLE$(\kappa;\vec\rho)$ hulls started from $(W(e^{ix_0});W(e^{iq_1}),\dots,W(e^{iq_n}))$.\label{cord-change-2}
 \end{Theorem}

 \begin{Corollary}
   Let $(L_t)$  be the backward radial SLE$(\kappa;-\kappa-6)$ hulls started from $(e^{ix_0};e^{iq_0})$. Let $W$ map $\D$ conformally onto $\HH$ such that $W(e^{ix_0})=0$ and $W(e^{iq_0})=\infty$. Then, after a time-change, $(W^*(L_t))$ are the backward chordal SLE$_\kappa$ hulls started from $0$.\label{chordal-radial}
 \end{Corollary}

\no{\bf Remarks.}
\begin{enumerate}
  \item The above theorems resemble the work in \cite{SW} for forward SLE$(\kappa;\vec\rho)$ processes. The condition in their paper is $\sum \rho_k=\kappa-6$. This is one reason why we may view backward SLE$_\kappa$  as SLE$_{-\kappa}$.
      \item The definition of backward SLE$(\kappa;\vec\rho)$ process differ from Sheffield's definition in \cite{SS-GFF} by a minus sign in (\ref{kappa-rho-chordal}) and (\ref{kappa-rho-radial}) before the $\rho_k$'s. If Sheffield's definition were used, the condition for conformal invariance would be $\sum\rho_k=\kappa+6$ instead of $\sum\rho_k=-\kappa-6$.
  \item We may allow interior force points as in \cite{SW}. For the chordal (resp.\ radial) SLE$(\kappa;\vec\rho)$ process, if $q_k\in\HH$ (resp.\ $e^{iq_k}\in\D$) is a force point, we use $\Ree f^\lambda_t(q_k)$ (resp.\ $\Ree \til f^\lambda_t(q_k)$) instead of $f^{\lambda}_t(q_k)$ (resp.\ $\til f^{\lambda}_t(q_k)$) in (\ref{kappa-rho-chordal}) (resp.\ (\ref{kappa-rho-radial})). In the radial case, adding $0$ to be a force point or change the force for $0$ does not affect the process. Theorems \ref{cord-change-1} and Theorem \ref{cord-change-2} still hold if some or all force points lie inside $\HH$ or $\D$. For the proofs, we apply Lemma \ref{Mob-W} with real parts taken on the displayed formulas. One particular example is the following corollary.
      \begin{Corollary}
        Let $L_t$, $0\le t<\infty$, be a backward radial SLE$_\kappa$ process. Let $W$ be a M\"obius transformation that maps $\D$ onto $\HH$ such that $W(1)\ne \infty$. Let $T$ be the maximum number such that $W^{-1}(\infty)\not\in S_{L_t}$, $0\le t<T$. Then, after a time-change, $W^*(L_t)$, $0\le t<T$, are the backward chordal SLE$(\kappa;-\kappa-6)$ hulls started from $(W(1);W(0))$.
      \end{Corollary}
\item Using the properties of Bessel process and applying Girsanov's theorem, one may define backward chordal or radial SLE$(\kappa;\vec \rho)$ processes with exactly one degenerate force point, if the corresponding force $\rho_1$ satisfies $\rho_1<-2$ (which corresponds to a Bessel or Bessel-like process of dimension $d=1-\frac{4+2\rho_1}\kappa>1$). Theorems \ref{cord-change-1} and \ref{cord-change-2} still hold when a degenerate force point exists. Unlike forward SLE$(\kappa;\vec \rho)$ process, it is impossible to define a backward SLE$(\kappa;\vec \rho)$ process with two different degenerate force points.
\item Consider the radial case with one force point. Suppose the force $\rho_1\le -\frac\kappa 2-2$. Let $X_t=\lambda(t)-\til f^\lambda_t(q_1)$. Then $X_t$ is a Bessel-like process with dimension $d=1-\frac{4+2\rho_1}\kappa\ge 2$, which implies that $X_t$ never hits $2\pi\Z$. So $T=\infty$ and ${e^{iq_1}}\not\in S_t$ for any $t$. From Lemma \ref{simple-trace-lem-r}, $S_\infty=\TT\sem \{e^{iq_1}\}$. If, in addition, $\kappa\in(0,4]$, then a backward radial SLE$(\kappa;\rho_1)$ process induces a random welding $\phi$, which is a involution of $\TT$ with exactly two fixed point, $e^{i\lambda(0)}$ and $e^{iq_1}$, which are the initial point and the force point of the process.
\end{enumerate}

\section{Commutation Relations} \label{comm}
\begin{Definition}
  Let $\kappa_1,\kappa_2>0$,  $n\in\N$, and $\vec\rho_1,\vec\rho_2\in\R^n$. Let $z_1,z_2,w_k$, $2\le k\le n$, be distinct points  on $\R$ (resp.\ $\TT$). We say that a backward chordal (resp.\ radial) SLE$(\kappa_1;\vec\rho_1)$ started from $(z_1;z_2,w_2,\dots,w_n)$ commutes with a backward chordal (resp.\ radial) SLE$(\kappa_2;\vec\rho_2)$ started from $(z_2;z_1,w_2,\dots,w_n)$ if there exists a coupling of two processes $(L_1(t);0\le t<T_1)$ and $(L_2(t);0\le t<T_2)$ such that
  \begin{enumerate}
    \item[(i)] For $j=1,2$,  $(L_j(t),0\le t<T_j)$ is a complete backward chordal (resp.\ radial) SLE$(\kappa_j;\vec\rho_j)$ process started from $(z_j;z_{3-j},w_2,\dots,w_n)$.
    \item[(ii)] For $j\ne k\in\{1,2\}$, if $\bar t_k<T_k$ is a stopping time w.r.t.\ the complete filtration $(\F^k_t)$ generated by $(L_k(t))$, then conditioned on $\F^k_{\bar t_k}$, after a time-change, $f_k({\bar t_k},\cdot)^*(L_j({t_j}))$, $0\le t_j<T_j(\bar t_k)$, has the distribution of a partial backward chordal (resp.\ radial)  SLE$(\kappa_j;\vec\rho_j)$ process started from
        $$(f_k({\bar t_k},(z_j));\mathring \lambda_k(\bar t_k),f_k({\bar t_k},w_2),\dots,f_k({\bar t_k},w_n)),$$ where $f_k({\bar t_k},\cdot):=f_{L_k({\bar t_k})}$, $T_j(\bar t_k):=\sup \{t_j<T_j:S_{L_j({t_j})}\cap S_{L_{k}({\bar t_k})}=\emptyset\}$, $\mathring \lambda_k(\bar t_k)=\lambda_k(\bar t_k)$ in the chordal case (resp.\ $e^{i\lambda_k(\bar t_k)}$ in the radial case), and $\lambda_k$ is the driving function for $(L_k(t))$.
  \end{enumerate}
  Here a partial backward SLE$(\kappa;\vec\rho_j)$ process is a complete SLE$(\kappa;\vec\rho_j)$ process stopped at a positive stopping time. 
  If the commutation holds for any distinct points $z_1,z_2,w_k$, $2\le k\le n$ on $\R$ (resp.\ $\TT$), then we simply say that backward chordal (resp.\ radial) SLE$(\kappa_1;\vec\rho_1)$ commutes with backward chordal (resp.\ radial) SLE$(\kappa_2;\vec\rho_2)$. \label{def-coupl}
\end{Definition}

\begin{Theorem}
For any $\kappa>0$, backward chordal (resp.\ radial) SLE$(\kappa;-\kappa-6)$ commutes with backward chordal (resp.\ radial) SLE$(\kappa;-\kappa-6)$. 
\label{coupl}
\end{Theorem}

We will prove this theorem in the next two subsections.

\subsection{Ensemble}
We first consider the radial case. Fix $\kappa>0$ and $z_1\ne z_2\in\TT$. Write $z_j=e^{i\til z_j}$, $j=1,2$. For $j=1,2$, let $L_j(t)$, $0\le t<T_j$, be a backward radial SLE$(\kappa;-\kappa-6)$ process started from $(z_j;z_{3-j})$; let $\lambda_j$ be the driving function, and let $f_j(t,\cdot)$ and $\til f_j(t,\cdot)$, $0\le t<T_j$, be the corresponding maps and covering maps. At first, we suppose that the two processes are independent. Then for $j=1,2$, $\lambda_j$ satisfies $\lambda_j(0)=\til z_j$ and the SDE:
\BGE d\lambda_j(t)=\sqrt\kappa dB_j(t)-\frac{-\kappa-6}2 \cot_2(\lambda_j(t)-\til f_j(t,\til z_{3-j}))dt,\quad 0\le t<T_j,\label{dlambda}\EDE
where $B_1(t)$ and $B_2(t)$ are independent standard Brownian motions. For $j=1,2$, let $(\F^j_t)$ denote the complete filtration generated by $B_j(t)$.

Define ${\cal D}=\{(t_1,t_2)\in [0,T_1)\times[0,T_2): S_{L_1(t_1)}\cap S_{L_2(t_2)}=\emptyset\}$.
Then for $(t_1,t_2)\in\cal D$, we have $(L_1(t_1),L_2(t_2))\in{\cal P}_*$. So we may define
$$(L_{1,t_2}(t_1),L_{2,t_1}(t_2))=f^*(L_1(t_1),L_2(t_2)).$$
Let $f_{1,t_2}(t_1,\cdot)=f_{L_{1,t_2}(t_1)}$ and $f_{2,t_1}(t_2,\cdot)=f_{L_{2,t_1}(t_2)}$.
From a radial version of Theorem \ref{bijection-gf*}, we see that 
 \BGE f_{1,t_2}(t_1,\cdot)\circ f_2(t_2,\cdot)=f_{L_1(t_1)\vee L_2(t_2)}=f_{2,t_1}(t_2,\cdot)\circ f_1(t_1,\cdot).\label{circ-f}\EDE
Fix $j\ne k\in\{1,2\}$. From a radial version of Corollary \ref{continuous-Cor} (ii), the random map $f_{j,t_k}(t_j,\cdot)$ is $\F^j_{t_j}\times \F^k_{t_k}$-measurable. Let $u_{j,t_k}(t_j)=\dcap(L_{j,t_k}(t_j))$. From Propositions \ref{backward-eqn-chain} and \ref{conformal-Loewner-backward}, for any fixed $t_k\in[0,T_k)$, $f_{j,t_k}(t_j,\cdot)$ are backward radial Loewner maps via the time-change $u_{j,t_k}$. Let $\til f_{j,t_k}(t_j,\cdot)$ be the corresponding covering maps. So $e^i\circ \til f_{j,t_k}(t_j,\cdot)= f_{j,t_k}(t_j,\cdot)\circ e^i$. From continuity, we see that $\til f_{j,t_k}(t_j,\cdot)$ is also $\F^j_{t_j}\times \F^k_{t_k}$-measurable, and from (\ref{circ-f}) we have
\BGE \til f_{1,t_2}(t_1,\cdot)\circ \til f_2(t_2,\cdot)=\til f_{2,t_1}(t_2,\cdot)\circ \til f_1(t_1,\cdot).\label{circ-f-til}\EDE
Define $\mA$ on $\cal D$ by $\mA(t_1,t_2)=\dcap(L_{1}(t_1)\vee L_{2}(t_2))$. From (\ref{circ-f}) we get
\BGE \mA(t_1,t_2)=u_{1,t_2}(t_1)+t_2=u_{2,t_1}(t_2)+t_1.\label{mA}\EDE

Apply the argument in the first paragraph of Section \ref{Conformal-radial} with $\lambda=\lambda_j$, $L_{t_j}=L_j(t_j)$, $W=f_k(t_k,\cdot)$, and $\til W=\til f_k(t_k,\cdot)$, where $t_k\in[0,T_k)$ is fixed. Then we have correspondence: $L^*_{t_j}=L_{j,t_k}(t_j)$, $u=u_{j,t_k}$, and $\til f^*_{t_j}=\til f_{j,t_k}(t_j,\cdot)$.
Since $\til W_{t_j}$ is an analytic extension of $\til f^*_{t_j}\circ \til W\circ \til f_{t_j}^{-1}$, from (\ref{circ-f-til}), we find that
$\til W_{t_j}=\til f_{k,t_j}(t_k,\cdot)$. Thus, $e^{i\lambda_j(t_j)}$ (resp.\ $\lambda_j(t_j)$) lies in the domain of $f_{k,t_j}(t_k,\cdot)$ (resp.\ $\til f_{k,t_j}(t_k,\cdot)$) as long as $(t_1,t_2)\in\cal D$.

Write $\til F_{k,t_k}(t_j,\cdot)=\til f_{k,t_j}(t_k,\cdot)$. We will use $\pa_t$ to denote the partial derivative w.r.t.\ the first variable inside the parentheses, and use $'$ and the superscript $(h)$ to denote the partial derivatives w.r.t.\ the second variable inside the parentheses. For $h=0,1,2,3$, define $A_{j,h}$ on $\cal D$ by
\BGE A_{j,h}(t_1,t_2)=\til f_{k,t_j}^{(h)}(t_k,\lambda_j(t_j))=\til F_{k,t_k}^{(h)}(t_j,\lambda_j(t_j)).\label{Ajh}\EDE
Use the superscript $(S)$ to denote the (partial) Schwarzian derivative. Define $A_{j,S}$ on $\cal D$ by
\BGE A_{j,S}(t_1,t_2)=\til f_{k,t_j}^{(S)}(t_k,\lambda_j(t_j))=\til F_{k,t_k}^{(S)}(t_j,\lambda_j(t_j)) \label{AS}\EDE
From Section \ref{Conformal-radial}, we know that $L_{j,t_k}(t_j)$ are backward radial Loewner hulls via the time-change $u_{j,t_k}$ driven by $\lambda_{j,t_k}$, which can be chosen such that
\BGE \lambda_{j,t_k}(t_j)=A_{j,0}(t_1,t_2).\label{lambda-j-tk}\EDE
Moreover, from (\ref{u'-til}), (\ref{-3-til}), and (\ref{W''/W'-radial}), we have
\BGE u_{j,t_k}'(t_j)
=A_{j,1} ^2,\label{ujtk}\EDE
\BGE \pa_{t} \til F_{k,t_k}(t_j,\lambda_j(t_j))=3A_{j,2},\label{3A}\EDE
\BGE \frac{ \pa_{t} \til F_{k,t_k}'(t_j,\lambda_j(t_j))}{\til F_{k,t_k}'(t_j,\lambda_j(t_j))}= -\frac12\Big(\frac{A_{j,2}}{A_{j,1}}\Big)^2 +\frac43 \frac{A_{j,3}}{A_{j,1}} +\frac 16A_{j,1}^2-\frac 16,\label{W''/W'-ensemble}\EDE
where all $A_{j,h}$ are valued at $(t_1,t_2)$.


From now on, we fix an $(\F^k_t)$-stopping time $t_k$ with $t_k<T_k$. Then the process of conformal maps $(\til F_{k,t_k}(t_j,\cdot))$ is $(\F^j_{t_j}\times \F^k_{t_k})_{t_j\ge 0}$-adapted. Let $T_j(t_k)$ be the maximal number such that for any $t_j<T_j(t_k)$, we have $(t_1,t_2)\in\cal D$. Then $T_j(t_k)$ is an $(\F^j_{t_j}\times \F^k_{t_k})_{t_j\ge 0}$-stopping time. Recall that $(\lambda_j(t))$ is an $(\F^j_{t_j})$-adapted local martingale with $\langle \lambda_j\rangle_t =\kappa t$. From now on, we will apply It\^o's formula repeatedly. All SDEs below are $(\F^j_{t_j}\times \F^k_{t_k})_{t_j\ge 0}$-adapted, and $t_j$ runs in the interval $[0,T_j(t_k))$.

From (\ref{lambda-j-tk}), (\ref{Ajh}), and (\ref{3A}), we get
\BGE d\lambda_{j,t_k}(t_j)=A_{j,1} d\lambda_j(t_j)+\Big(\frac\kappa 2+3\Big) A_{j,2}dt,\quad 0\le t<T_j(t_k).\label{dlambdajtk}\EDE
From (\ref{Ajh}) and (\ref{W''/W'-ensemble}) we get
\BGE \frac{\pa_{t_j} A_{j,h}}{A_{j,h}}=\frac{A_{j,2}}{A_{j,1}}d\lambda_j+\Big[-\frac12\Big(\frac{A_{j,2}}{A_{j,1}}\Big)^2 +\Big(\frac\kappa 2+\frac43\Big)\frac{A_{j,3}}{A_{j,1}}+\frac 16A_{j,1}^2-\frac 16\Big] dt_j.\label{dAjh/Ajh}\EDE

Let
$$\alpha=\frac{6-(-\kappa)}{2(-\kappa)},\quad \cc=\frac{(8-3(-\kappa))(-\kappa-6)}{2(-\kappa)}.$$
Note that if $-\kappa$ is replaced by $\kappa$, then $\cc$ becomes the central charge for forward SLE$_\kappa$. So we view the $\cc$ here the central charge for backward SLE$_\kappa$, which runs in the interval $[25,\infty)$. Since $A_{j,S}=\frac{A_{j,3}}{A_{j,1}}-\frac 32(\frac{A_{j,2}}{A_{j,1}})^2$,
from (\ref{dAjh/Ajh}) we get
\BGE\frac{\pa_{t_j} A_{j,1}^{\alpha}}{ A_{j,1}^{\alpha}}=\alpha \frac{A_{j,2}}{A_{j,1}}d \lambda_j+
\Big[-\frac {\cc} 6 A_{j,S} +\frac \alpha6   A_{j,1}^2 - \frac \alpha6\Big]  dt_j.\label{pa-A-j-alpha}\EDE

Now we study $\pa_{t_j} A_{k,h}$ and $\pa_{t_j} A_{k,S}$. From (\ref{Ajh}) we have $A_{k,h}(t_1,t_2)=\til F_{j,t_j}^{(h)}(t_k,\lambda_k(t_k))$. Recall that $\til F_{j,t_j}(t_k,\cdot)=\til f_{j,t_k}(t_j,\cdot)$, and $\til f_{j,t_k}(t_j,\cdot)$ are backward covering radial Loewner maps via the time-change $u_{j,t_k}$ driven by $\lambda_{j,t_k}$.
From (\ref{lambda-j-tk}) and (\ref{ujtk}), we get
\BGE\pa_t \til f_{j,t_k}(t_j,z)=-A_{j,1}^2\cot_2(\til f_{j,t_k}(t_j,z)-A_{j,0}).\label{patfu}\EDE
Differentiate the above formula w.r.t.\ $z$, we get
\BGE\frac{\pa_t \til f_{j,t_k}'(t_j,z)}{\til f_{j,t_k}'(t_j,z)}=-A_{j,1}^2\cot_2'(\til f_{j,t_k}(t_j,z)-A_{j,0}) .\label{patfu'}\EDE
Differentiating the above formula w.r.t.\ $z$, we get
$$ \pa_t \frac{ \til f_{j,t_k}''(t_j,z)}{\til f_{j,t_k}'(t_j,z)}=-A_{j,1}^2\cot_2''(\til f_{j,t_k}(t_j,z)-A_{j,0})\til f_{j,t_k}'(t_j,z).$$ 
Since $f^{(S)}=(\frac{f''}{f'})'-\frac 12(\frac{f''}{f'})^2$, from the above formula, we get
\BGE \pa_t  { \til f_{j,t_k}^{(S)}(t_j,z)} =-A_{j,1}^2\cot_2'''(\til f_{j,t_k}(t_j,z)-A_{j,0})\til f_{j,t_k}'(t_j,z)^2.\label{patFS}\EDE
Letting $z=\lambda_k(t_k)$ in (\ref{patfu}), (\ref{patfu'}), and (\ref{patFS}), we get
\BGE \pa_{t_j} A_{k,0}=-A_{j,1}^2\cot_2(A_{k,0}-A_{j,0})dt_j;\label{paAk0}\EDE
\BGE \frac{\pa_{t_j} A_{k,1}}{A_{k,1}}=-A_{j,1}^2\cot_2'(A_{k,0}-A_{j,0})dt_j;\label{paAk1}\EDE
\BGE \pa_{t_j}A_{k,S}=-A_{j,1}^2A_{k,1}^2\cot_2'''(A_{k,0}-A_{j,0})dt_j.\label{paAkS}\EDE

Define $X_j$ on $\cal D$ such that $X_j=A_{j,0}-A_{k,0}$. Then $X_1+X_2=0$. Since $e^{i\lambda_j(t_j)}$ lies in the domain of $f_{k,t_j}(t_k,\cdot)$, $e^{iA_{j,0}}=f_{k,t_j}(t_k,e^{i\lambda_j(t_j)})$ lies in the range of $f_{k,t_j}(t_k,\cdot)$, i.e., $\ha\C\sem L_{k,t_j}(t_k)$. On the other hand, since via a time-change, $L_{k,t_j}(t_k)$ are backward radial Loewner hulls driven by $\lambda_{k,t_j}(t_k)=A_{k,0}$, from Lemma \ref{radial-S} we have $e^{iA_{k,0}}\in L_{k,t_j}(t_k)$ when $t_k>0$. Thus, $e^{iA_{j,0}}\ne e^{iA_{k,0}}$ if $t_k>0$. Switching $j$ and $k$, the inequality also holds if $t_j>0$. If $t_j=t_k=0$, then $e^{iA_{j,0}}=e^{i\til z_j}\ne e^{i\til z_k}=e^{iA_{k,0}}$. Thus, $X_j,X_k\not\in 2\pi\Z$. So we may define
 $$Y=|\sin_2(X_1)|^{-2\alpha}=|\sin_2(X_2)|^{-2\alpha}.$$
From (\ref{lambda-j-tk}), (\ref{dlambdajtk}), and (\ref{paAk0}), we get
$$ \pa_{t_j} X_j=A_{j,1} d\lambda_j+\Big(\frac\kappa 2+3\Big) A_{j,2}dt-A_{j,1}^2\cot_2(X_j)dt.$$ 
From It\^o's formula, we get
 $$\frac{\pa_{t_j} Y}Y=-\alpha \cot_2(X_j)A_{j,1}d\lambda_j-\alpha(\frac\kappa 2+3)A_{j,2}\cot_2(X_j)dt_j$$
 \BGE-\frac\alpha 2 A_{j,1}^2\cot_2^2(X_j)dt_j+\frac{\alpha\kappa}4A_{j,1}^2dt_j.\label{paY/Y}\EDE

Define $Q$ and $F$ on $\cal D$ such that $Q=\cot_2'''(X_1)=\cot_2'''(X_2)$ and
\BGE F(t_1,t_2)=\exp\Big(\int_{0}^{t_2}\!\int_0^{t_1} A_{1,1}(s_1,s_2)^2A_{2,1}(s_1,s_2)^2Q(s_1,s_2)ds_1ds_2\Big).\label{F}\EDE
Since $\til F_{k,t_k}^{(S)}(0,\cdot)=\id$, from (\ref{AS}) we have $A_{j,S}=0$ when $t_j=0$. From (\ref{paAkS}) we get
\BGE \frac{\pa_{t_j} F}F=-A_{j,S}dt_j.\label{paF}\EDE

Define a positive function $\ha M$ on $\cal D$ by
\BGE\ha M=A_{1,1}^\alpha A_{2,1}^\alpha Y F^{-\frac {\cc}6}e^{\frac{\cc}{12} \mA}.\label{haM}\EDE
From (\ref{mA}), (\ref{ujtk}), (\ref{pa-A-j-alpha}), (\ref{paAk1}), (\ref{paY/Y}), and (\ref{paF}), we have
\BGE \frac{\pa_{t_j}\ha M}{\ha M}=\alpha \frac{A_{j,2}}{A_{j,1}}d \lambda_j-\alpha \cot_2(X_j)A_{j,1}d\lambda_j- \frac \alpha6 dt_j.\label{pahaM}\EDE
When $t_k=0$, we have $A_{j,1}=1$, $A_{j,2}=0$, $\mA=t_j$, and $X_j=\lambda_j(t_j)-\til f_{j}(t_j,\til z_k)$, so the RHS of (\ref{pahaM}) becomes
\BGE \frac1\kappa \Big(\frac\kappa 2+3\Big)\cot_2(\lambda_j(t_j)-\til f_{j}(t_j,\til z_k))d\lambda_j- \frac \alpha6dt_j.\label{tk0}\EDE
Define another positive function $M$ on  $\cal D$ by
\BGE M(t_1,t_2)=\frac{\ha M(t_1,t_2) \ha M(0,0)}{\ha M(t_1,0)\ha M(0,t_2)}.\label{M}\EDE
Then $M(\cdot,0)\equiv M(0,\cdot)\equiv 1$. From (\ref{dlambda}), (\ref{pahaM}), and (\ref{tk0}), we have
$$\frac{\pa_{t_j} M}{ M}=\Big[ -\Big(3+\frac\kappa2\Big) \frac{A_{j,2}}{A_{j,1}}-\frac{-\kappa-6}{2}\cot_2(X_j) A_{j,1}$$
\BGE+\frac{-\kappa-6}{2}\cot_2(\lambda_j(t_j)-\til f_{j}(t_j,\til z_k)) \Big] \frac{d B_j(t_j)}{\sqrt\kappa}.\label{paM}\EDE
 So when $t_k\in [0,p)$ is a fixed $(\F^k_t)$-stopping time, $M$ as a function in $t_j$ is an $(\F^j_{t_j}\times \F^k_{t_k})_{t_j\ge 0}$-local martingale.

\subsection{Coupling measures}

Let $\JP$ denote the set of disjoint pairs of closed arcs $(J_1,J_2)$ on $\TT$ such that $z_j=e^{i\til z_j}$ is contained in the interior of $J_j$, $j=1,2$. Let $T_j(J_j)$ denote the first time that $S_{L_j(t)}$ intersects $\lin{\TT\sem J_j}$.
Then for every $(J_1,J_2)\in\JP$, if $t_j\le T_{j}(J_j)$, then $S_{L_j(t_j)}\subset J_j$, which implies that $L_j(t_j)\in {\cal H}_{J_j}$. So
$[0,T_{1}(J_1)]\times[0,T_{2}(J_2)]\subset \cal D$.

\begin{Proposition}  (Boundedness) For any $(J_1,J_2)\in\JP$, $|\ln(M)|$ is uniformly bounded on $[0,T_1(J_1)]\times[0,T_2(J_2)]$  by a constant depending only on $J_1$ and $J_2$. \label{bounded}
\end{Proposition}
\begin{proof} Fix $(J_1,J_2)\in\JP$. In this proof, all constants depend only on $(J_1,J_2)$, and we say a function is uniformly bounded if its values on $[0,T_1(J_1)]\times[0,T_2(J_2)]$ are bounded in absolute value by a constant. From (\ref{haM}) and (\ref{M}), it suffices to show that $\ln(A_{1,1})$, $\ln(A_{2,1})$, $\ln(Y)$, $\ln(F)$, and $\mA$ are all uniformly bounded.

Note that if $t_j\le T_j(J_j)$, then $L_j(t_j)\in {\cal H}_{J_j}$. From a radial version of Theorem \ref{bijection-gf*} (iii), we have
\BGE \{L_1(t_1)\vee L_2(t_2):t_j\in[0, T_j(J_j)],\,j=1,2\}\subset {\cal H}_{J_1\cup J_2}.\label{compact-L}\EDE
Since $J_1\cup J_2\subsetneqq\TT$, from Lemma \ref{compact-D-F}, the righthand side is a compact set. So the lefthand side is relatively compact. Since $H\mapsto \dcap(H)$ is continuous, and $\mA(t_1,t_2)=\dcap(L_1(t_1)\vee L_2(t_2))$, we see that  $\mA$ is uniformly bounded. For $j=1,2$, since $ T_j(J_j)\le\mA$, $T_j(J_j)$ is also uniformly bounded.

Let $S_1$ and $S_2$ be the two components of $\TT\sem(J_1\cup J_2)$. For $s=1,2$, let $E_s\subset S_s$ be a compact arc.
From Lemma \ref{compact-new-r}, $L_n\to L$ in ${\cal H}_{J_1\cup J_2}$ implies that $f_{L_n}\luto f_L$ in $\C\sem (J_1\cup J_2)$, which then implies that $f_{L_n}'\luto f_L'$ in $\C\sem (J_1\cup J_2)$.
From (\ref{compact-L}), the compactness of ${\cal H}_{J_1\cup J_2}$, and that $E_1\cup E_2$ are compact subsets of $\C\sem(J_1\cup J_2)$, we conclude that there is a constant $c_1>0$ such that $|f_{L_1(t_1)\vee L_2(t_2)}'(z)|\ge c$ for any $t_j\le T_j(J_j)$, $j=1,2$, and $z\in E_1\cup E_2$. Thus, for $t_j\in[0, T_j(J_j)]$, $j=1,2$, the length of $f_{L_1(t_1)\vee L_2(t_2)}(E_s)$, $s=1,2$, is bounded below by a constant $c_2>0$. Suppose $t_j\in(0,T_j(t_j)]$, $j=1,2$. From Lemma \ref{radial-S}, $e^{iA_{j,0}}\in B_{L_{j,t_{3-j}}(t_j)}$, $j=1,2$. Note that $f_{L_1(t_1)\vee L_2(t_2)}(E_1\cup E_2)$ disconnects $B_{L_{1,t_2}(t_1)}$ from $B_{L_{2,t_1}(t_2)}$ on $\TT$. Thus, there is a constant $c_3>0$ such that $|e^{iA_{1,0}(t_1,t_2)}-e^{iA_{2,0}(t_1,t_2)}|\ge c_3$ for $t_j\in(0,T_j(t_j)]$, $j=1,2$. From continuity, this still holds if $t_j\in[0, T_j(J_j)]$, $j=1,2$. Thus, $\ln(Y)=-2\alpha\ln|\sin_2(X_j)|$, $|\cot_2'(X_j)|$, and $|\cot_2'''(X_j)|$, $j=1,2$, are all uniformly bounded.

We may find a Jordan curve $\sigma$, which is disjoint from $J_1\cup J_2$, such that its interior contains $J_1$ and its exterior contains $J_2$. From compactness, $\sup_{z\in\sigma}\ln|f_j'(t_j,z)|$ and $\sup_{z\in\sigma} \ln|f_{L_1(t_1)\vee L_2(t_2)}'(z)|$ are both uniformly bounded. From (\ref{circ-f}) we see that the value $\sup_{w\in f_j(t_j,\sigma)} \ln|f_{3-j,t_j}'(t_k,w)|$ is also uniformly bounded. Note that the interior of $f_j(t_j,\sigma)$ contains $\ha{L_j(t_j)}$, which contains $e^{i\lambda_j(t_j)}$ if $t_j>0$. From maximal principle, there is $c_4\in(0,\infty)$ such that $A_{j,1}(t_1,t_2)=|f_{3-j,t_j}'(t_{3-j},e^{i\lambda_j(t_j)})|\le c_4$ if $t_j\in(0,T_j(J_j)]$ and $t_{3-j}\in[0,T_{3-j}(J_{3-j})]$. From continuity, $A_{j,1}$ is uniformly bounded, $j=1,2$. From (\ref{paAk1}) and the uniformly boundedness of $|\cot_2'(X_j)|$ we see that $\ln(A_{j,1})$ is uniformly bounded, $j=1,2$. From (\ref{F}) and the uniformly boundedness of $|\cot_2'''(X_j)|$ we see that $\ln(F)$ is also uniformly bounded, which completes the proof.
\end{proof}

Let $\mu_j$ denote the distribution of $(\lambda_j)$, $j=1,2$. Let $\mu=\mu_1\times\mu_2$. Then
$\mu$ is the joint distribution of $(\lambda_1)$ and $(\lambda_2)$, since $\lambda_1$ and $\lambda_2$ are independent.
 Fix $(J_1,J_2)\in\JP$. From  the local martingale property  of $M$ and  Proposition \ref{bounded}, we have
$ \EE_\mu[M(T_1(J_1),T_2(J_2))]=M(0,0)=1$. Define $\nu_{J_1,J_2}$ by
$d\nu_{J_1,J_2}/d\mu=M(T_1(J_1),T_2(J_2))$. Then $\nu_{J_1,J_2}$ is a probability measure.
Let $\nu_1$ and $\nu_2$ be the two marginal measures of $\nu_{J_1,J_2}$. Then
$d\nu_1/d\mu_1=M(T_1(J_1),0)=1$ and $d\nu_2/d\mu_2=M(0,T_2(J_2))=1$, so $\nu_j=\mu_j$, $j=1,2$.
Suppose temporarily that the joint distribution of $(\lambda_1)$ and $(\lambda_2)$ is $\nu_{J_1,J_2}$ instead of $\mu$.
Then the distribution of each $(\lambda_j)$ is still $\mu_j$.

Fix an $(\F^2_t)$-stopping time $t_2\le T_2(J_2)$. From (\ref{dlambda}), (\ref{paM}), and Girsanov theorem
(c.f.\ \cite{RY}), under the probability measure $\nu_{J_1,J_2}$,
there is an $(\F^1_{t_1}\times\F^2_{t_2})_{t_1\ge 0}$-Brownian
motion $\til B_{1,t_2}(t_1)$  such that $\lambda_1(t_1)$,  $0\le t_1\le T_1(J_1)$, satisfies the
$(\F^1_{t_1}\times\F^2_{ t_2})_{t_1\ge 0}$-adapted SDE: $$ d\lambda_1(t_1)=\sqrt\kappa d \til B_{1,t_2}(t_1)-\Big(3+\frac\kappa 2\Big)
\frac{A_{1,2} }{A_{1,1} }dt_1-\frac{-\kappa-6}2\cot_2(X_1) A_{1,1} dt_1,$$ 
which together with (\ref{Ajh}), (\ref{lambda-j-tk}), (\ref{3A}), and It\^o's formula, implies that
$$ d\lambda_{1,t_2}(t_1)=A_{1,1}  \sqrt\kappa d\til B_{1,t_2}(t_1) -\frac{-\kappa-6}2\cot_2(X_1)A_{1,1} ^2dt_1.$$ 
From (\ref{Ajh}) and (\ref{lambda-j-tk}) we get $X_1=A_{1,1}-A_{2,1}=\lambda_{1,t_2}(t_1)-\til f_{1,t_2}(t_1,\lambda_2(t_2))$.
Note that $\lambda_{1,t_2}(0)=\til f_{2,0}(t_2,\til z_1)=\til f_2(t_2,\til z_1)$.
Since $L_{1,t_2}(t_1)$ and $\til f_{1,t_2}(t_1,\cdot)$ are backward radial Loewner hulls and covering maps via the time-change $u_{1,t_2}$, from (\ref{ujtk}) and the above equation, we find that, under the measure $\nu_{J_1,J_2}$, conditioned on $\F^1_{t_2}$ for any $(\F^2_t)$-stopping time $t_2\le T_2(J_2)$, via the time-change $u_{1,t_2}$, $L_{1,t_2}(t_1)=f_2(t_2,\cdot)^*(L_1(t_1))$, $0\le t_1\le  T_1(J_1)$, is a partial backward radial SLE$(\kappa;\frac{-\kappa-6}2)$ process started from $e^i\circ\til f_2(t_2,\til z_1)=f_2(t_2,z_1)$ with marked point $e^i(\lambda_2(t_2))$. Similarly, the above statement holds true if the subscripts ``$1$'' and ``$2$'' are exchanged.

The joint distribution $\nu_{J_1,J_2}$ is a local coupling such that the desired properties in the statement of Theorem \ref{coupl} holds true up to the stopping times $T_1(J_1)$ and $T_2(J_2)$. Then we can apply the maximum coupling technique developed in \cite{reversibility} to construct a global coupling using the local couplings within different pairs $(J_1,J_2)$. The reader is referred to
Theorem 4.5 and Section 4.3 in \cite{duality} for the construction of a global coupling between two forward SLE processes. For the coupling of backward SLE processes, the method is essentially the same. A slight difference is that for the forward SLE processes, a pair of hulls were used to control the growth of $M(\cdot,\cdot)$, which stays uniformly bounded up to the time that the SLE hulls grow out of the given hulls; while for the backward SLE processes, we here used a pair of arcs to control the growth of $M(\cdot,\cdot)$. One fact that is worth mentioning is that here we may choose a dense countable set $\JP^*\subset \JP$ such that, when $S_{L_1(t_1)}\cap S_{L_2(t_2)}=\emptyset$, there exists $(J_1,J_2)\in\JP^*$ with $S_{L_j(t_j)}\subset J_j$, $j=1,2$, from which follows that \BGE T_j(t_k)=\sup\{T_j(J_j):(J_1,J_2)\in\JP^*,t_k\le T_k(J_k)\},\quad j\ne k\in\{1,2\}.\label{Tjtk}\EDE
This finishes the proof 
of Theorem \ref{coupl} in the radial case.

Now we briefly describe the proof for the chordal case. The proof in this case is  simpler because there are no covering maps. Suppose the two backward chordal SLE$(\kappa;-\kappa-6)$ processes start from $(z_j;z_k)$, where $z_1\ne z_2\in\R$.
 Formula (\ref{dlambda}) holds with all tildes removed and the function $\cot_2$ replaced by $z\mapsto \frac 2z$. The domain $\cal D$ and the $\HH$-hulls $L_{1,t_2}(t_1)$ and $L_{2,t_1}(t_2)$ are defined in the same way. Then (\ref{circ-f}) still holds. From Corollary \ref{continuous-Cor} (ii), $f_{1,t_2}(t_1,\cdot)$ and $f_{2,t_1}(t_2,\cdot)$ are $\F^1_{t_1}\times \F^2_{t_2}$-measurable. Define $\mA(t_1,t_2)=\hcap(L_1(t_1)\vee L_2(t_2))/2$. Then (\ref{mA}) holds with $u_{j,t_k}(t_j):=\hcap(L_{j,t_k}(t_j))/2$.

Now we apply the argument in Section \ref{Conformal-chordal} with $W=f_k(t_k,\cdot)$. Then $W_t=f_{k,t_j}(t_k,\cdot)$. Let $F_{k,t_k}(t_j,\cdot)=f_{k,t_j}(t_k,\cdot)$, and define $A_{j,h}$ and $A_{j,S}$ using (\ref{Ajh}) and (\ref{AS}) with all tildes removed. Using (\ref{W=}), (\ref{u'}), (\ref{-3}), and (\ref{W''/W'-chordal}), we see that (\ref{lambda-j-tk}) still holds here; (\ref{ujtk}) and (\ref{3A}) hold with all tildes removed; and (\ref{W''/W'-ensemble}) holds without the tildes and the terms $+\frac16A_{j,1}^2-\frac 16$. 
Then we get the SDEs (\ref{dlambdajtk}) and (\ref{pa-A-j-alpha}) without the terms $+\frac \alpha6   A_{j,1}^2 - \frac \alpha6$. Formulas (\ref{paAk0}), (\ref{paAk1}), and (\ref{paAkS}) hold with $\cot_2$ replaced by $z\mapsto \frac 2z$. We still define $X_j=A_{j,1}-A_{k,1}$. Then $X_j\ne 0$ in $\cal D$. Define $Y$ on $\cal D$ by $Y=|X_1|^{-2\alpha}=|X_2|^{-2\alpha}$. Then (\ref{paY/Y}) holds with $\cot_2$ replaced by $z\mapsto \frac 2z$ and the term $+\frac{\alpha\kappa}4A_{j,1}^2dt_j$ removed. Define $F$ using (\ref{F}) with $Q=-\frac{12}{X_1^4}=-\frac{12}{X_2^4}$. Then (\ref{paF}) still holds. Define $\ha M$ using (\ref{haM}) without the factor $e^{\frac{\cc}{12} \mA}$. Then (\ref{pahaM}) holds with $\cot_2$ replaced by $z\mapsto \frac 2z$ and the term $- \frac \alpha6 dt_j$ removed. Define $M$ using (\ref{M}). Then (\ref{paM}) holds with all tildes removed and $\cot_2$ replaced by $z\mapsto \frac 2z$.

We define $\JP$ to be the set of disjoint pairs of closed real intervals $(J_1,J_2)$ such that $z_j$ is contained in the interior of $J_j$. Then Proposition \ref{bounded} holds with a similar proof, where Lemma \ref{compact-new} is applied here, and we can show that $|X_1|$ is uniformly bounded away from $0$. The argument on the local couplings hold with all tildes and $e^i$ removed and $\cot_2$ replaced by $z\mapsto \frac 2z$. Finally, we may apply the maximum coupling technique to construct a global coupling with the desired properties. Formula (\ref{Tjtk}) still holds here and is used in the construction. This finishes the proof in the chordal case.

\subsection{Other results}
Besides Theorem \ref{coupl}, one may also prove the following two theorems, which are similar to the couplings for forward SLE that appear in \cite{Julien-Comm} and \cite{duality}.

\begin{Theorem}
  Let $\kappa_1,\kappa_2>0$ satisfy $\kappa_1\kappa_2=16$, and $c_1,\dots,c_n\in\R$ satisfy $\sum_{k=1}^n c_k=\frac 32$.
  Let $\vec\rho_j=(\frac{\kappa_j}2,c_1(-\kappa_j-4),\dots,c_n(-\kappa_j-4))$, $j=1,2$.
  Then backward chordal (resp.\ radial) SLE$(\kappa_1;\vec\rho_1)$ commutes with backward chordal (resp.\ radial) SLE$(\kappa_2;\vec\rho_2)$.
\end{Theorem}

\begin{Theorem}
Let $\kappa>0$ and $\vec\rho\in\R^n$, whose first coordinate is $2$. Then backward chordal (resp.\ radial) SLE$(\kappa;\vec\rho)$ commutes with backward chordal (resp.\ radial) SLE$(\kappa;\vec\rho)$.
\end{Theorem}

\section{Reversibility of Backward Chordal SLE} \label{section-reversibility}

\begin{Theorem}
Let $\kappa\in(0,4]$ and $z_1\ne z_2\in\TT$. Suppose a backward radial SLE$(\kappa;-\kappa-6)$ process $(L_1(t))$ started from $(z_1;z_2)$ commutes with a backward radial SLE$(\kappa;-\kappa-6)$ process $(L_2(t))$ started from $(z_2;z_1)$. Then a.s.\ they induce the same welding.
\end{Theorem}
\begin{proof}
For $j=1,2$, let $S^j_t=S_{L_j(t)}$ and $f^j_t=f_{L_j(t)}$. Let $T_j(\cdot)$, $j=1,2$, be as in Definition \ref{def-coupl}. Let $\phi_j$ be the welding induced by $(L_j(t))$. Since $-\kappa-6\le -\kappa/2-2$, from the last remark in Section \ref{section-kappa-rho}, we see that, for $j=1,2$, a.s.\ $T_j=\infty$, $S^j_\infty=\TT\sem\{z_{3-j}\}$, and $\phi_j$ is an involution of $\TT$ with exactly two fixed points: $z_1$ and $z_2$.

Fix $t_2>0$. Since $(L_1(t))$ and $(L_2(t))$ commute, the following is true. Conditioned on $(L_2(t))_{t\le t_2}$, $(f^2_{t_2})^*(L_1(t_1))$, $0\le t_1<T_1(t_2)$, is a partial backward radial SLE$(\kappa;-\kappa-6)$ process, after a time-change, started from $(f^2_{t_2}(z_1);B_{L_2(t_2)})$. Here we use $B_{L_2(t_2)}$ also to denote the unique point in the base of $L_2(t_2)$, which is equal to $e^{i\lambda_2(t_2)}$, where $\lambda_2$ is a driving function for $(L_2(t_2))$. We have
\BGE S:= \bigcup_{0\le t_1<T_1(t_2)} S_{(f^2_{t_2})^*(L_1(t_1))}=f^2_{t_2}(\bigcup_{0\le t_1<T_1(t_2)} S^1_{t_1})=f^2_{t_2}(S^1_{T_1(t_2)^-}).\label{f*}\EDE
Recall that $f^2_{t_2}$ is a homeomorphism from $\TT\sem S^2_{t_2}$ onto $\TT\sem B_{L_2(t_2)}$. From the definition of $T_1(t_2)$, we see that $S^1_{T_1(t_2)}$ intersects $S^2_{t_2}\ne\emptyset$ at one or two end points of both arcs. If they intersect at only one point, then $S^1_{T_1(t_2)^-}$ is a proper subset of $\TT\sem S^2_{t_2}$, and these two arcs share an end point. From (\ref{f*}), this then implies that the arc $S$ is a proper subset of $\TT\sem B_{L_2(t_2)}$, and $B_{L_2(t_2)}$ is an end point of $S$. Recall that, after a time-change, $(f^2_{t_2})^*(L_1(t_1))$, $0\le t_1<T_1(t_2)$, is a partial backward radial SLE$(\kappa;-\kappa-6)$ process. Since $S\ne\TT\sem B_{L_2(t_2)}$, the process is not complete. Then we conclude that $S$ is contained in a closed arc on $\TT$ that does not contain $B_{L_2(t_2)}$ because the force point is not swallowed by the process at any finite time, which contradicts that $B_{L_2(t_2)}$ is an end point of $S$. Thus, a.s.\ $S^1_{T_1(t_2)}$ and $S^2_{t_2}$ share two end points. Since $\phi_j$ swaps the two end points of any $S^j_t$, $j=1,2$, we see that a.s.\ $\phi_2=\phi_1$ on $\pa_{\TT} S^2_{t_2}$. Let $t_2>0$ vary in the set of rational numbers, we see that a.s.\ $\phi_2=\phi_1$ on $\bigcup_{t\in\Q_{>0}} \pa_{\TT} S^2_{t_2}$, which is a dense subset of $\TT$. The conclusion follows since $\phi_1$ and $\phi_2$ are continuous.
\end{proof}

We now state the reversibility of backward chordal SLE$_\kappa$ for $\kappa\in(0,4]$ in terms of its welding. Recall that a backward chordal SLE$_\kappa$ welding  is an involution of $\ha\R$ with two fixed points: $0$ and $\infty$.

\begin{Theorem}
  Let $\kappa\in(0,4]$, and $\phi$ be a backward chordal SLE$_\kappa$ welding. Let $h(z)=-1/z$. Then $h\circ \phi\circ h$ has the same distribution as $\phi$.\label{reversibility}
\end{Theorem}
\begin{proof}
 Let $(L_1(t))$ and $(L_2(t))$ be commuting backward radial SLE$(\kappa;-\kappa-6)$ procesees as in Theorem \ref{coupl}, which induce the weldings $\psi_1$ and $\psi_2$, respectively. The above theorem implies that a.s.\ $\psi_1=\psi_2$.
 For $j=1,2$, let $W_j$ be a M\"obius transformation that maps $\D$ onto $\HH$ such that $W_j(z_j)=0$ and $W_j(z_{3-j})=\infty$, and $W_2=h\circ W_1$. From Corollary \ref{chordal-radial}, $K_j(t):=W_j^*(L_j(t))$, $0\le t<\infty$, is a backward chordal SLE$_\kappa$, after a time-change, which then induces backward chordal SLE$_\kappa$ welding $\phi_j$, $j=1,2$. Then $\phi_1$ and $\phi_2$ have the same law as $\phi$. From (\ref{welding-trans}), we get $\phi_j=W_j\circ \psi_j\circ W_j^{-1}$, $j=1,2$, which implies that a.s.\ $\phi_2=h\circ \phi_1\circ h$. The conclusion follows since $\phi_1$ and $\phi_2$ has the same distribution as $\phi$.
\end{proof}

\begin{Lemma}
  Let $\kappa>0$. Let $f_t$, $0\le t<\infty$, be backward chordal SLE$_\kappa$ maps. Then for every $z_0\in\HH$, a.s.\ (\ref{z0}) holds. \label{z0-SLE}
\end{Lemma}
\begin{proof}
  Let $Z_t=f_t(z_0)$, $X_t=\Ree z_t$, and $Y_t=\Imm z_t$.  Then
  $$dX_t=-\sqrt\kappa dB(t)-\frac{2X_t}{X_t^2+Y_t^2}\,dt,\quad dY_t=\frac{2Y_t}{X_t^2+Y_t^2}\,dt$$
  Let $R_t=|f_t'(z_0)|$. Then $ \frac{dR_t}{R_t}=\frac{2(X_t^2-Y_t^2)}{(X_t^2+Y_t^2)^2}dt$.
  Let $N_t=Y_t/R_t$ and $A_t=X_t/Y_t$. Then
  $$\frac{dN_t}{N_t}=\frac{4Y_t^2}{(X_t^2+Y_t^2)^2}\,dt,\quad dA_t=-\frac{\sqrt\kappa dB(t)}{Y_t}-\frac{4A_t}{X_t^2+Y_t^2}\,dt.$$
  Let $u(t)=\ln(Y_t)$. Then $u'(t)=\frac{2}{X_t^2+Y_t^2}$. Let $T=\sup u([0,\infty))$ and define $\ha N_s=N_{u^{-1}(s)}$ and $\ha A_s=A_{u^{-1}(s)}$ for $0\le s<T$. Then
  $$\frac{d\ha N_s}{\ha N_s}=\frac{2}{\ha A_s^2+1}\,ds,\quad d\ha A_s=-\sqrt{1+\ha A_s^2}\sqrt{\kappa/2} d\ha B(s)-2\ha A_sds,$$
  where $\ha B(s)$ is another Brownian motion. We claim that $T=\infty$. Suppose $T<\infty$. Then $\lim_{t\to\infty} Y(t)=e^T\in\R$. From the SDE for $A_s$, we see that a.s.\ $\lim_{s\to T} A_s\in\R$, which implies that $\lim_{t\to \infty} A_t\in\R$ and $\lim_{t\to\infty} X_t\in\R$ as $X_t=Y_tA_t$. Then we have a.s.\ $s'(t)=\frac{2}{X_t^2+Y_t^2}$ tends to a finite positive number as $t\to\infty$, which contradicts that $T=\sup\{s(t), 0\le t < \infty\}<\infty$. So the claim is proved. Using It\^o's formula, we see that $\ha A_s$, $0\le s < \infty$, is recurrent. Since $(\ln(\ha N_s))'=\frac2{\ha A_s^2+1}$, we see that a.s.\ $\ha N_s\to \infty$ as $s\to\infty$. So a.s.\ $N_t=\frac{\Imm f_t(z_0)}{|f_t'(z_0)|}\to\infty$ as $t\to\infty$, i.e., (\ref{z0}) holds.
\end{proof}

If $\kappa\in(0,4]$, then since the backward chordal traces are simple, (\ref{beta>0}) holds.
From the above lemma and Section \ref{normalized}, we see that, for $\kappa\in(0,4]$, the backward chordal SLE$_\kappa$ a.s.\ generates a normalized global backward chordal trace $\beta$, which we call a normalized global backward chordal SLE$_\kappa$ trace.
Recall that $\beta(t)$, $0\le t<\infty$, is simple with $\beta(0)=0$, and $i\not\in \beta$; and there is $F_\infty:\HH\conf\C\sem\beta$, whose continuation maps $\R$ onto $\beta$ such that (\ref{norm}) holds, and for any $x\in\R$, $F_\infty(x)=F_\infty(\phi(x))\in\beta$. Now we state the reversibility of the backward chordal SLE$_\kappa$ for $\kappa\in(0,4)$ in terms of $\beta$.

\begin{Theorem}
  Let $\kappa\in(0,4)$, and $\beta$ be a normalized global backward chordal SLE$_\kappa$ trace. Let $h(z)=-1/z$. Then $h(\beta\sem\{0\})$ has the same distribution as $\beta\sem\{0\}$ as random sets.
\end{Theorem}
\begin{proof}
  For $j=1,2$, let $\phi_j$ be a backward chordal SLE$_\kappa$ welding and $\beta_j$ be the corresponding normalized global trace. Then $\beta_j$ is a simple curve with one end point $0$, and there exists $F_j:\HH\conf \C\sem\beta_j$ such that $F_j(i)=i$, $F_j(0)=0$, and $F_j(x)=F_j(\phi_j(x))$ for $x\in\R$. From Theorem \ref{reversibility} we may assume that $\phi_2=h\circ \phi_1\circ h^{-1}$. Now it suffices to show that $h(\beta_2\sem\{0\})=\beta_1\sem\{0\}$.

  Define $G=h\circ F_2\circ h\circ F_1^{-1}$. Then $G$ is a conformal map defined on $\C\sem\beta_1$. It has continuation to $\beta_1\sem\{0\}$. In fact, if $z\in\C\sem\beta_1$ and $z\to z_0\in\beta_1\sem\{0\}$, then $F_1^{-1}(z)\to \{x,\phi_1(x)\}$ for some $x\in\R\sem\{0\}$, which then implies that $h\circ F_1^{-1}(z)\to \{h(x),h\circ \phi_1(x)\}$; since $\phi_2\circ h=h\circ\phi_1$, we find that $F_2\circ h\circ F_1^{-1}(z)$ tends to some point on $\beta_2\sem \{0\}$, so $G(z)$ tends to some point on $h(\beta_2\sem\{0\})$. It was proved in \cite{RS-basic} that a forward SLE$_\kappa$ trace is the boundary of a H\"older domain. Then the same is true for backward chordal SLE$_\kappa$ traces and the normalized global trace. From the results in \cite{JS}, we see that $\beta_1\sem\{0\}$ is conformally removable, which means that $G$ extends to a conformal map from $(\C\sem\beta_1)\cup(\beta_1\sem\{0\})=\C\sem\{0\}$ onto $\C\sem\{0\}$, and maps $\beta_1\sem\{0\}$ to $h(\beta_2\sem\{0\})$. Since $G(i)=i$, either $G=\id$ or $G=h$. Suppose $G=h$. Then $F_1=F_2\circ h$. Since $F_1(0)=F_2(0)=0$, for $j=1,2$, $F_j$ maps a neighborhood of $0$ in $\HH$ onto a neighborhood of $0$ in $\C$ without a simple curve. Since $F_1=F_2\circ h$,  $F_1$ also maps a neighborhood of $\infty$ in $\HH$ onto a neighborhood of $0$ without a simple curve, which contradicts the univalent property of $F_1$.  Thus, $G=\id$, and we get $h(\beta_2\sem\{0\})=G(\beta_1\sem\{0\})=\beta_1\sem\{0\}$, as desired.
\end{proof}

Now we propose a couple of questions. First, let's consider backward chordal SLE$_\kappa$ for $\kappa>4$. Since the process does not generate simple backward chordal traces, the random welding $\phi$ can not be defined. However, the lemma below and the discussion in Section \ref{normalized} show that we can still define a global backward chordal SLE$_\kappa$ trace.

\begin{Lemma}
Let $\kappa\in(0,\infty)$. Suppose $\beta_{t}$, $0\le t<\infty$, are backward chordal traces driven by $\lambda(t)=\sqrt\kappa B(t)$. Then a.s.\ (\ref{beta>0}) holds.
\end{Lemma}
\begin{proof} If $\kappa\in(0,4]$, a.s.\ the traces are simple, so (\ref{beta>0}) holds. Now suppose $\kappa>4$. Let $f_t$ and $L_t$ be the corresponding maps and hulls. It suffices to show that, for any $t_0>0$, a.s.\ there exists $t_1>t_0$ such that $\beta_{t_1}([0,t_0])\subset\HH$.

Let $g_t$ and $K_t$, $0\le t<\infty$, be the forward chordal Loewner maps and hulls driven by $\sqrt\kappa B(t)$. From Theorem 6.1 in \cite{duality2}, for any deterministic time $t_1\in(0,\infty)$, the continuation of $g_{t_1}^{-1}$ a.s.\ maps the interior of $S_{K_{t_1}}$ into $\HH$. From Lemma \ref{ft2t1} and the property of Brownian motion, we see that, for any $t_1\in(0,\infty)$, $f_{t_1}$ has the same distribution as $\lambda(t_1)+g_{t_1}^{-1}(\cdot-\lambda(t_1))$, which implies that the continuation of $f_{t_1}$ a.s.\ maps the interior of $S_{L_{t_1}}$ into $\HH$.

Since a.s.\ $\bigcup_{n=1}^\infty S_n=S_\infty=\R\supset \lambda([0,t_0])$, and $(S_t)$ is an increasing family of intervals, we see that a.s.\ there is $N\in\N$ such that the interior of $S_N$ contains $\lambda([0,t_0])$. Let $t_1=N$. Then $f_{t_1}$ maps $\lambda([0,t_0])$ into $\HH$, which implies that $\beta_{t_1}(t)=f_{t_1}(\lambda(t))\in\HH$ for $0\le t\le t_0$.
\end{proof}

\begin{Question}
  Do we have the reversibility of the global backward chordal SLE$_\kappa$ trace for $\kappa>4$?
\end{Question}

Second, let's consider backward radial SLE$_\kappa$ processes. One can show that (\ref{beta>0-r}) a.s.\ holds. Since $T=\infty$, we may define a global backward radial SLE$_\kappa$ trace.

\begin{Question}
  Does a global backward radial SLE$_\kappa$ trace satisfy some reversibility property of any kind?
\end{Question}
Recall that the forward radial SLE$_\kappa$ trace does not satisfy the reversibility property in the usual sense. However, it's proved in \cite{whole} that, for $\kappa\in(0,4]$, the whole-plane SLE$_\kappa$, as a close relative of radial SLE$_\kappa$, satisfies reversibility.

Finally, it is worth mentioning the following simple fact. Recall that, if $\kappa\in(0,4]$, a backward radial SLE$_\kappa$ welding is an involution of $\TT$ with two fixed points, one of which is $1$. The following theorem gives the distribution of the other fixed point $\zeta$, and says that a backward radial SLE$_\kappa$ process conditioned on $\zeta$ is a backward radial SLE$(\kappa;-4)$ process with force point $\zeta$. It is similar to Theorem 3.1 in \cite{duality2}, and we omit its proof.

\begin{Theorem}
  Let $\kappa\in(0,4]$. Let $\mu$ denote the distribution of a backward radial SLE$_\kappa$ process. For $\theta\in(0,2\pi)$, let $\nu_\theta$ denote the distribution of a backward radial SLE$(\kappa;-4)$ process started from $(1;e^{i\theta})$. Let $f(\theta)=C\sin_2(\theta)^{4/\kappa}$, where $C>0$ is such that $\int_0^{2\pi} f(\theta)d\theta=1$. Then $$\mu=\int_0^{2\pi} \nu_\theta f(\theta)d\theta.$$
\end{Theorem}

\appendixpage
\addappheadtotoc
\appendix
\section{Carath\'eodory Topology} \label{A}
\begin{Definition} Let $(D_n)_{n=1}^\infty$ and $D$ be domains in $\C$. We say that $(D_n)$ converges to $D$, and write $D_n\dto
D$, if for every $z\in D$, $\dist(z,\C\sem D_n)\to \dist
(z,\C\sem D)$. This is equivalent to the following:
\begin{enumerate}
  \item [(i)]  every compact subset of $D$ is contained in all but finitely
many $D_n$'s;
\item [(ii)] for every point $z_0\in\pa D$, there exists $z_n\in\pa D_n$ for each $n$ such that $z_n\to z_0$.
\end{enumerate} \label{def-lim}
\end{Definition}

\no {\bf Remark.}
 A sequence of domains may converge to two different domains. For example, let $D_n=\C\sem((-\infty,n])$. Then $D_n\dto\HH$, and
$D_n\dto -\HH$ as well. But two different limit domains of the same domain sequence must be disjoint from each other, because if they
have nonempty intersection, then one contains some boundary point of the other, which implies a contradiction.


\begin{Lemma} Suppose $D_n\dto D$, $f_n:D_n\conf E_n$, $n\in\N$, and $f_n\luto f$ in $D$. Then either
$f$ is constant on $D$, or $f$ is a conformal map on $D$. In the latter case, let $E=f(D)$. Then $E_n\dto E$ and $f_n^{-1}\luto f^{-1}$
in $E$. \label{domain convergence*}
\end{Lemma}

\no{\bf Remark.} 
The above lemma resembles the Carath\'eodory kernel theorem (Theorem 1.8, \cite{Pom-bond}), but the domains here don't have to be simply connected.
The main ingredients in the proof are Rouch\'e's theorem and Koebe's $1/4$ theorem. The lemma
 also holds in the case that $D_n$ and $D$ are domains of any Riemann surface, if the metric in the underlying space is used in place of the Euclidean metric for Definition \ref{def-lim} and locally uniformly convergence. In particular, if we use the spherical metric, then Lemma \ref{domain convergence*} holds for domains of $\ha\C$.

\section{Topology on Interior Hulls}\label{B}
Let ${\cal H}$ denote the set of all interior hulls in $\C$. Recall that for any $H\in\cal H$, $\phi_H^{-1}$ is defined on $\{|z|>\rad(H)\}$, and for a nondegenerate interior hull, $\psi_H(z)=\vphi_H^{-1}(z)=\phi_H^{-1}(\rad(H)z)$ is defined on $\{|z|>1\}$.
It's shown in Section 2.5 of \cite{int-LERW} that there is a metric $d_{\cal H}$ on $\cal H$ such that for any $H_n,H\in\cal H$, the followings are equivalent:
\begin{enumerate}
  \item $d_{\cal H}(H_n,H)\to 0$
  \item $\rad(H_n)\to\rad(H)$ and $\phi_{H_n}^{-1}\luto\phi_H^{-1}$ in $\{|z|>\rad(H)\}$.
  \item $\C\sem H_n\dto \C\sem H$.
\end{enumerate}
In particular, we see that $\rad$ is a continuous function on $({\cal H},d_{\cal H})$. Thus, for nondegenerate interior hulls, $d_{\cal H}(H_n,H)\to 0$ iff $\psi_{H_n}\luto\psi_H$ in $\{|z|>1\}$. The following lemma is Lemma 2.2 in \cite{int-LERW}.

\begin{Lemma}
  For any $F\in\cal H$, the set $\{H\in{\cal H}:H\subset F\}$ is compact.\label{compact-lem}
\end{Lemma}

\begin{Corollary}
  For any $F\in\cal H$ and $r>0$, the set $\{H\in{\cal H}:H\subset F,\rad(H)\ge r\}$ is compact. \label{compact-cor}
\end{Corollary}

\section{Topology on $\HH$-hulls} \label{C}
From Section 5.2 in \cite{LERW}, there is a metric $d_{\cal H}$ on the space of $\HH$-hulls such that $d_{\cal H}(H_n\to H_\infty)$ iff $f_{H_n}\luto f_{H_\infty}$ in $\HH$. From Lemma \ref{domain convergence*}, this implies that $\HH\sem H_n\dto \HH\sem H_\infty$. But $\HH\sem H_n\dto \HH\sem H_\infty$ does not imply $d_{\cal H}(H_n\to H_\infty)$. A counterexample is $H_n=\{z\in\HH:|z-2n|\le n\}$ and $H_\infty=\emptyset$. Since $H_1\cdot H_2=H_3$ iff $f_{H_1}\circ f_{H_2}=f_{H_3}$,  the dot product is continuous.

Formula (5.1) in \cite{LERW} states that for any $\HH$-hull $H$, there is a positive measure $\mu_H$ supported by $S_H^*$, the convex hull of $S_H$, such that for any $z\in\C\sem S_H^*$,
\BGE f_H(z)=z+\int \frac{-1}{z-x} d\mu_H(x).\label{f-z}\EDE
In particular, if $H$ is bounded by a crosscut, then $\mu_H$ is absolutely continuous w.r.t.\ the Lebesgue measure, and $d\mu_H/dx=\frac 1\pi \Imm f_H(x)$, where the value of $f_H$ on $S_H^*$ is the continuation of $f_H$ from $\HH$. If $H$ is approximated by a sequence of $\HH$-hulls $(H_n)$, then $\mu_H$ is the weak limit of $(\mu_{H_n})$. We may choose each $H_n$ to be bounded by a crosscut, whose height is not bigger than $h+1/n$, where $h$ is the height of $H$. Then each $\mu_{H_n}$ has a density function, whose $L^\infty$ norm is not bigger than $(h+1/n)/\pi$. Thus, $\mu_H$ also has a density function, whose $L^\infty$ norm is not bigger than $h/\pi$. We use $\rho_H$ to denote the density function of $\mu_H$.
Since $f_H:\C\sem S_H^*\conf \C\sem \ha H^*$ and $f_H'(\infty)=1$, we see that $\rad(\ha H^*)=\rad(S_H^*)=|S_H^*|/4$. Thus, $\diam(\ha H^*)\le 4\rad(\ha H^*)=|S_H^*|$. On the other hand, the diameter of $\ha H^*$ is at least twice the height of $H$. So $\|\rho_H\|_{\infty}\le \frac{|S_H^*|}{2\pi}$.

By approximating any $\HH$-hull $H$ using a sequence of $\HH$-hulls $(H_n)$, each of which is the union of finitely many mutually disjoint $\HH$-hulls bounded by crosscuts in $\HH$, we see that $\mu_H$ is in fact supported by $S_H$. By continuation, (\ref{f-z}) holds for any $z\in\C\sem S_H$. Furthermore, the support of $\mu_H$ is exactly $S_H$ because from (\ref{f-z}) $f_H$ extends analytically to the complement of the support of $\mu_H$, while from Lemma \ref{extend}$f_H$ can not be extended analytically beyond $\C\sem S_H$. So we obtain the following lemma.

\begin{Lemma}
  For any $\HH$-hull $H$, $\mu_H$ has a density function $\rho_H$, whose support is $S_H$, and whose $L^\infty$ norm is no more than $\frac{|S_H^*|}{2\pi}$. Moreover, (\ref{f-z}) holds for any $z\in\C\sem S_H$.
\end{Lemma}

The following lemma extends Lemma 5.4 in \cite{LERW}, and we now give a proof.

\begin{Lemma}
  For any compact $F\subset \R$, ${\cal H}_F:=\{H:S_H\subset F\}$ is compact, and $H_n\to H$ in ${\cal H}_F$ implies that $f_{H_n}\luto f_H$ in $\C\sem F$. \label{compact-new}
\end{Lemma}
\begin{proof}
  Suppose $(H_n)$ is a sequence in ${\cal H}_F$. Let $|F^*|$ denote the length of the convex hull of $F$. Then for each $n$, $\rho_{H_n}$ is supported by $S_{H_n}\subset F$, and the $L^\infty$ norm of $\rho_{H_n}$ is no more than $\frac{|S_{H_n}^*|}{2\pi}\le \frac{|F^*|}{2\pi} $. Thus, $(\rho_{H_n})$ contains a subsequence $(\rho_{H_{n_k}})$, which converges in the weak-$*$ topology to a function $\rho$ supported by $F$. From (\ref{f-z}) we see that $f_{H_{n_k}}$ converges uniformly on each compact subset of $\C\sem F$, and if $f$ is the limit function, then $f(z)-z=\int_F \frac{-1}{z-x} \rho(x)dx$, $z\in\C\sem F$.
  So $f(z)-z\to 0$ as $z\to \infty$. This means that $f$ can not be constant. From Lemma \ref{domain convergence*}, $f$ is a conformal map on $\C\sem F$. Since $f(z)-z\to 0$ as $z\to\infty$, $\infty$ is a simple pole of $f$. Thus, $f(\C\sem F)$ contains a neighborhood of $\infty$. Let $G=\C\sem f(\C\sem F)$. Then $G$ is compact. Since every $f_{H_{n_k}}$ is $\R$-symmetric, so is $f$. Let $H=G\cap \HH$. Then $f$ maps $\HH$ conformally onto $\HH\sem H$. This implies that $H$ is an $\HH$-hull and $f=f_H$ on $\HH$ because $f(z)-z\to 0$ as $z\to \infty$. Since $f$ extends $f_H|_\HH$, from Lemma \ref{extend}, we see that $S_H\subset F$ and $f=f_H$ in $\C\sem F$. Since $f_{H_{n_k}}\luto f_H$ in $\HH$, we get $H_{n_k}\to H\in{\cal H}_F$. This shows that ${\cal H}_F$ is compact.
  The above argument also gives $f_{H_{n_k}}\luto f_H$ in $\C\sem F$. If $H_n\to H$, then any subsequence $(H_{n_k})$ of $(H_n)$ contains a subsequence $(H_{n_{k_l}})$ such that $f_{H_{n_{k_l}}}\luto f_H$ in $\C\sem F$, which implies that $f_{H_n}\luto f_H$ in $\C\sem F$.
\end{proof}

\section{Topology on $\D$-hulls} \label{D}
Define a metric $d_{\cal H}$ on the space of $\D$-hulls such that
\BGE d_{\cal H}(H_1,H_2)=\sum_{n=1}^\infty \frac{1}{2^n}\sup_{|z|\le 1-1/n}\{|f_{H_1}(z)-f_{H_2}(z)|\}.\label{metric-D}\EDE
It is clear that $d_{\cal H}(H_n,H)\to 0$ iff $f_{H_n}\luto f_H$ in $\D$.
From Lemma \ref{domain convergence*}, this implies that $\D\sem H_n\dto \D\sem H$. 
On the other hand, from Lemma \ref{compact-D-M} below, one see that $\D\sem H_n\dto\D\sem H$ also implies that $H_n\to H$.
Since $f_{H_n}\luto f_H$ in $\D$ implies that $f_{H_n}'(0)\to f_H'(0)$, we see that $\dcap$ is a continuous function. Moreover, the dot product is also continuous.

\begin{Lemma}
  For any $M<\infty$, $\{H:\dcap(H)\le M\}$ is compact. \label{compact-D-M}
\end{Lemma}
\begin{proof} Suppose $(H_n)$ is a sequence of $\D$-hulls with $\dcap(H_n)\le M$ for each $n$. Then $f_{H_n}'(0)=e^{-\dcap(H_n)}\ge e^{-M}$.
Since $(f_{H_n})$ is uniformly bounded in $\D$, it contains a subsequence $(f_{H_{n_k}})$, which converges locally uniformly in $\D$. Let $f$ be the limit. Then $f'(0)=\lim_{k\to\infty} f_{H_{n_k}}'(0)\ge e^{-M}$. Thus, $f$ is not constant. From Lemma \ref{domain convergence*}, $f$ is conformal in $\D$. Since $f(0)=\lim_{k\to\infty} f_{H_{n_k}}(0)$ and $f'(0)>0$, we see that $f=f_H|_\D$ for some $\D$-hull $H$. Since $f'(0)\ge e^{-M}$, $\dcap(H)\le M$. From $f_{H_{n_k}}\luto f_H$ in $\D$ we get $H_{n_k}\to H$. \end{proof}

\no{\bf Remark.}
We may compactify the space of $\D$-hulls by adding one element $H_\infty$ with the associated function $f_{H_\infty}\equiv 0$ in $\D$, and defining the metric $d_{\cal H}$ in the extended space using (\ref{metric-D}).

\begin{Lemma}
  For any compact $F\subsetneqq\TT$, ${\cal H}_F:=\{H:S_H\subset F\}$ is compact. \label{compact-D-F}
\end{Lemma}
\begin{proof}
Let $H\in{\cal H}_F$. From conformal invariance, the harmonic measure of $\TT\sem \ha H$ in $\D\sem H$ seen from $0$ equals to the harmonic measure of $\TT\sem S_H$ in $\D$ seen from $0$, which is bounded below by $|\TT\sem F|/|\TT|>0$. This implies that the distance between $0$ and $H$ is bounded below by a positive constant $r$ depending on $F$, which then implies that $\dcap(H)$ is bounded above by $-\ln(r)<\infty$. From Lemma \ref{compact-D-M}, we see that ${\cal H}_F$ is relatively compact.

It remains to show that ${\cal H}_F$ is bounded. Let $(H_n)$ be a sequence in ${\cal H}_F$, which converges to $H$. We need to show that $H\in{\cal H}_F$. Since $H_{n}\in{\cal H}_F$, each $f_{H_{n}}$ is analytic in $\ha\C\sem F$. We have $f_{H_{n}}\luto f_H$ in $\D$. From $\TT$-symmetry,  $f_{H_{n}}\luto f_H$ in $\D^*$. Let $J=\{|z|=2\}\subset\D^*$. Then $f_{H_{n_k}}\to f_H$ uniformly on $J$. Sine $f_{H_{n}}$ maps $\{|z|< 2\}\sem F$ into the Jordan domain bounded by $f_{H_{n}}(J)$, we see that the family $(f_{H_{n}})$ is uniformly bounded in $\{|z|< 2\}\sem F$. So it contains a subsequence $(f_{H_{n_{k}}})$, which converges locally uniformly in $\{|z|< 2\}\sem F$. The limit function is analytic in $\{|z|< 2\}\sem F$ and agrees with $f_H$ on $\D$, which implies that $f_H$ extends analytically across $\TT\sem F$. So $S_H\subset F$, i.e., $H\in{\cal H}_F$.
\end{proof}

There is an integral formula for $\D$-hulls which is similar to (\ref{f-z}).
For any $\D$-hull $H$, there is a positive measure $\mu_H$ with support $S_H$ such that
\BGE f(z)=z\cdot\exp\Big(\int_\TT -\frac{x+z}{x-z}d\mu_H(x)\Big),\quad z\in \C\sem S_H,\label{f/z}\EDE
and $H_n\to H$ iff $\mu_{H_n}\to \mu_H$ weakly. Moreover, $\mu_H$ is absolutely continuous w.r.t.\ the Lebesgue measure on $\TT$, and the density function is bounded. From this integral formula, it is easy to get the following lemma.

\begin{Lemma}
   For any compact $F\subset\TT$, $H_n\to H$ in ${\cal H}_F$ implies that $f_{H_n}\luto f_H$ in $\C\sem F$. \label{compact-new-r}
\end{Lemma}

\end{document}